%% file: main.tex
\documentclass[11pt]{amsart} 
\usepackage{amsmath,amssymb,a4wide} 
\usepackage{color}
\usepackage{mathrsfs}
\usepackage{dsfont}
\usepackage{mathabx}
\usepackage{tikz,pgfplots}
\usepackage{hyperref}
\usepackage{ulem}
\usepackage{float}
\usepackage{caption}
\usepackage{subcaption}
\usepackage{comment}

\newtheorem{theorem}{Theorem} 
 
\newtheorem{proposition}[theorem]{Proposition} 
\newtheorem{remark}[theorem]{Remark}

\newcommand{\dis}{\displaystyle}

\newcommand{\R}{{\mathbb R}} 
\newcommand{\C}{{\mathbb C}}

\newcommand{\dd}{{\rm d}}

\usepackage{color}

\title[Numerical study ]
{A numerical study of vortex nucleation in 2d rotating Bose--Einstein condensates}

\author[G. Dujardin]{Guillaume Dujardin}
\address[G. Dujardin]{Univ. Lille, Inria, CNRS, UMR 8524 - Laboratoire Paul Painlev\'e, F-59000 Lille}
\author[I. Lacroix-Violet]{Ingrid Lacroix-Violet}
\address[I. Lacroix-Violet]{Universit\'e de Lorraine, CNRS, IECL, F-54000 Nancy, France}
\author[A. Nahas]{Anthony Nahas}
\address[A. Nahas]{Univ. Lille, CNRS, UMR 8524, Inria - Laboratoire Paul Painlev\'e, F-59000 Lille}

\definecolor{green}{rgb}{0.1,0.7,0.4}

\begin{document}


\begin{abstract}
  This article implements a numerical method for the minimization under constraints
of a discrete energy modeling multicomponents rotating Bose--Einstein condensates
in the regime of strong confinement and with rotation.
Moreover, this method allows to consider both segregation and coexistence regimes between
the components.
The method includes a discretization of a continuous energy in space dimension 2 and
a gradient algorithm with adaptive time step and projection for the minimization.
It is well known that, depending on the regime, the minimizers may display different
structures, sometimes with vorticity
(from singly quantized vortices, to vortex sheets and giant holes).
The goal of this paper is to study numerically the structures of the minimizers.
In order to do so, we introduce a numerical algorithm for the computation of the indices
of the vortices, as well as an algorithm for the computation of the indices of vortex sheets.
Several computations are carried out, to illustrate the efficiency of the method,
to cover different physical cases, to validate recent theoretical results
as well as to support conjectures.
Moreover, we compare this method with an alternative method from the literature.
\end{abstract}


\maketitle
{\small\noindent} 



\section{Introduction}
Bose--Einstein condensation was predicted by Satyendra Nath Bose \cite{Bose24}
and Albert Einstein \cite{Einstein25} in 1924 and 1925.
It describes a state of matter in which separate atoms or subatomic particles,
cooled to near absolute zero (a few $\mu K$), coalesce into a single quantum mechanical entity—that is,
one that can be described by a single wave function—on a near-macroscopic scale.
Bose--Einstein condensates were first realized experimentally in 1995 \cite{Ketterle95,Anderson95}.
When the so-called Bose--Einstein condensates (BEC) are set to rotation,
topological defects often manifest themselves as vortices that correspond
to zeros of the wave function with phase circulation.
This phenomenon was first observed in two components BEC \cite{PhysRevLett.83.2498}.
When a BEC is set to a high rotation in a strong confinement regime,
the vortices align and form unique structures.
In the case of a single component condensate, we observe singly quantized vortices,
forming triangular lattices once they are numerous \cite{triangular_lattice,PhysRevA.84.033611}.
In the case of two components condensate, depending on the interaction between the two components,
many structures may appear.
For example, we can observe coreless vortices, which refer to having singly quantized vortices
in one component while having a corresponding peak in the second component \cite{Mizushima_2004},
or vortex sheets \cite{PhysRevA.84.033611,KASAMATSU_2005}. 

Different models of BEC with one or several components have already been studied
in the mathematical literature.
For example, the minimization of the Gross--Pitaevskii functional in $\R^2$ is studied
theoretically in \cite{IGNAT2006260} (see also references therein).
In \cite{AftalionSandier2020,aftalion:hal-00960881}, the authors study different structures
of a BEC in a strong confinement and coupling regime in a bounded domain of $\R^2$.
Several methods have been developed for the numerical computation of approximations
of minimizers of Gross--Pitaevskii energies.
For example, in \cite{BaoTang2003,Bao2004,BaoDu2004,BAO2006612,BAO2006836,BaoCaiWang2010,BaoCai2011,BaoCai2013,MR3145291,BaoJiangTangZhang2015},
the authors develop numerical methods in several physical contexts
and with several space discretizations,
which require solving a linear system at each time-step.
A theoretical analysis of the convergence of some of these methods
has been carried out in several simple contexts (see for example \cite{FaouJezequel17}
for the classical Gross--Pitaevskii energy in one space dimension, for the one component case,
without rotation nor confinement, in a neighborhood of the ground state).
Another option is to use Sobolev gradients, as opposed to $L^2$ gradients,
as developed in \cite{danaila:hal-00432061}.
	
In this article, we study numerically the behaviour of a rotating Bose--Einstein condensate (BEC)
in two dimensions in a strong confinement regime based on the numerical minimization
of the Gross--Pitaevskii (GP) energy \eqref{eq:model}.
We consider the both cases of one component and two components condensates.
For two components condensates,
we consider both segregation and coexistence regimes between the two components. 
We introduce a new discretization for the GP energy using the Fast Fourier transformation (FFT) scheme.
We use an {\it explicit} $L^2$ gradient method with adaptive step and projection
for the minimization of our discrete energy under constraints.
We derive a stopping criterion based on the evaluation of the gradient of the energy
on the constrained manifold using the residue of the Euler-Lagrange equation corresponding to the
constraints, which we use at the discrete level.
Moreover we provide a post processing algorithm for computing the indices
of the vortices and of the vortex sheets of the minimizers.
Finally, we compare the efficiency of our explicit gradient method with projection,
named EPG, to the method used in GPELab (\cite{GPE1,GPE2}).
The numerical results of the test cases in this paper will illustrate how an explicit projected
gradient method together with an energy discretization allowing for the use of FFT
in the computation of its gradient makes it possible to outperform (linearly)
implicit methods such as that of GPELab. 

The outline of this paper is as follows.
We introduce the continuous model for rotating Bose--Einstein condensates with two components
in Section \ref{sec:model} and we recall the different regimes for one component
and two components condensates. 
In order to numerically study these regimes,
we discretize in Section \ref{sec:discretmin} the continuous Gross--Pitaevskii energy.
Our discretization uses the discrete Plancherel formula and allows for using FFT
in the computation of the gradient.
We implement in this section a gradient method for the minimization of the energy
with a projection step to take the constraints into account.
One of the interests of this method is that it allows for the derivation of a stopping criterion
which we develop in the same section.
We develop in Section \ref{sec:post_proc} two post processing algorithms
for the computation of indices.
The first deals with vortices and the second deals with vortex sheets.
We present numerical results in Section \ref{sec:numericalresults} for the regimes
described in Section \ref{sec:model}.
In particular, we validate numerically recent theoretical results and we support some conjectures
as for example the existence of vortex sheets in a segregation regime.
Section \ref{sec:comparaisonGPE} is devoted to the comparison of the efficiency
of the EPG method with that of GPELab \cite{GPE1,GPE2}.	

\section{The model and the regimes}
\label{sec:model}

\subsection{The model}
We consider the following model for the energy of a two components rotating Bose--Einstein condensate
in dimension $2$ in the limit of strong confinement and strong rotation
studied in \cite{AftalionSandier2020}. This model reads
\begin{equation*}\label{eq:model}
  {\mathcal E}_{\varepsilon,\delta}^\Omega (u_1,u_2) = \sum_{\ell=1}^2 \frac12 \int_{D} \left\|\nabla u_\ell-i\Omega u_\ell
  {\bold x}^\perp
  \right\|^2 \dd x\dd y + W_{\varepsilon,\delta}(u_1,u_2),
\end{equation*}
where $D\subset\R^2$ is the bounded physical domain of interest,
$\Omega\in\R$ is the rotation speed, $\varepsilon$ and $\delta$ are positive constants,
${\bold x}^\perp=(-y,x)$, $u_1,u_2\in H^1(D,\C)$
are the wave functions related to each component of the condensate.
Moreover, $W_{\varepsilon,\delta}$ is the confining part of the energy defined as
\begin{eqnarray*}
  W_{\varepsilon,\delta}(u_1,u_2) &=&
  \frac{1}{4\varepsilon^2} \int_D (\rho(r)-|u_1|^2)^2 \dd x \dd y
  +\frac{1}{4\varepsilon^2} \int_D (\rho(r)-|u_2|^2)^2 \dd x \dd y \\
  &{}& + \frac{\delta}{2\varepsilon^2} \int_D |u_1|^2 |u_2|^2 \dd x \dd y
  - \frac{1}{4\varepsilon^2} \int_D \rho^2(r)\dd x \dd y,
\end{eqnarray*}
where $\rho$ is a function of $r=\sqrt{x^2+y^2}$ to be defined later,
$\delta>0$ measures the stength of the interaction
between the two components, $1/\varepsilon^2$ measures the strength of the internal
interaction in each component of the condensate.
In this paper, we take $D$ as the disk of radius $R>0$ centered at the origin of $\R^2$.
Moreover, we consider $\rho\equiv 1$ in $D$, which corresponds to a flat trap,
as in \cite{Correggi2007anharmonic,Rougerie2011} (see also \cite{AftalionSandier2020}).

\begin{remark}
  \label{rem:centrifuge}
  A straightfoward computation yields for all $\ell\in\{1,2\}$,
\begin{eqnarray}
 \lefteqn{ \frac12 \int_{D} \left\|\nabla u_\ell-i\Omega u_\ell
  {\bold x}^\perp
\right\|^2 \dd x\dd y  } \nonumber\\
&=& \frac12 \int_D \|\nabla u_\ell\|^2\dd x\dd y
- \Omega \int_D \Re\left(iu_\ell{\bold x}^\perp\cdot \overline{\nabla u_\ell}\right)\dd x \dd y
+ \frac{\Omega^2}{2} \int_D r^2|u_\ell|^2 \dd x \dd y.\label{eq:identiteremarquable}
\end{eqnarray}
\end{remark}
Following Remark 1 of \cite{AftalionSandier2020} and the one component analysis carried out
in \cite{Rougerie2011}, in the regime $|\Omega|<<1/\varepsilon$, the contribution of the third term in
\eqref{eq:identiteremarquable} plays no role in the asymptotics $\varepsilon\to 0$.
Therefore, from now on and unless stated otherwise, we shall consider in the regime
$|\Omega|<<1/\varepsilon$ the minimization of the energy
\begin{eqnarray}
  \label{eq:defNRJ}
  E_{\varepsilon,\delta}^\Omega (u_1,u_2) & = & \sum_{\ell=1}^2 \left[
  \frac12 \int_{D} \left\|\nabla u_\ell\right\|^2 \dd x \dd y
  -\Omega \int_D \Re\left(i u_\ell {\bold x}^\perp \cdot \overline{\nabla u_\ell}\right)
  \dd x\dd y
  \right]
  + W_{\varepsilon,\delta}(u_1,u_2)\\ \nonumber
  & = & {\mathcal E}_{\varepsilon,\delta}^\Omega(u_1,u_2)
  - \frac{\Omega^2}{2} \int_D r^2(|u_1|^2+|u_2|^2) \dd x \dd y.
\end{eqnarray}

\noindent
Our aim is to compute, in the regime $\varepsilon\to 0$,
minimizers of the energy $E^\Omega_{\varepsilon,\delta}$ defined in \eqref{eq:defNRJ}
(and {\it not} directly $\mathcal E^\Omega_{\varepsilon,\delta}$) in $H^1_0(D,\C)$ with the constraints
\begin{equation}
  \label{eq:contraintes}
  \int_D |u_1|^2 \dd x \dd y = M N_1,
  \qquad \text{and} \qquad
  \int_D |u_2|^2 \dd x \dd y = M N_2,
\end{equation}
where $M=\int_{D}\rho(r) \dd x \dd y$, and $N_1,N_2\geq 0$ with $N_1+N_2=1$.
Note that, in the sequel, $\delta>0$ and $\Omega\in\R$ may
or may not depend on $\varepsilon$ when considering a sequence of minimizing problems with smaller
and smaller $\varepsilon$.

Note that one has also
\begin{equation}
  \label{eq:defW}
  W_{\varepsilon,\delta}(u_1,u_2) =
  \frac{1}{4\varepsilon^2} \int_D (\rho(r)-|u_1|^2-|u_2|^2)^2 \dd x \dd y
  + \frac{\delta-1}{2\varepsilon^2} \int_D |u_1|^2 |u_2|^2 \dd x \dd y.
\end{equation}
In view of this expression of $W_{\varepsilon,\delta}$,
we refer to the case $\delta>1$ as the segregation regime since minimizing the energy
tends to split the two components in this case,
and to the case $\delta<1$ as the coexistence regime since
minimizing the energy tends to mix the two components of the condensate in this other case.

\subsection{The different regimes}

We recall in this section the theoretical results obtained in
\cite{Rougerie2011, AftalionSandier2020, GoldmanMerlet2017, PhysRevLett.91.150406, PhysRevA.84.033611}.
These results describe the asymptotics of the minimizers of \eqref{eq:defNRJ} in different
regimes.
These regimes depend on the behavior of the rotation speed $\Omega$ in the strong confinement
limit ($\varepsilon\to 0$) in one component condensates (see \ref{subsubsec:onecomponent}).
They also depend on the value of the segregation/coexistence parameter ($\delta\geq1$ or $\delta<1$)
in two components condensates (see \ref{subsubsec:twocomponent} and \ref{subsubsec:twocomponentcoex}).

\subsubsection{One component condensates}
\label{subsubsec:onecomponent}

According to \cite{Rougerie2011}, for one component condensates ($N_1=1,N_2=0$),
the behaviour of the structure of the vortices of the minimizers of \eqref{eq:defNRJ}
with constraints \eqref{eq:contraintes} in the regime $\varepsilon\to 0$ changes
with the dependency of the rotation speed $\Omega=\Omega_\varepsilon$ with respect to $\varepsilon$.
Namely, there exists three critical rotation speeds
$\Omega_{\varepsilon}^1<<\Omega_{\varepsilon}^2<<\Omega_{\varepsilon}^3$ such that
\begin{itemize}
\item if $\Omega_\varepsilon<\Omega_{\varepsilon,1}$, then minimizers $u_1^\varepsilon$
  have no vortices,
\item if $\Omega_{\varepsilon}^1<\Omega_\varepsilon<\Omega_{\varepsilon}^2$, then the vortices
  of minimizers $u_1^\varepsilon$ appear on a hexagonal lattice and are singly quantized,
\item if $\Omega_{\varepsilon}^2<\Omega_\varepsilon<\Omega_{\varepsilon}^3$, the centrifugal
  force comes into play: a hole appears in minimizers $u_1^\varepsilon$ and the condensate
  looks like an annulus, with vortices located on a lattice on the annulus,
\item if $\Omega_{\varepsilon}^3<\Omega_\varepsilon$, the centrifugal force is so important
  that the vortices retreat in the central hole of the condensate, creating a central giant vortex
  with high index.
\end{itemize}
Moreover, one has
\begin{equation}
\label{eq:char_speed}
  \Omega_{\varepsilon}^1 \sim \log(1/\varepsilon),
  \qquad
  \Omega_{\varepsilon}^2 \sim \frac{1}{\varepsilon},
  \qquad \text{and} \qquad
  \Omega_{\varepsilon}^3 \sim \frac{1}{\varepsilon^2 \log(1/\varepsilon)}.
\end{equation}
Note that $\delta$ plays no role in this case since $u_2\equiv 0$ because $N_2=0$.

According to Remark \ref{rem:centrifuge}, the regimes $\Omega_\varepsilon<\Omega_\varepsilon^2$
are similar for minimizers of $\mathcal E^{\Omega_\varepsilon}_{\varepsilon,\delta}$
and $E^{\Omega_\varepsilon}_{\varepsilon,\delta}$. In the regimes $\Omega^2_\varepsilon<\Omega_\varepsilon$,
the annulus behaviour is due to the centrifugal force in $E^{\Omega_\varepsilon}_{\varepsilon,\delta}$.

One of the goals of this paper is to confirm numerically the four regimes above
(described in \cite{Rougerie2011}). To do so, we introduce a numerical algorithm in Section \ref{sec:discretmin} and we use it to provide numerical simulations in Section \ref{subsec:simulation_onecomponent}.

\subsubsection{Two components condensates in the segregation regime ($\delta>1$)}
\label{subsubsec:twocomponent}
In the recent paper \cite{AftalionSandier2020}, the authors consider two components condensates
in the segregation case $\delta>1$.
For $N_1\in(0,1)$ (recall that $N_2=1-N_1$), they introduce the minimizing perimeter
\begin{equation*}
  \ell_{N_1} = \min_{\stackrel{\omega\subset D}{|\omega|=N_1}} \text{per}(\omega),
\end{equation*}
and they perform an analysis of the minimizers of $E_{\varepsilon,\delta_\varepsilon}^{\Omega_\varepsilon}$ in the regime $\varepsilon \to 0$ depending on whether $\Omega_\varepsilon=0$ or not.  

First, they address the minimization of \eqref{eq:defNRJ} with constraints \eqref{eq:contraintes}
when $\Omega=0$. In this case, the squared moduli of the minimizers $u_1^\varepsilon$
and $u_2^\varepsilon$ tend to 1 in two separate regions of $D$ and
the authors of \cite{AftalionSandier2020} prove that their sum
$v_\varepsilon^2=|u_1^\varepsilon|^2+|u_2^\varepsilon|^2$
and the normalized energy $\varepsilon E^0_{\varepsilon,\delta_{\varepsilon}}$
have the following behaviour
\begin{itemize}
\item in the regime $\delta_{\varepsilon}\varepsilon^2\to +\infty$ (strong segregation regime \cite{AftalRoyo2015}),
  then $\inf_D v_\varepsilon^2$ tends to $0$ and
  $\varepsilon E^0_{\varepsilon,\delta_\varepsilon}$ tends to some constant times $\ell_{N_1}$;
\item in the regime $\varepsilon\to 0$ with $\delta>1$ fixed, then $\inf_D v_\varepsilon^2$
  tends to some number between $0$ and $1$ and
  $\varepsilon E^0_{\varepsilon,\delta_\varepsilon}$ tends to some constant (which depends on $\delta$)
  times $\ell_{N_1}$ \cite{GoldmanRoyo2015};
\item in the regime $\delta_\varepsilon\to 1$ with $\tilde\varepsilon = \varepsilon/\sqrt{\delta_\varepsilon-1} \to 0$, the rescaled energy $\tilde\varepsilon E^0_{\varepsilon,\delta_\varepsilon}$ tends
  to $\ell_{N_1}/2$ and it is expected that $\inf_D v^\varepsilon$ tends to $1$ \cite{GoldmanMerlet2017}.
\end{itemize}

Second, they address the minimization of \eqref{eq:defNRJ} with constraints \eqref{eq:contraintes} when
$\Omega=\Omega^\varepsilon\to +\infty$ as $\varepsilon$ tends to $0$. With $\tilde \varepsilon$ defined
as above, they consider a regime where $\delta=\delta_\varepsilon \to 1$ and $\tilde \varepsilon \to 0$
as $\varepsilon$ tends to $0$ and they prove that
\begin{itemize}
\item there exists two constants $C_1,C_2>0$ such that, for all $i$,
  if $\Omega^\varepsilon>C_i\log(1/\tilde\varepsilon)$, then the infimum of the
  limiting density $|u_i^\varepsilon|^2$ vanishes as $\varepsilon$ tends to $0$;
\item in the regime of moderate rotational speed
  $\log(1/\tilde\varepsilon)<<\Omega^\varepsilon<<1/\tilde\varepsilon$, the limiting density
  $|u_i^\varepsilon|^2$ is uniform in each region, and hence does not depend on the shape of the region;
\item for higher rotational speeds $1/(\tilde\varepsilon \log(1/\tilde\varepsilon))<<\Omega^\varepsilon<<1/\tilde\varepsilon^2$, the leading order
  in the energy $E^{\Omega^\varepsilon}_{\varepsilon,\delta_\varepsilon}$ is the vortex energy,
  and the authors of \cite{AftalionSandier2020} conjecture the possibility of observing vortex sheets.

\end{itemize}

{\subsubsection{Two components condensates in the coexistence regime ($\delta < 1$)}
  \label{subsubsec:twocomponentcoex}

  In the regime $\delta <1$, the two components still have mass in the region where $\rho>0$,
  but the supports of the two components tend to overlap
  (see \cite{PhysRevLett.91.150406,PhysRevA.84.033611}).
  In this regime ($\delta <1$), depending on the rotational speed $\Omega$,
  we should observe four different qualitative behaviours
  (see \cite{PhysRevLett.91.150406,PhysRevA.84.033611})
  for the minimizers when $\varepsilon$ tends to $0$:
	\begin{itemize}
        \item The first behaviour ($\delta\in[0,1)$), with a very low velocity $\Omega$,
          is when the two components coexist and there are no vortices: each component
          has mass in the disc where $\rho>0$ and the profile of the components depends
          on $N_1$ and $N_2$.
	\item The second behaviour, when $\delta=0$ (which means there is no interactions
          between the two components), we should observe,
          depending on the (high) rotation speed $\Omega$,
          the existence of a triangular vortex lattice.
        \item As $\delta\in(0,1)$ increases, the positions of vortex cores in one component
          gradually shift from those of the other component and the triangular lattices are distorted.
          Eventually, in this case, after a certain value of $\delta$,
          the vortices in each component form a square lattice.
	\item The last behaviour corresponds to $\delta\rightarrow 1$
          (when $\varepsilon\rightarrow 0$).
          In that case, we should observe either stripe or double-core vortex lattices
          in the minimizers.
	\end{itemize} 
}
\section{The discretisation of the problem
  and the minimization method}\label{sec:discretmin}

Minimizing the energy $E_{\varepsilon,\delta}^\Omega$ defined in \eqref{eq:defNRJ}
under the constraints \eqref{eq:contraintes} is a continuous minimization problem over
an infinite dimensional manifold that we replace in this section by a discrete minimization
problem over a finite dimensional manifold.
We consider a gradient descent method with projection over the constrained manifold,
which is a first-order iterative optimization algorithm for finding a local minimum
of a differentiable function. We name this method EPG.
By calculating the gradient of the energy $E_{\varepsilon,\delta}^\Omega$,
we are able to approach more and more towards the solution of the discrete problem after each
iteration of the method.
In Section \ref{subsec:discrNRJ}, we describe the discretization of the energy
$E_{\varepsilon,\delta}^\Omega$ defined in \eqref{eq:defNRJ}.
Then, we compute its gradient in Section \ref{subsec:gradient}.
We establish a criterion for the minimization of the energy on the finite dimensional manifold
in Section \ref{subsec:criterion}.
We propose a full description of the minimization method in Section \ref{subsec:methodegradient}
as a conclusion.

\subsection{Discretization of the energy}
\label{subsec:discrNRJ}

For the discretization of the energy defined in \eqref{eq:defNRJ}, we decide to include
the disk $D$ of radius $R$ centered at the origin of $\R^2$ in a larger square of side length $2L$
centered at the origin.
In addition, we extend the function $\rho$ used to define the energy in \eqref{eq:defNRJ} to the
square in such a way that it takes (very) negative values outside $D$.
By imposing homogeneous Dirichlet conditions on the boundary of the square, such
an extension of $\rho$ should, in the regime $\varepsilon \to 0$, lead to minimizers with very
small squared modulus outside the disk, the restriction to the disk of which
can be considered as approximations to the continuous
minimizers of $E^\Omega_{\varepsilon,\delta}$ defined in \eqref{eq:defNRJ}.
This is due in particular to the first term in the definition of $W_{\varepsilon,\delta}$
in \eqref{eq:defW}.

Therefore, we choose $L>R$ so that the disc $D$ of radius $R$ centered at the origin of $\R^2$
is included in the square $[-L,L]^2$.
Then, we discretize $[-L,L]^2$ with $N+2$ equidistant points with respect to the $x$-axis
and $N+2$ equidistant points with respect to the $y$-axis for some integer $N$.
We do this by setting $\delta_x=\delta_y=\frac{2L}{N+1}$ and
$$x_n=-L+n\delta_x, \qquad \text{and} \qquad y_k=-L+k\delta_y,$$
for $n, k \in\{0,...,N+1\}$.

We use the letter $\psi$ to denote the discrete counterpart to the continuous wave functions denoted by $u$ in Section \ref{sec:model}.
As explained above, we think of $\psi^\ell_{n,k}$ as an approximation of $u_\ell(x_n,y_k)$
for $\ell=1,2$. In matrix form, we use the notation:
\begin{align*}
    \psi^\ell &= \begin{pmatrix}
           \psi^\ell_{0,0} & \psi^\ell_{0,1} & ... & \psi^\ell_{0,N+1}\\
           \psi^\ell_{1,0} & \psi^\ell_{1,1} & ... & \psi^\ell_{1,N+1} \\
           \vdots & \vdots & \vdots & \vdots\\
           \psi^\ell_{N+1,0} & \psi^\ell_{N+1,1} & ... & \psi^\ell_{N+1,N+1}
         \end{pmatrix} \in \C^{(N+2)^2},
\end{align*}
with the convention that $\psi^\ell_{n,k}=0$ if either $n=0$ or $k=0$ or $n=N+1$ or $k=N+1$.
In order to separate real and imaginary parts of the unknowns,
we set $\psi^\ell_{n,k}=p^\ell_{n,k}+iq^\ell_{n,k}$ for $\ell=1,2$, and set accordingly
$\psi^\ell=P^\ell+i Q^\ell$.
This allows us to write functions of the complex-valued variables $\psi^\ell$ as functions
of the real-valued variables $P^\ell$ and $Q^\ell$.

We decide to use the discrete Fourier transform and its inverse for the discretization
of the terms in the energy $E_{\varepsilon,\delta}^\Omega$ defined in \eqref{eq:defNRJ}
that involve gradients.
This choice allows the use of Fast Fourier Transform algorithms for the computation of the gradient
of the discrete energy (see section \ref{subsec:gradient}).
To do so, we set the following definitions, which are related to that of Python Numpy.
For $v \in \mathbb{C}^{(N+2)^2}$ and for $n,p,k,q\in \{0,\dots,N+1\}$,
the discrete Fourier transforms of $v$ in the $x$- and $y$-direction and their inverses are given by:
\begin{itemize}
	\item $\big(\textbf{fft}_x(v)\big)_{p,k}=\sum_{m=0}^{N+1} v_{m,k}e^{i\pi m}e^{-2\pi i\frac{mp}{N+1}}$,
	\item $\big(\textbf{ifft}_x(v)\big)_{n,k}=\frac{1}{N+1}e^{-i\pi n}\sum_{p=0}^{N+1} v_{p,k} e^{2\pi i\frac{np}{N+1}}$,
	\item $\big(\textbf{fft}_y(v)\big)_{n,q}=\sum_{l=0}^{N+1} v_{n,l}e^{i\pi l}e^{-2\pi i\frac{lq}{N+1}}$,
	\item $\big(\textbf{ifft}_y(v)\big)_{n,k}=\frac{1}{N+1}e^{-i\pi k}\sum_{q=0}^{N+1} v_{n,q} e^{2\pi i\frac{kq}{N+1}}$.
\end{itemize}

\begin{remark}
  Note that, with the definitions above, the discrete Fourier transforms are {\it not}
  inverses by pairs. Indeed, we have for all $v$,
  $\textbf{ifft}_x(\textbf{fft}_x)(v)=v+\mathcal O(\delta x)$
  and $\textbf{ifft}_y(\textbf{fft}_y)(v)=v + \mathcal O(\delta y)$.
  We will not need any such inversion formula in this paper however.
\end{remark}

For the discretization of the rotational energy, we define $X,Y, \Xi,\Lambda \in \mathbb{R}^{(N+2)^2}$ for all $n,k\in \{0,\hdots,N+1\}$ by
$$X_{n,k}=x_n,\quad Y_{n,k}=y_k, \quad \Xi_{n,k}=\xi_n,\quad \text{and} \quad \Lambda_{n,k}=\lambda_k,$$
with
$$\xi_k=-\frac{\pi(N+1)}{2L}+k\delta_\xi, \quad \text{and} \quad \lambda_n=-\frac{\pi(N+1)}{2L}+n\delta_\lambda,$$ where $\delta_\xi=\delta_\lambda=\frac{\pi}{L}.$ Observe that all the columns of the matrix $\Xi$ are equal to the vector $\xi=(\xi_k)_{0\leq k \leq N+1}$ and all the rows of the matrix $\Lambda$ are equal to the vector $\lambda=(\lambda_n)_{0 \leq n \leq N+1}$.
Moreover, we denote by $\ast$ the term by term multiplication operator between 2 matrices.

We define $E^\Delta_{\varepsilon,\delta}(\psi^1,\psi^2)$ as a discrete counterpart of $E_{\varepsilon,\delta}^\Omega (u_1,u_2)$ by setting for $\psi^1, \psi^2 \in \C^{(N+2)^2}$
\begin{equation}\label{energie_discr}
    E^\Delta_{\varepsilon,\delta}(\psi^1,\psi^2)=\sum_{\ell=1,2}\left(({E_{cin}})^\Delta(\psi^\ell)+({E_{rot}})^\Delta_{\varepsilon}(\psi^\ell)\right)+({E_{W}})^\Delta_{\varepsilon,\delta}(\psi^1,\psi^2),
\end{equation}
where the superscript $\Delta$ indicates that these energies depend on discrete variables.
Let us describe the discretized energy terms which appear above:
\begin{itemize}
    \item $(E_{cin})^\Delta$ corresponds to the kinetic energy and is defined by
    \begin{equation*}
        (E_{cin})^\Delta (\psi^\ell)=\frac{\delta_x^2}{2}\sum_{n,k=0}^{N+1}\bigg(\Big|\textbf{ifft}_x\big(i\Xi\ast\textbf{fft}_x(\psi^\ell)\big)\Big|^2+\Big|\textbf{ifft}_y\big(i\Lambda\ast\textbf{fft}_y(\psi^\ell)\big)\Big|^2\bigg)_{n,k},
    \end{equation*}
    \item $(E_{rot})_\varepsilon^\Delta$ corresponds to the rotational energy and is defined by
    \begin{eqnarray*}
\lefteqn{(E_{rot})_\varepsilon^\Delta (\psi^\ell) }\\
&=&-\Omega_\varepsilon \delta_x^2 \sum_{n,k=0}^{N+1}\Re\left(-i\overline{\psi^\ell}\ast\Big[-Y\ast \textbf{ifft}_x\big(i\Xi\ast\textbf{fft}_x({\psi^\ell})\big)+X\ast \textbf{ifft}_y\big(i\Lambda\ast\textbf{fft}_y({\psi^\ell})\big)\Big]\right)_{n,k},
\end{eqnarray*}
    \item $(E_{W})_{\varepsilon,\delta}^\Delta$ corresponds to the confinement energy and is defined by
\begin{equation*}
    \begin{split}
    (E_{W})_{\varepsilon,\delta}^\Delta(\psi^1,\psi^2)&=\frac{\delta_x^2}{4\varepsilon^2}\sum_{n,k=0}^{N+1}\bigg(\rho(r_{n,k})-|\psi_{n,k}^1|^2-|\psi_{n,k}^2|^2\bigg)^2\\
    &+\frac{\delta_x^2(\delta-1)}{2\varepsilon^2}\sum_{n,k=0}^{N+1} |\psi_{n,k}^1|^2|\psi_{n,k}^2|^2 ,   
    \end{split}
\end{equation*}
{ where $r_{n,k}=\sqrt{x_n^2+y_k^2}$.}
\end{itemize}

\subsection{Computation of the gradient of the discrete energy \eqref{energie_discr}}
\label{subsec:gradient}

Each energy term in \eqref{energie_discr} is a function of the $4N^2$ real variables
$(P^1_{n,k})_{1\leq n,k\leq N}$, $(P^2_{n,k})_{1\leq n,k\leq N}$, $(Q^1_{n,k})_{1\leq n,k\leq N}$
and $(Q^2_{n,k})_{1\leq n,k\leq N}$ using the definitions from Section \ref{subsec:discrNRJ}. 
Let us compute the gradient of each of these energy terms (still using the convention
$p_{n,k}=q_{n,k}=0$ if $n=0$ or $k=0$ or $n=N+1$ or $k=N+1$)
with respect to these variables for the usual scalar product on $\R^{4N^2}$. 
Let $V_{conf} \in \mathbb{R}^{(N+2)^2}$ be defined as:
$$V_{conf}=\begin{pmatrix}
\rho (r_{0,0}) & \rho (r_{0,1}) & \hdots & \rho (r_{0,N+1})\\
\rho (r_{1,0}) & \rho (r_{1,1}) & \hdots & \rho (r_{1,N+1})\\
\vdots\\
\rho (r_{N+1,0}) & \rho (r_{N+1,1}) & \hdots & \rho (r_{N+1,N+1})\\
\end{pmatrix}.$$
For $\ell\in\{1,2\}$ we have for the discrete kinetic energy:
\begin{eqnarray*}
\frac{\partial (E_{cin})^\Delta}{\partial P^\ell}(\psi^\ell)=\delta_x^2\Re\Bigg(\left[\textbf{ifft}_x\big(\Xi^{\ast 2}\ast\textbf{fft}_x(\psi^\ell)\big)\right]_{n,k}+\left[\textbf{ifft}_y\big(\Lambda^{\ast 2}\ast\textbf{fft}_y(\psi^\ell)\big)\right]_{n,k}\Bigg)_{1 \leq n,k \leq N}, \\
\frac{\partial (E_{cin})^\Delta}{\partial Q^\ell}(\psi^\ell)=\delta_x^{2}\Re\Bigg(\left[\textbf{ifft}_x\big(-i\Xi^{\ast 2}\ast\textbf{fft}_x(\psi^\ell)\big)\right]_{n,k}+\left[\textbf{ifft}_y\big(-i\Lambda^{\ast 2}\ast\textbf{fft}_y(\psi^\ell)\big)\right]_{n,k}\Bigg)_{1 \leq n,k \leq N}.
\end{eqnarray*}
For $\ell\in\{1,2\}$ we have for the discrete rotational energy:
\begin{eqnarray*}
\frac{\partial (E_{rot})^\Delta_\varepsilon}{\partial P^\ell}(\psi^\ell)=-2\Omega_\varepsilon\delta_x^2\Re\Bigg( \left[X\ast\textbf{ifft}_y\big(\Lambda\ast\textbf{fft}_y(\psi^\ell)\big)\right]_{n,k}-\left[Y\ast\textbf{ifft}_x\big(\Xi\ast\textbf{fft}_x(\psi^\ell)\big)\right]_{n,k}\Bigg)_{1 \leq n,k \leq N},\\
\frac{\partial (E_{rot})^\Delta_\varepsilon}{\partial Q^\ell}(\psi^\ell)=-2\Omega_\varepsilon\delta_x^2\Re\Bigg( -i\left[X\ast\textbf{ifft}_y\big(\Lambda\ast\textbf{fft}_y(\psi^\ell)\big)\right]_{n,k}+i\left[Y\ast\textbf{ifft}_x\big(\Xi\ast\textbf{fft}_x(\psi^\ell)\big)\right]_{n,k}\Bigg)_{1 \leq n,k \leq N}.
\end{eqnarray*}
For $\ell\in\{1,2\}$ we have for the discrete interaction energy:
\begin{eqnarray*}
\frac{\partial (E_{W})^\Delta_{\varepsilon,\delta}}{\partial P^\ell}(\psi^1,\psi^2)=-\frac{\delta_x^2}{\varepsilon^2}\left(\left[P^\ell\ast\Big(V_{conf}-\left| \psi^1\right|^2-\left| \psi^2\right|^2\Big)\right]_{n,k}+(1-\delta)\left[P^\ell\ast|\psi^{3-\ell}|^2\right]_{n,k}\right)_{1 \leq n,k \leq N},\\
\frac{\partial (E_{W})^\Delta_{\varepsilon,\delta}}{\partial Q^\ell}(\psi^1,\psi^2)=-\frac{\delta_x^2}{\varepsilon^2}\left(\left[Q^\ell\ast\Big(V_{conf}-\left| \psi^1\right|^2-\left| \psi^2\right|^2\Big)\right]_{n,k}+(1-\delta) \left[Q^\ell\ast|\psi^{3-\ell}|^2\right]_{n,k}\right)_{1 \leq n,k \leq N}.
\end{eqnarray*}

\subsection{A criterion for the minimization of $E^\Delta_{\varepsilon,\delta}$ under constraints }
\label{subsec:criterion}
We identify $\C^{2N^2}$ with the subspace of $\mathcal{M}_{N+2}(\C)^2$
consisting in pairs of matrices with zero first and last row and column.
Let us denote the usual scalar product on $\C^{2N^2}$ by
$$\dis \langle u,v \rangle=\Re \left(\sum_{n,k=1}^{N}u_{n,k}\widebar{v_{n,k}}\right),$$
and we set $ \dis  \|v\|_\Delta^2=\delta_x^2\langle v,v \rangle$. 
We replace the continuous constraints \eqref{eq:contraintes} with the discrete analogues 
\begin{equation}
    \label{eq:contraintesdiscretes}
    \|\psi^\ell\|_\Delta^2=\delta_x^2\sum_{k=0}^{N+1}\sum_{n=0}^{N+1} |\psi^\ell_{k,n}|^2=MN_\ell\qquad (\ell=1,2).
\end{equation}

In order to derive a stopping criterion for the minimization of $E_{\varepsilon,\delta}^\Delta$
defined in \eqref{energie_discr} under the constraints \eqref{eq:contraintesdiscretes},
we prove the following proposition.

\begin{proposition}
  \label{prop:CNmin}
  If $(\psi^{1*}, \psi^{2*}) \in \mathbb{C}^{(N+2)^2}$ are minimizers of $E_{\varepsilon,\delta}^\Delta$
  under the constraints \eqref{eq:contraintesdiscretes}
  satisfying the homogeneous Dirichlet boundary conditions,
  then the four $N\times N$ real-valued matrices
  $K^\Delta_{1,P}, K^\Delta_{2,P}, K^\Delta_{1,Q}, K^\Delta_{2,Q}$ defined by\footnote{Recall that
    $\ast$ is the term-by-term product defined in Section \ref{subsec:discrNRJ}.}
  \begin{equation*}
        K^\Delta_{\ell,P} =\frac{\partial E_{\varepsilon,\delta}^\Delta}{\partial P^{\ell}}(\psi^{1*},\psi^{2*})-\frac{\delta_x^2}{\|\psi^{\ell*}\|^2_\Delta}\sum_{n,k=1}^{N}\bigg(\frac{\partial E_{\varepsilon,\delta}^\Delta}{\partial P^{\ell}}(\psi^{1*},\psi^{2*}) \ast P^{\ell*} + \frac{\partial E_{\varepsilon,\delta}^\Delta}{\partial Q^{\ell}}(\psi^{1*},\psi^{2*}) \ast Q^{\ell*}\bigg)_{n,k}P^{\ell*}  ,
\end{equation*}
and
\begin{equation*}
  K^\Delta_{\ell,Q} =\frac{\partial E_{\varepsilon,\delta}^\Delta}{\partial Q^{\ell}}(\psi^{1*},\psi^{2*})-\frac{\delta_x^2}{\|\psi^{\ell*}\|^2_\Delta}\sum_{n,k=1}^{N}\bigg(\frac{\partial E_{\varepsilon,\delta}^\Delta}{\partial P^{\ell}}(\psi^{1*},\psi^{2*}) \ast P^{\ell*} +\frac{\partial E_{\varepsilon,\delta}^\Delta}{\partial Q^{\ell}}(\psi^{1*},\psi^{2*}) \ast Q^{\ell*}\bigg)_{n,k} Q^{\ell*},
\end{equation*}
for $\ell\in\{1,2\}$, vanish.
\end{proposition}

\begin{proof}
  Assume that $(\psi^{1*}, \psi^{2*}) \in \mathbb{C}^{(N+2)^2}$ are minimizers of
  $E_{\varepsilon,\delta}^\Delta$ under the constraints \eqref{eq:contraintesdiscretes}
  satisfying the homogeneous Dirichlet boundary conditions.
  For all $\psi^1, \psi^2 \in \mathbb{C}^{(N+2)^2}$ satisfying
  the homogeneous Dirichlet boundary conditions too and for
  all $t\in \mathbb{R}$ small enough, we define
\begin{equation}
\label{eq:fonction_f}
    f^\Delta (t)=E_{\varepsilon,\delta}^\Delta \left(\frac{\psi^{1*}+t\psi^1}{\|\psi^{1*}+t\psi ^1\|_\Delta}\sqrt{N_1M},\frac{\psi^{2*}+t\psi^2}{\|\psi^{2*}+t\psi^2\|_\Delta}\sqrt{N_2M}\right).
\end{equation}
The fact that $(\psi^{1*}, \psi^{2*})$ are minimizers of $E_{\varepsilon,\delta}^\Delta$
under the constraints \eqref{eq:contraintesdiscretes} implies that
\begin{equation}
\label{eq:laderiveenulle}
(f^\Delta)'(0)=0,
\end{equation}
for all $(\psi^{1}, \psi^{2})$.
In order to rewrite \eqref{eq:laderiveenulle} as a criterion depending on $(\psi^{1*}, \psi^{2*})$,
we set $\psi^{\ell*}=P^{\ell*}+iQ^{\ell*}$ and $\psi^{\ell}=P^{\ell}+iQ^{\ell}$ with $\ell=1, 2$
using the definitions from Section \ref{subsec:discrNRJ}.
Let us denote for all $\ell=1,2$ and small $t\in\R$:
\begin{equation*}
  g^\Delta_{\ell,P}(t)=\frac{P^{\ell*}+t P^\ell}{\|\psi^{\ell*}+t\psi^\ell\|_\Delta}\sqrt{N_\ell M},
  \qquad\text{and}\qquad
  g^\Delta_{\ell,Q}(t)=\frac{Q^{\ell*}+t Q^\ell}{\|\psi^{\ell*}+t\psi^\ell\|_\Delta}\sqrt{N_\ell M},
\end{equation*}
so that, using \eqref{eq:contraintesdiscretes}, we have
\begin{equation*}
  \left\|g^\Delta_{\ell,P}(t)+i g^\Delta_{\ell,Q}(t)\right\|_\Delta^2=N_\ell M.
\end{equation*}
Introducing the notation,
\begin{equation*}
  a_\ell=\frac{\delta_x^2}{\|\psi^{\ell*}\|_\Delta^2}\left(\sum_{n,k=1}^{N}(p_{n,k}^{l*}p_{n,k}^{l}+q_{n,k}^{l*}q_{n,k}^{l})\right)=\frac{\delta_x^2}{\|\psi^{\ell*}\|_\Delta^2}
  \left(
    \langle P^*_\ell , P_\ell\rangle + \langle Q^*_\ell , Q_\ell\rangle
  \right),
\end{equation*}
we infer
\begin{equation*}
  (g^\Delta_{\ell,P})'(0)=P^\ell-a_\ell P^{\ell*},
  \qquad\text{and}\qquad
  (g^\Delta_{\ell,Q})'(0)=Q^\ell-a_\ell Q^{\ell*}.
\end{equation*}
Using the definition \eqref{eq:fonction_f} of $f^\Delta$ and the chain rule,
we obtain
\begin{eqnarray*}
  \lefteqn{(f^\Delta)'(0)}\\
  & = & \displaystyle
    \sum_{\ell=1}^2
    \left(
    \left\langle
    \frac{\partial E^\Delta_{\varepsilon,\delta}}{\partial P^\ell} (\psi^{1*},\psi^{2*}) , (g^\Delta_{\ell,P})'(0)
    \right\rangle
    +
    \left\langle
    \frac{\partial E^\Delta_{\varepsilon,\delta}}{\partial Q^\ell} (\psi^{1*},\psi^{2*}) , (g^\Delta_{\ell,Q})'(0)
    \right\rangle
    \right)
  \\[2mm]
  & = & \displaystyle
    \sum_{\ell=1}^2 \sum_{n,k=1}^{N}\bigg(\frac{\partial E_{\varepsilon,\delta}^\Delta}{\partial P^{\ell}}(\psi^{1*},\psi^{2*})\ast\big(P^\ell-a_\ell P^{\ell*}\big)+\frac{\partial E_{\varepsilon,\delta}^\Delta}{\partial Q^{\ell}}(\psi^{1*},\psi^{2*})\ast\big(Q^\ell-a_\ell Q^{\ell*}\big) \bigg)_{n,k},
 \end{eqnarray*}
 As expected, $(f^\Delta)'(0)$ is an $\R$-linear form depending on $(P^1,Q^1,P^2,Q^2)$.
 The expression above implies that
\begin{equation}
\label{eq:fprime02}
(f^\Delta)'(0)= \sum_{n,k=1}^{N}\left(K^\Delta_{1,P} \ast P^1 +  K^\Delta_{1,Q} \ast Q^1+ K^\Delta_{2,P} \ast P^2 +  K^\Delta_{2,Q} \ast Q^2\right)_{n,k},
\end{equation} 
with $K_{\ell,P}^\Delta$ and $K_{\ell,Q}^\Delta$ the matrices
defined in the statement of Proposition \ref{prop:CNmin}.
This concludes the proof.

\end{proof}

Therefore, in the method described in Section \ref{subsec:methodegradient}, we will use the fact that
\begin{equation}
\label{eq:defKdelta}
K^\Delta=\sum_{\ell=1}^2\|K_{\ell,P}^\Delta\|_\Delta+\sum_{\ell=1}^2\|K_{\ell,Q}^\Delta\|_\Delta,
\end{equation}
is below some threshold as a criterion for numerical convergence.

\subsection{Gradient Method}
\label{subsec:methodegradient}

In order to minimize the discretized energy defined in \eqref{energie_discr}
under the constraints \eqref{eq:contraintesdiscretes},
we propose the following gradient method with projection and adaptive step.
\begin{itemize}
    \item Initialization:
    \begin{itemize}
        \item Choose $\psi^{1,1}, \psi^{2,1} \in \mathbb{C}^{(N+2)^2},$ normalized as: 
        $\|\psi^{\ell,1}\|_\Delta=\sqrt{N_\ell M},~ (\ell=1, 2)$ and satisfying the homogeneous Dirichlet boundary conditions,
        \item choose a step $h>0$,
        \item choose two tolerance parameters $h_0, K_0 >0$ for the convergence test.
    \end{itemize}
    \item Iteration:
    \begin{enumerate}
         \item\label{item:step1} Compute the gradient: $\nabla E_{\varepsilon,\delta}^\Delta\left(\psi^{1,m},\psi^{2,m}\right)$.
         \item\label{item:step2} Compute the auxiliary step	$\begin{pmatrix}
                            \tilde{\psi}^{1,m} \\
                            \tilde{\psi}^{2,m}
                            \end{pmatrix}=\begin{pmatrix}
                            \psi^{1,m} \\
                            \psi^{2,m}
                            \end{pmatrix}-h\nabla E^\Delta_{\varepsilon,\delta}(\psi^{1,m},\psi^{2,m})$,
        	and set the homogeneous Dirichlet boundary conditions on $\tilde{\psi}^{1,m}$ and $\tilde{\psi}^{2,m}$.
         \item Normalize the auxiliary step to obtain an attempt for the next step: 
         $$\psi^{\ell,{m+1}}=\frac{\tilde{\psi}^{\ell,m}}{\|\tilde{\psi}^{\ell,m}\|_\Delta}\sqrt{N_\ell M},\qquad \ell=1, 2.$$
       \item If $E^\Delta_{\varepsilon,\delta}(\psi^{1,{m+1}},\psi^{2,{m+1}})> E^\Delta_{\varepsilon,\delta}(\psi^{1,m},\psi^{2,m})$, then we replace $h$ by $h/2$ (provided $h/2 \geq h_0$, otherwise we stop without convergence) and we compute a new auxiliary step $(\tilde{\psi}^{1,m}, \tilde{\psi}^{2,m})$ for the same $m$ by going back to step \ref{item:step2}.
         Otherwise, we compute $K^\Delta$ at point $(\psi^{1,m+1}, \psi^{2,m+1})$
         using \eqref{eq:defKdelta}.
         If $K^\Delta \geq K_0$, then we replace $m$ by $m+1$ and we restart at step \ref{item:step1}.
         Otherwise, we take $(\psi^{1,m+1}, \psi^{2,m+1})$
         as an approximate minimizer and we stop with convergence.
    \end{enumerate}    
\end{itemize}

\section{Post-processing algorithms}\label{sec:post_proc}
\subsection{Indices of the Vortices}
\label{subsec:index}
In this section, we introduce an algorithm that we developed for the computation
of the indices of the vortices in the minimizers of the discrete energy functional defined in
\eqref{energie_discr} under the constraints \eqref{eq:contraintesdiscretes}, obtained using
the minimization method described in Section \ref{sec:discretmin}.
This algorithm is used in Section \ref{sec:numericalresults}.

The algorithm relies on 4 numerical parameters $tol_1>0$, $tol_2>0$, $N_{min}\in \mathbb{N}^*$ and $N_{max}\in\mathbb{N}^*$ with $N_{min}\leq N_{max}$. It follows the 4 steps below. The three first steps identify the vortices' centers and the last step computes the vortices' indices. In this section, $\psi$ denotes the square complex matrix with $N^2$ entries for either of the two components of the Bose--Einstein condensate. 

\underline{Step 1}: We determine the potential centers of the vortices and establish a list of candidates as the set of $(n,k)\in \{1,\hdots,N\}^2$ such that $|\psi_{n,k}|^2<tol_1$.

\underline{Step 2}: We build a second list $\mathbb{P}$ based on the first list above using the following rule. For each potential center $(n,k)$ in the list established in Step 1, we consider the values of $|\psi|^2$ on the squares 
\begin{equation*}
    \mathbb{S}_\lambda(n,k)=\bigg\{(n\pm\lambda,k+j) ~\big|~ j\in\{-\lambda,\hdots,\lambda \}  \bigg\}\bigcup\bigg\{(n+j,k\pm\lambda) ~\big|~ j\in\{-\lambda,\hdots,\lambda \}\bigg\},
\end{equation*}
of length $2\lambda\delta_x$, for $\lambda \in \{N_{min},\hdots,N_{max}\}$. If for one of these $\lambda=\lambda(n,k)$, the values of $|\psi|^2$ at all points of the square $\mathbb{S}_\lambda(n,k)$ are such that $|\psi|^2-|\psi_{n,k}|^2>tol_2$, then we add the center $(n,k)$ to the second list $\mathbb{P}$, and we set $\lambda(n,k)$ as the caracteristic length of the potential vortex.
In other words, we have determined a list $\mathbb{P}$ of couple of points $(n,k)$ satisfying the following conditions:
\begin{itemize}
    \item $|\psi_{n,k}|^2<tol_1$,
    \item $|\psi_{i,j}|^2>|\psi_{n,k}|^2+tol_2$, for all the couples $(i,j)$ in $\mathbb{S}_{\lambda(n,k)}(n,k)$. 
\end{itemize} 

\underline{Step 3}: We consider each center $(n,k)$ from the list $\mathbb{P}$ and we identify if another center is inside the square $\bigcup_{1 \leq \lambda \leq \lambda(n,k)} \mathbb{S}_\lambda(n,k).$
If this is the case, we eliminate the center with the biggest $|\psi|^2$ at the center from the list. We repeat this step until we are left with isolated centers.
Let us denote by $\mathbb{T}$ the list of all the couples $(n,k)$ corresponding to isolated centers.

\underline{Step 4}: We compute the indices using the following rules. For each $(n,k)\in \mathbb{T}$, we start off with any couple $(i,j)$ in $\mathbb{S}_{\lambda(n,k)}(n,k)$. Then we compute their associated angle $\theta_0=\arg(\psi_{i,j})$. Note that there are $8\lambda(n,k)$ couples in $\mathbb{S}_{\lambda(n,k)}(n,k)$ (see Figure \ref{fig:S_k}). After computing the first angle $\theta_0$, we proceed to compute the other angles $\theta_1, \theta_2, \cdots ,\theta_{8\lambda(n,k)}$ in the following way:
    \begin{itemize}
        \item After computing the angle $\theta_m$, the next angle $\tilde{\theta}_{m+1}$ is computed as an argument of the next value of $\psi$ on the square $\mathbb{S}_{\lambda(n,k)}(n,k)$ with anticlockwise orientation (see figure \ref{fig:S_k}).
        \item We set $\theta_{m+1}:=\tilde{\theta}_{m+1}+2k\pi$ with $\dis k={\rm argmin}_{l\in\mathbb{Z}} |\tilde{\theta}_{m+1}-\theta_m+2\pi l|$.
    \end{itemize}
Eventually, index of the vortex $(n,k)$ is equal to $(\theta_{8\lambda(n,k)}-\theta_0)/2\pi$.
\begin{figure}[ht]
    \centering
    \includegraphics[width=0.5\textwidth]{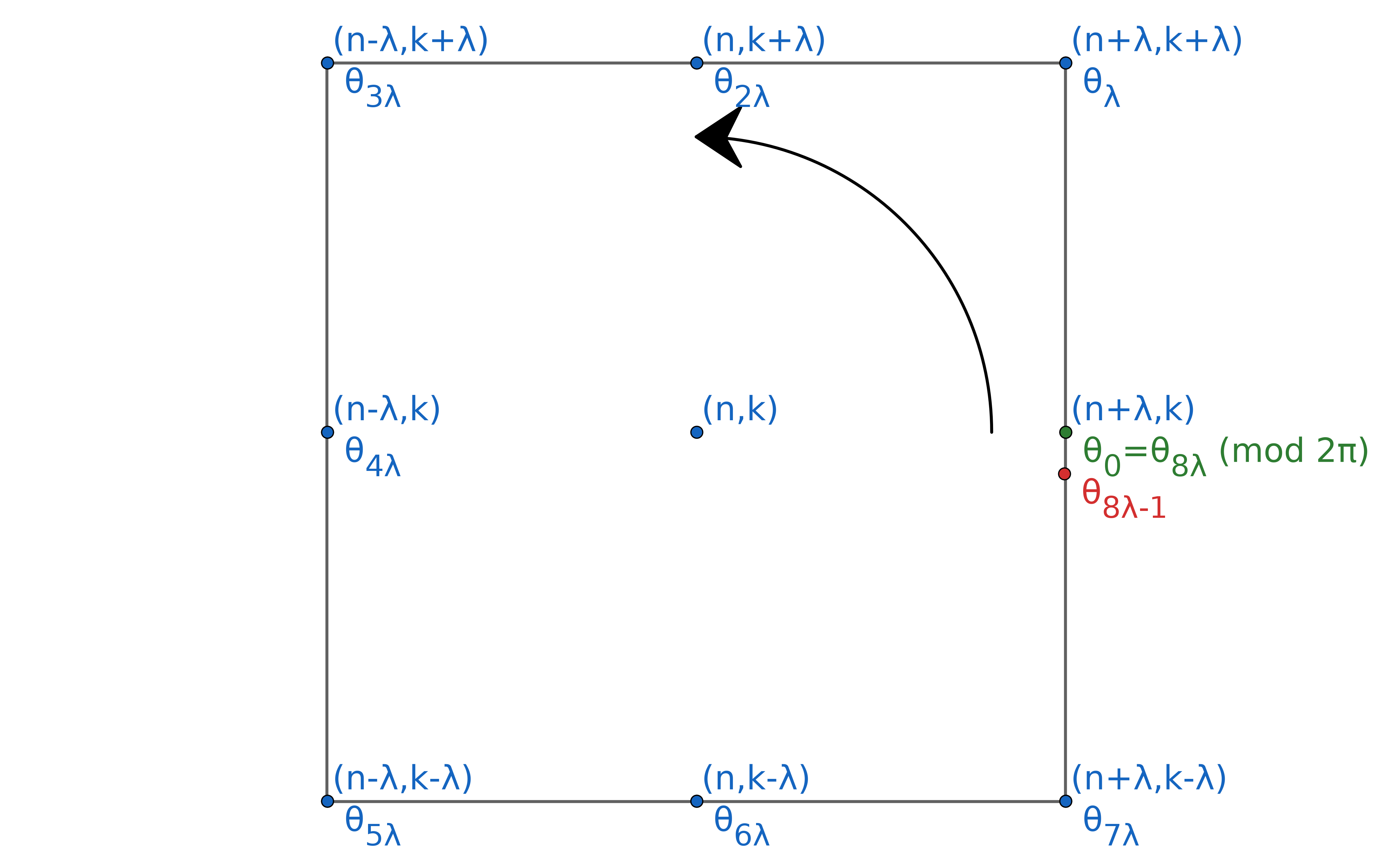}
    \caption{$\mathbb{S}_\lambda(n,k)$ with angles $\theta_0$ to $\theta_{8\lambda}$.}
    \label{fig:S_k}
\end{figure}
\subsection{Indices of the vortex sheets}\label{subsec:sheet_index}
In this section, we introduce an algorithm that we developed
for the computation of the index along a vortex sheet of a minimizer
of the discrete energy functional defined in \eqref{energie_discr} under the constraints
\eqref{eq:contraintesdiscretes} using the minimization method described
in Section \ref{sec:discretmin}.
Let us note that this algorithm {\it requires} some input from a human being at some point.
This algorithm is used in Section \ref{sec:numericalresults}.

The algorithm relies on 5 main steps.
It uses 3 numerical values: $m \leq M$ (close to $1/2$) and a small $tol_3>0$.
The first step consists in identifying the contours of the vortex sheets.
The second step splits these contours into discrete connex curves.
The third step requires a human being to decide whether each contour should be discarded or
merged with another one.
The fourth step sorts the optimized contours, and provides them with an orientation.
The fifth and last step computes each vortex sheet's index alongside its contour.

\begin{figure}[ht]
	\centering
	\begin{subfigure}[t]{0.45\textwidth}
		\centering
		\includegraphics[width=0.8\textwidth]{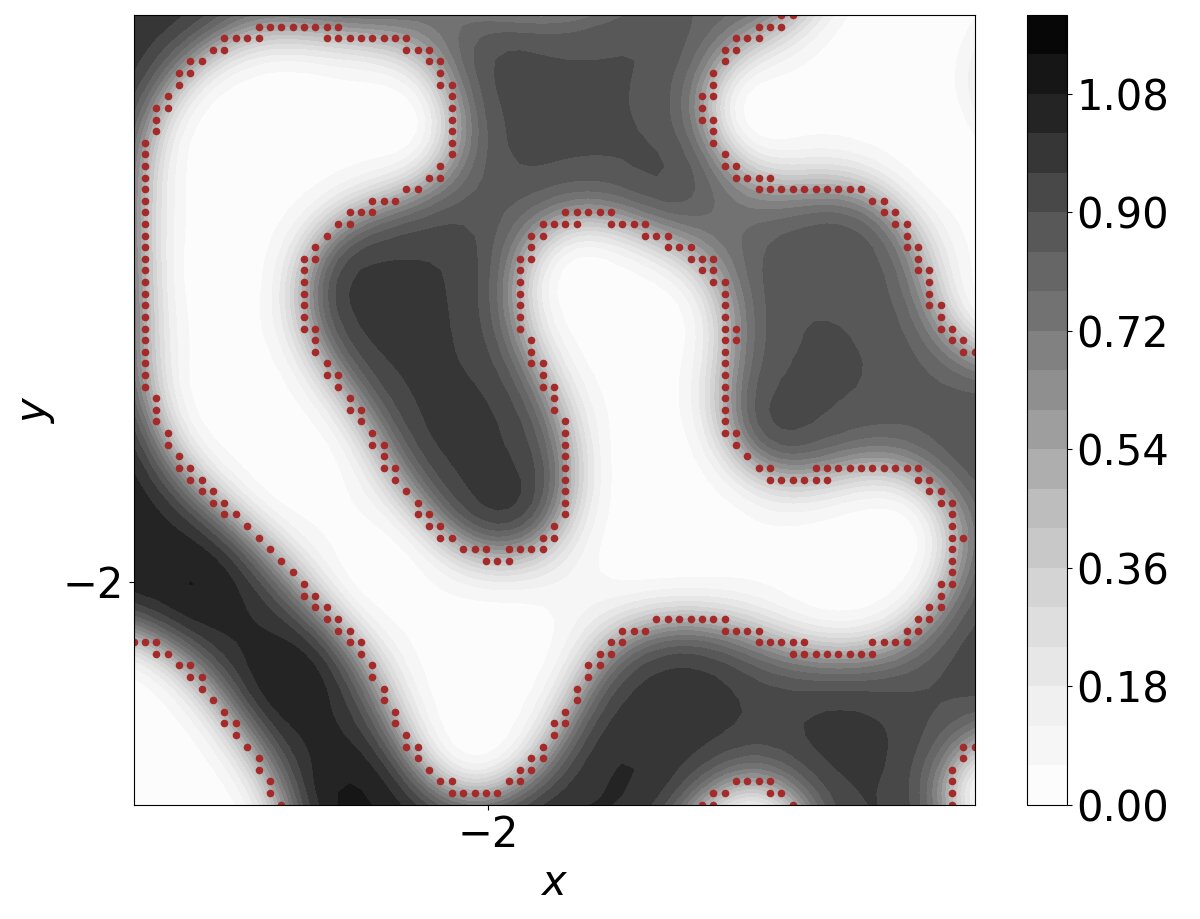}
		\subcaption{An example of the contours detected after the first part of the first step
                  of the contour detection algorithm.}
                \label{subfig:etape1a}
	\end{subfigure}
	\hfill
	\begin{subfigure}[t]{0.45\textwidth}
		\centering
		\includegraphics[width=0.8\textwidth]{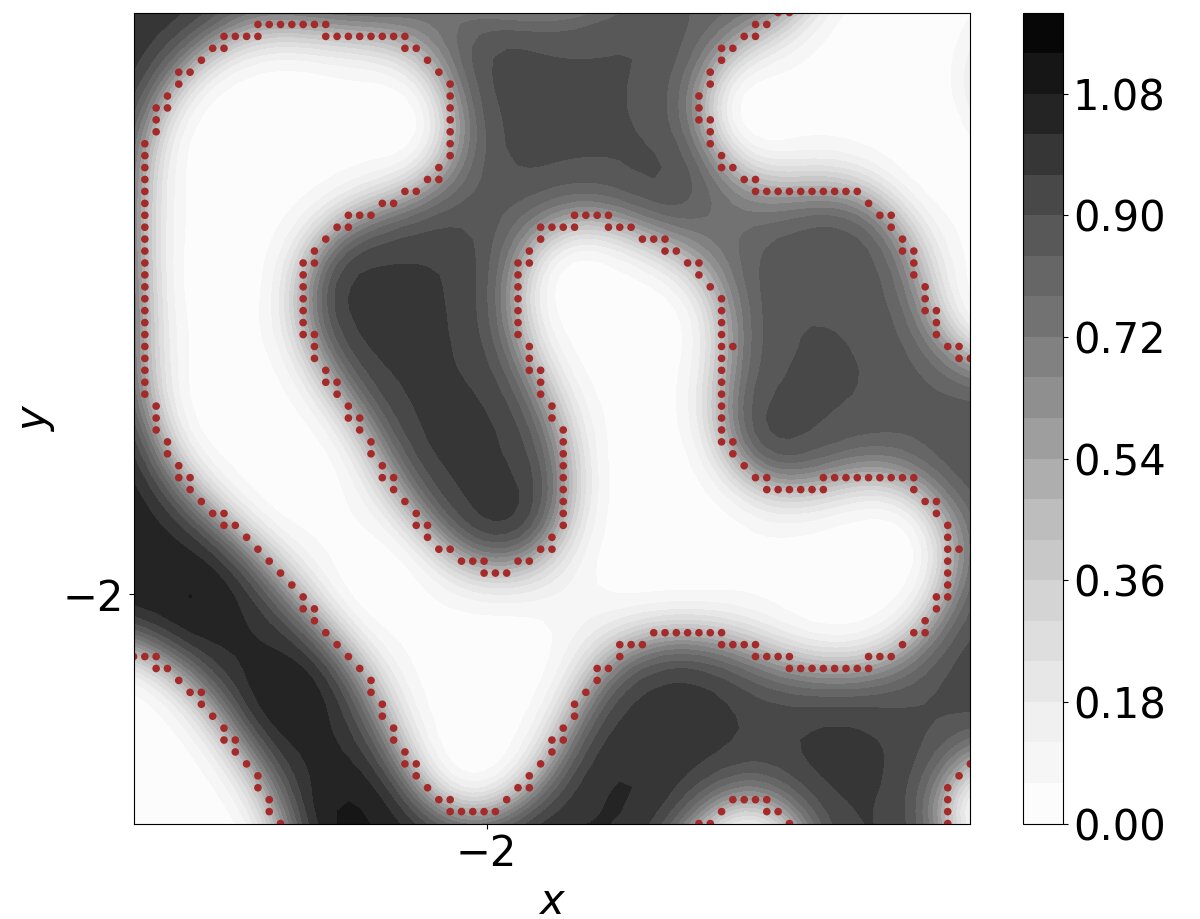}
		\subcaption{An example of the {\it connex} curves obtained at the end of the first
                  step of the contour detection algorithm.}
                \label{subfig:etape1b}
	\end{subfigure}
	\caption{An example of the contours detected after the first step of the contour
          detection algorithm ($m=0.4$, $M=0.6$, and $tol_3=0.3$).
          We display just a part of the minimizer obtained in Fig \ref{fig:2_espece_om_15_seg_1}.}
      \end{figure}
      
\underline{Step 1}: The contour of a vortex sheet in either of the components of the condensate
consists in a region where the squared modulus of the minimizer moves fast from $0$
to $1=\max \rho$ 
(or the other way around).
The first step consists in finding regions on the discrete square $[-L,L]^2$ where this occurs.
We identify, for each component $\psi^\ell$, the coordinates on the grid for the values $|\psi^\ell|^2$
between $m$ and $M$.
Next, we add the grid coordinates close to the circle $\partial D$,
of which we retain only the coordinates where $|\psi^\ell|^2\leq tol_3$.
These grid points constitute the union of potential contours of vortex sheets
(see Figure \ref{subfig:etape1a} for an example).
Let us sort these points, to ease the search for connex components in contours in the next step.
Let us note $\mathbb{K}$ the set of coordinates we have found so far.
We are looking for a union of {\it connex} curves that describe the borders of each vortex sheet.
Each curve can be determined with as many points as we want,
but we want to avoid taking too many points per curve. Therefore,
we limit the number of neighbours to any point on the curve to $3$.
For each grid point of coordinates $(i,j)$, we denote by $\mathbb{S}_1(i,j)$ the set of its
neighbours.
For all grid point $(i,j) \in \mathbb{K}$, we act as follows:
\begin{enumerate}
	\item If $(i,j)$ has no neighbours, we just remove it from the set $\mathbb{K}$.
	\item If $(i,j)$ has 1 neighbour, we add to the set $\mathbb{K}$ a couple $(i',j')\notin \mathbb{K}$ verifying $$(i',j')={\rm argmin}_{(n,k)\in \mathbb{S}_1(i,j)}\big||\psi_{i,j}|^2-|\psi_{n,k}|^2\big|.$$
        \item If $(i,j)$ has 2 or 3 neighbours, we do nothing.
	\item If $(i,j)$ has at least 4 neighbours, we remove the couple $(i',j')$ from $\mathbb{K}$ verifying: $$(i',j')={\rm argmax}_{(n,k)\in \mathbb{S}_1(i,j)\cap \mathbb{K}}\big||\psi_{i,j}|^2-|\psi_{n,k}|^2\big|,$$ and we repeat this step until we are left with only 3 neighbours.
\end{enumerate}        
At this point, we have detected a union of different curves defining the borders of the vortex sheets.
An example if displayed in Fig \ref{subfig:etape1b}.
Our next step is to separate the connex components in $\mathbb K$,
to categorize the different sheets (since each minimizer can have more than one vortex sheet).
\begin{figure}[ht]
	\centering
	\begin{subfigure}[t]{0.45\textwidth}
		\centering
		\includegraphics[width=0.8\textwidth]{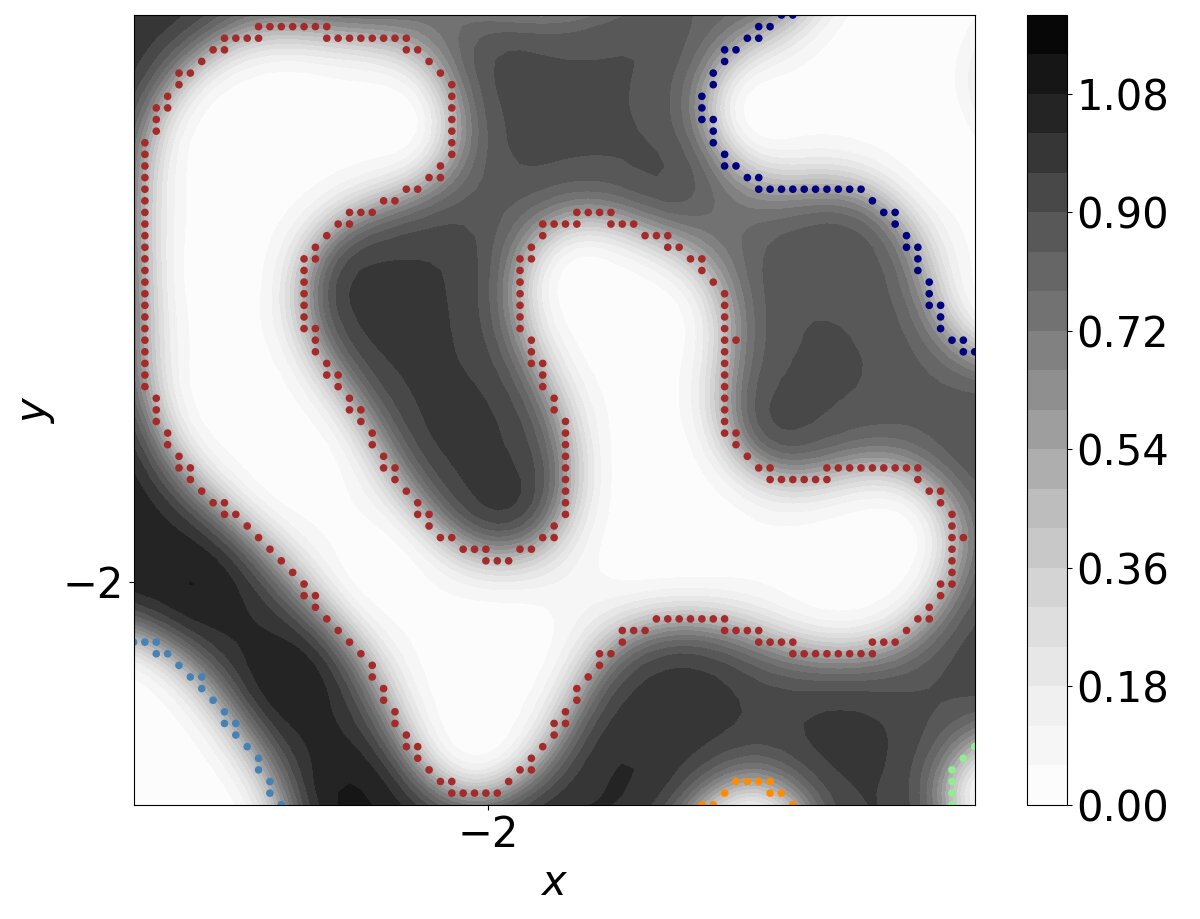}
		\subcaption{After Step 3: We can use a different colour of each contour
                  (the colour changes with $p$).}
                \label{subfig:etape2}
	\end{subfigure}
	\hfill
	\begin{subfigure}[t]{0.45\textwidth}
		\centering
		\includegraphics[width=0.8\textwidth]{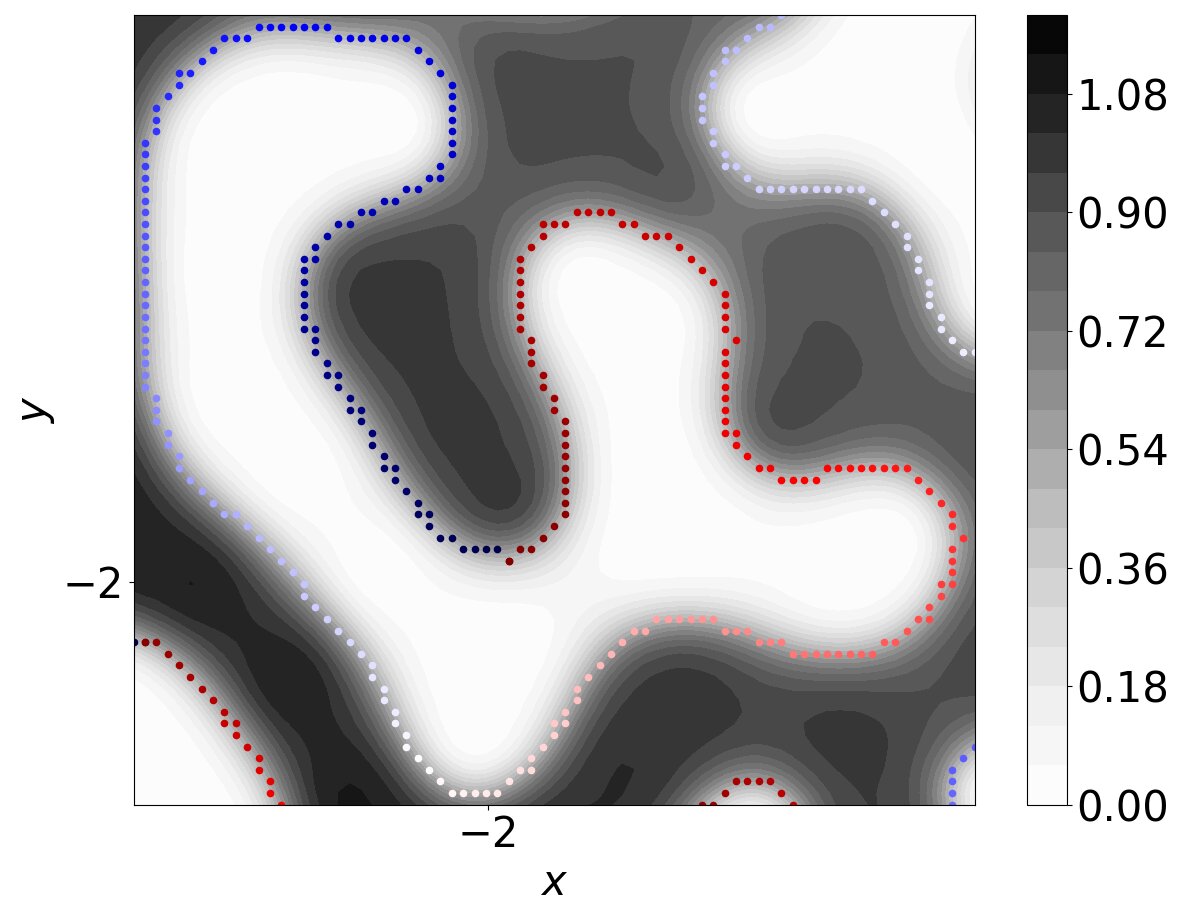}
		\subcaption{After Step 4: We can use color shading to show the
                  anticlockwise orientation of each contour.}
                \label{subfig:etape3}
	\end{subfigure}
	\caption{An example of the contours detected and orientated by the contour detection
          algorithm after the third and fourth steps.}
 \end{figure}

 \underline{Step 2}: We partition $\mathbb K$ into the union of a finite number
 of contours $\mathbb K_p$ corresponding to {\it connex} curves.
 We start with $p=1$. At the $p^{\rm th}$ vortex sheet, we create dynamically a new list
 $\mathbb{K}_p=\emptyset$.
We start by adding to $\mathbb K_p$ a grid point at random from the list $\mathbb{K}$
and we remove it from $\mathbb{K}$.
Then, we add all of its neighbours one by one to $\mathbb K_p$ and we remove them from $\mathbb{K}$.
We repeat this until we are left with no neighbours in $\mathbb{K}$
to all the grid points in $\mathbb{K}_p$.
Then, as long as there are points in $\mathbb K$,
we increase $p$ to $p+1$ and restart this step again until $\mathbb K$ is empty.

\underline{Step 3}: After dividing the contours of the vortex sheets into discrete connex curves
$\mathbb K_p$, we allow for human intervention.
First, a human being decides if each $\mathbb K_p$ should be considered or ignored.
Second, a human being decides whether couples of two discrete connex curves
$\mathbb{K}_p$ and $\mathbb{K}_{p'}$ should be merged or not.
To merge two discrete connex curves $\mathbb{K}_p$ and $\mathbb{K}_{p'}$,
we follow these steps.
First, we search for the closest two couples $(n,k)$ and $(n',k')$ from $\mathbb{K}_q$
and $\mathbb{K}_{q'}$ (using the usual distance in $\R^2$) and we compute their middle
$(i',j')=(\lfloor \frac{n+n'}{2} \rfloor,\lfloor \frac{k+k'}{2} \rfloor)$.
If this middle has neighbours in both categories,
we merge them into a new set $\mathbb{K}_q \cup \mathbb{K}_{q'} \cup (i',j')$.
Otherwise, we repeat {\it once} the same process between $\mathbb{K}_q \cup\{(i',j')\}$
and $\mathbb{K}_{q'} \cup\{(i',j')\}$.
In the end, we are left with different categories for different vortex sheet's contours.
An example of what we obtain after Step 3 is displayed in Fig \ref{subfig:etape2}.

\underline{Step 4}: We build, from each set $\mathbb K_p$ a {\it list} of grid points corresponding
to an anticlockwise path along the countour of the $p^{\rm th}$ vortex sheet.
We proceed in the following way:
\begin{enumerate}
\item For each list $\mathbb{K}_p$, we compute it's barycenter.
  We choose our starting couple from $\mathbb{K}_p$ as one with a close $x$-axis coordinate
  to that of the barycenter and with the biggest $y$-axis coordinate.
  We add this to our new list $\mathbb{K}^{\rm sort}_p$.
\item While the last point and the first point of $\mathbb{K}_p^{sort}$ are not neighbours:
	\begin{enumerate}
        	\item Out of all the neighbours in $\mathbb{K}_p$ to the last grid point we added
        	to $\mathbb{K}_p^{\rm sort}$, we choose one and prioritize the anticlockwise direction.
        	Then, we remove this grid point from $\mathbb{K}_p$.
        	\item If we don't find any neighbour to the last grid point we added do
        	$\mathbb K_p^{\rm sort}$, we delete this grid point from $\mathbb{K}_p^{\rm sort}$
        	and from $\mathbb{K}_p$. Next, we return to the last grid point
        	we added to $\mathbb K_p$ before that and we repeat the previous step. 
	\end{enumerate}
      \item If the length of $\mathbb{K}_p^{\rm sort}$ is less than 70\% of the original size
        of $\mathbb{K}_p$, then we delete the last grid point added from $\mathbb{K}_p^{\rm sort}$
        and from $\mathbb{K}_p$.
        Next, we return to the (new) last grid point we added to $\mathbb K_p^{\rm sort}$
        before that and we repeat the second point from Step 4.
        If the length of $\mathbb{K}_p^{\rm sort}$ exceeds 70\% of the original size of $\mathbb{K}_p$,
        then we stop the algorithm.
      \end{enumerate}
      An example of what we obtain after Step 4 is displayed in Fig \ref{subfig:etape3}.
      
  \begin{remark}
    In the third point of step 4, we choose as criterion $70\%$ arbitrarily.
    Other choices can be made as long as it is strictly bigger than $50\%$.
    If we choose a too high value, then the algorithm could be so constrained
    that it might not give a proper result.    
  \end{remark}
      
\underline{Step 5}: We compute the indices of each vortex sheet contour detected
$\mathbb{K}_p^{\rm sort}$.
The algorithm is similar to the last step in Section \ref{subsec:index} with the list
$\mathbb{K}_p^{\rm sort}$ instead of $\mathbb{S}_{\lambda(n,k)}(n,k)$.

\begin{remark}[Computation of the index of a giant hole]\label{rem:hole}
  For an estimation of the index of a giant hole, we first estimate its radius numerically.
  The algorithm goes as follows. We are given integers $l_0\geq 1$ and $m_0>> 1$ and
  a tolerance parameter $tol_4>0$. We choose $tol_4=0.1$ in all the numerical experiments.
  For $l\in\{0,\cdots,l_0\}$, we choose $r_l>0$ a possible radius for the giant hole in such
  a way that $l\mapsto r_l$ is increasing, $r_0$ is too small a radius and $r_{l_0}$ is
  too big a radius.
  For $m\in\{0,\cdots,m_0\}$, we set $\theta_m=\frac{2\pi m}{m_0}$.
  Let $\psi$ be a minimizer with a giant hole.
  Starting from $l=0$ and as long as $l\leq l_0$, we compute successively
  \begin{equation*}
          \beta_l=\bigcup_{m=0}^{m_0-1}
          \left(
            \left(
              \underset{n\in\{1,\cdots,N\}}{{\rm argmin}} |x_n-r_l \cos(\theta_m)|
            \right)
            \times
            \left(
              \underset{k\in\{1,\cdots,N\}}{{\rm argmin}} |y_k-r_l \cos(\theta_m)|
            \right)
          \right).
   \end{equation*}
   If $\underset{{(n,k) \in \beta_l }}{\max} |\psi_{n,k}|^2>tol_4$,
   then we stop the algorithm and we will use $r_l$ and $\beta_l$ for the computation of the index.
   Otherwise we repeat this step with $l=l+1$.
   Finally, to compute the index of the giant hole, we use the last step of Section
   \ref{subsec:index} with $\beta_l$ instead of $\mathbb{S}_{\lambda(n,k)}(n,k)$.
  

  
        
\end{remark}

\section{Numerical results}
\label{sec:numericalresults}

This section is devoted to the numerical experiments carried out using the EPG minimization
method introduced in Section \ref{sec:discretmin}
and the post-processing algorithms described in Section \ref{sec:post_proc}
for the detection of structures and the computation of the indices.
In Section \ref{subsec:simulation_onecomponent}, we consider a one component Bose--Einstein
condensate in strong confinement and rotation, in connection with the theoretical
results of Section \ref{subsubsec:onecomponent}.
In Section \ref{subsec:simulation_twocomponent}, we consider a two components Bose--Einstein
condensate in a strong confinement regime in the segregation case, without rotation.
This allows to connect with the theoretical results of Section \ref{subsubsec:twocomponent}.
In Section \ref{subsec:twocompseg}, we consider a two components Bose--Einstein condensate
with strong confinement and rotation in the segregation case.
This allows to connect with the theoretical results of Section \ref{subsubsec:twocomponent}.
Last, in Section \ref{subsec:twocompcoex}, we consider a two components Bose--Einstein
condensate with strong confinement and rotation in the coexistence case.
This allows for connection with theoretical conjectures of Section \ref{subsubsec:twocomponentcoex}.

\subsection{One component condensate with rotation}
\label{subsec:simulation_onecomponent}

The goal of this section is to illustrate numerically the theoretical results
described in Section \ref{subsubsec:onecomponent}.
These results identify four different regimes for the behavior of the one component Bose--Einstein
condensate depending on how big $\Omega_\varepsilon$ is as $\varepsilon$ tends to zero.
These four regimes are separated by three characteristic rotational speeds
$\Omega_\varepsilon^i, ~i=1,2,3$ (see \eqref{eq:char_speed}).
We introduce below the parameters we use for all the one component simulations.
Then, we explain how we identify the four different regimes numerically.
We conclude with numerical simulations for small $\varepsilon$ in each of the four regimes. 

\subsubsection{Parameters used for the one component simulations}

For all the one component simulations, we consider the following parameters.
The confinement is defined by the function 
\begin{equation}
\label{eq:defrho}
\rho(x,y)=\min[1,10 (R^2-x^2-y^2)].
\end{equation}
For the initial datum we choose, $\psi^0(x,y)=\exp({-10x^2-10y^2})/5$.
The square of computation is of length $2L=14$ and the radius of the disc $D$ is $R=4$.
We initialize the step size to $h=0.1$ and we set $h_0=10^{-12}$.
For example, for $N=256$, we have $M=\|\sqrt{\rho_{>0}}\|_\Delta^2 \sim 50.084$.

\subsubsection{Identification of the four regimes}

In order to identify the four regimes, we proceed as follows.
First, we take $\varepsilon=10^{-1}$ and we use the EPG algorithm described
in Section \ref{subsec:methodegradient} to compute minimizers of $E_{\varepsilon,\delta}^\Delta$
for several rotational speeds.
Then, we do the same for $\varepsilon=5 \times 10^{-2}$.
Based on these many simulations, and the expressions \eqref{eq:char_speed},
we estimate the three critical values for the rotation speed which are shown in Figure
\ref{fig:estimationzones}.
This provides us with an estimation of the three critical values of the rotational
speed when $\varepsilon=10^{-2}$.
Choosing rotation speed between these 3 estimated critical values for $\varepsilon=10^{-2}$,
we manage to observe the four different regimes, as detailed below.


\begin{figure}[ht]
    \centering
    \includegraphics[width=0.5\textwidth]{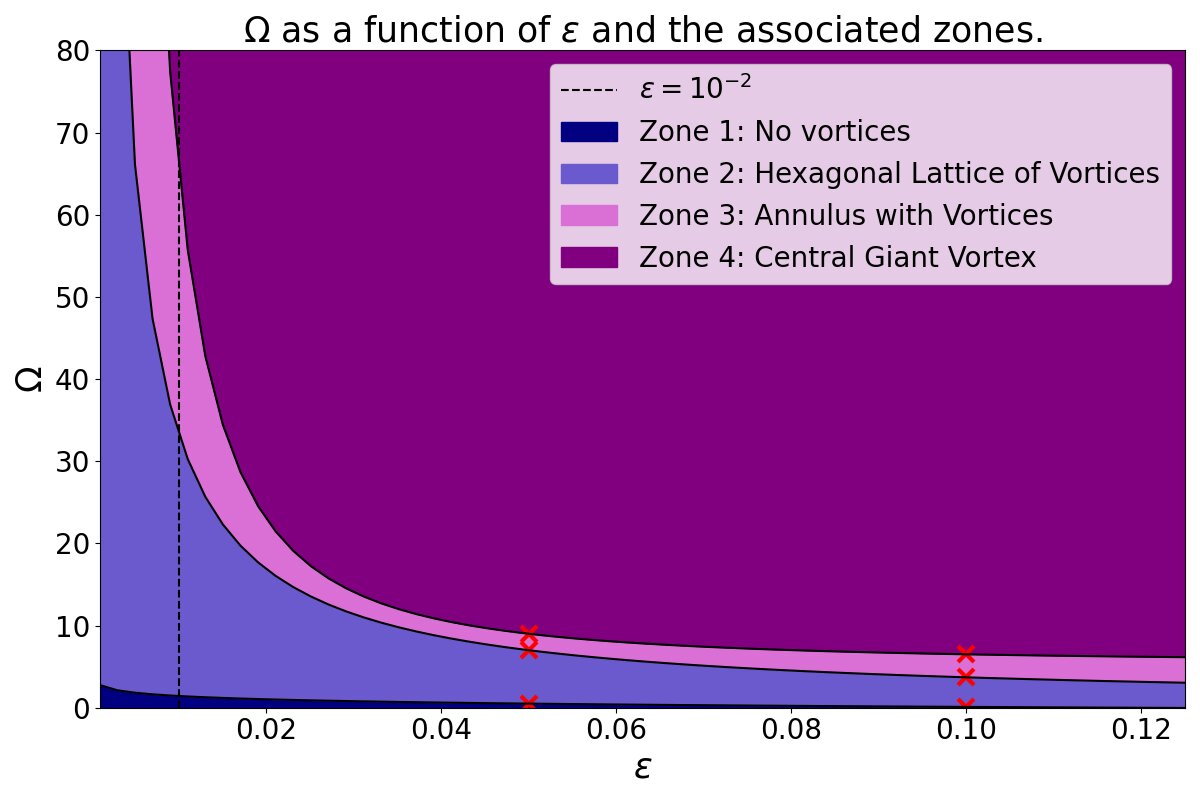}
    \caption{$\Omega$ as a function of $\varepsilon$ and the associated zones.}
    \label{fig:estimationzones}
\end{figure}

\subsubsection{Convergence test without rotation ($\Omega_\varepsilon=0$)}

For $N=128$ points, we perform a convergence analysis. For several values of $\varepsilon$ between $10^{-2}$ and $10^{0}$, we perform the gradient algorithm described in Section \ref{subsec:methodegradient} starting from $\psi^{1,0}=\sqrt{\rho_{>0}}$. We compute the residual $\||\psi^{1,m}|^2-{\rho_{>0}}\|_\Delta$ for the first index $m \geq 1$ for which $K^{\Delta}$ is beyond $10^{-0.0}$, $10^{-0.5}$, $10^{-1.0}$, $10^{-1.5}$, $10^{-2.0}$, $10^{-2.5}$, $10^{-3.0}$. Numerical results are displayed in Figure \ref{fig:convergence_test}. 

\begin{figure}[ht]
    \centering
    \input{test_convergence_N_128}
    \caption{Convergence test for the gradient method with $N=128$ for one component without rotation.}
    \label{fig:convergence_test}
\end{figure}
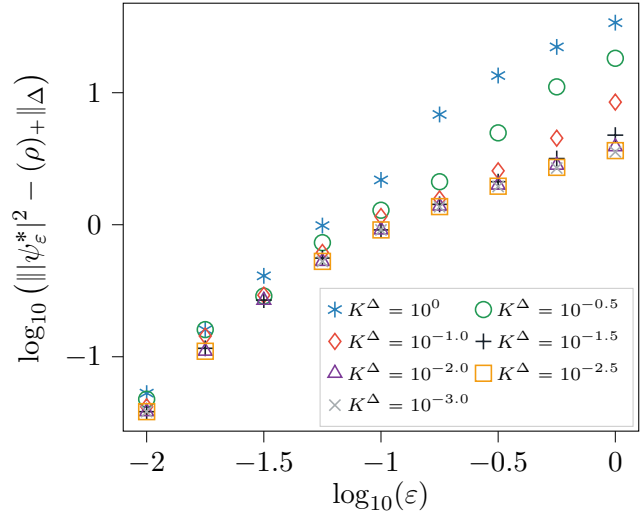

In this simple case (one component without rotation), where the limit as $\varepsilon$ tends to zero of the minimizer is known, we observe numerically that $K^{\Delta}=10^{-2}$ seems to be a good criterion to stop the gradient algorithm when $\varepsilon$ is between $10^{-2}$ and $10^{-1}$. Indeed, choosing $K^{\Delta}=10^{-2.5}$ or $K^{\Delta}=10^{-3}$ does not improve the convergence dramatically when compared with choosing $K^{\Delta}=10^{-2}$ on this numerical experiment when $\varepsilon$ is in the range of interest. 

Therefore, in the numerical experiments below, even with several components or with moderate rotation ($\Omega=\mathcal{O}(1/\varepsilon)$), we adopt $K^{\Delta}=10^{-2}$ as a criterion to stop the algorithm.

\subsubsection{Small rotational speed ($\Omega_\varepsilon < \Omega_\varepsilon^1$)}
We consider the case of no rotation ($\Omega=0$) and strong confinement ($\varepsilon=10^{-2}$).
We set the number of points in the $x$-axis and $y$-axis to $N=512$.
The gradient descent algorithm converged due to the stopping criterion
$K^\Delta$ (see \eqref{eq:defKdelta}) which has a value of $10^{-2}$.
The results are shown in Figure \ref{fig:1_espece_om_0}.
As we can see, there are no vortices in the numerical minimizer.
This is in accordance with the theory presented in section \ref{subsubsec:onecomponent}
(first case $\Omega < \Omega_\varepsilon^1$).
Moreover the squared modulus of a minimizer converges to the non-negative part of the function $\rho$
when $\varepsilon$ tends to 0, as illustrated in Figure \ref{fig:rho_and_minimizer}.

\begin{figure}[ht]
    \centering
    \begin{subfigure}[t]{0.45\textwidth}
        \centering
        \includegraphics[width=0.8\textwidth]{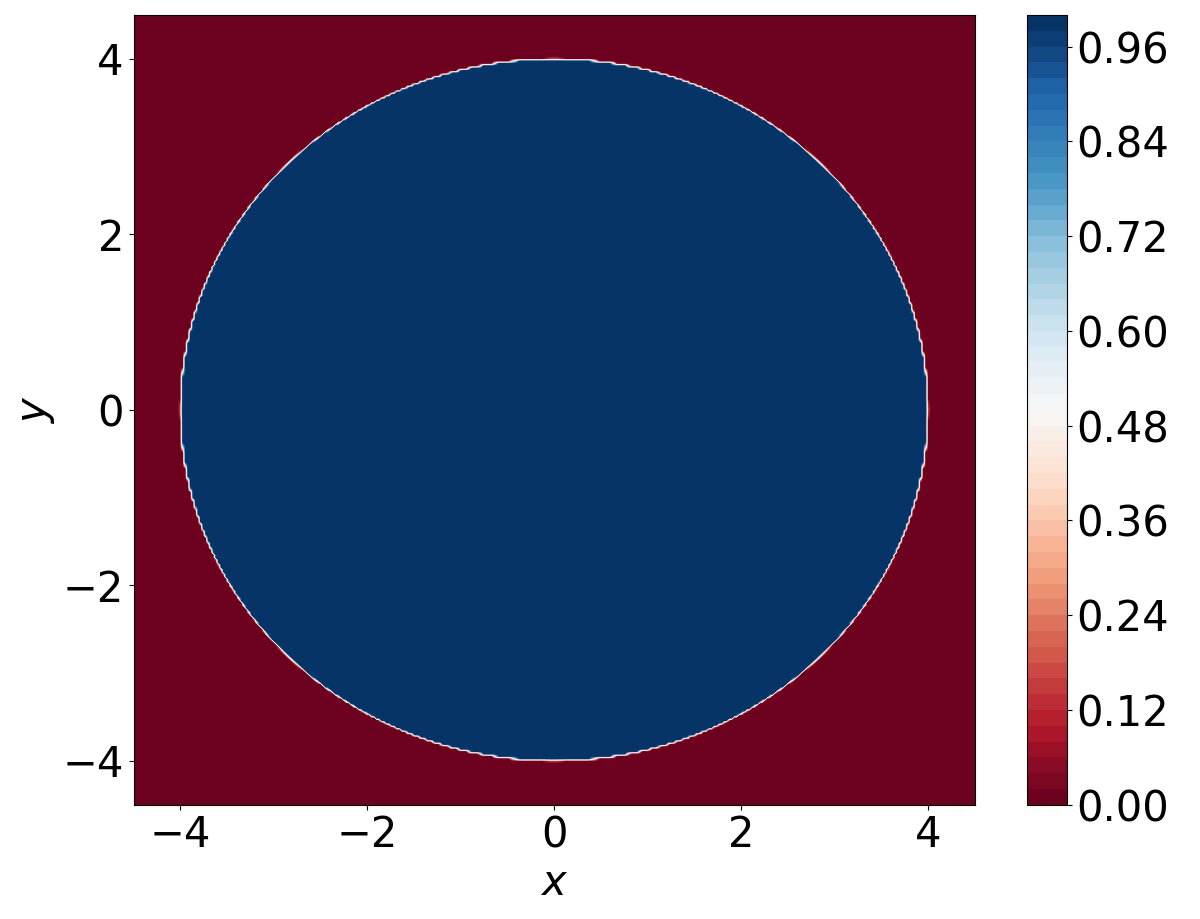}
        \subcaption{The graph of the non-negative part of $\rho$.}
    \end{subfigure}
    \hfill
    \begin{subfigure}[t]{0.45\textwidth}
        \centering
        \includegraphics[width=0.8\textwidth]{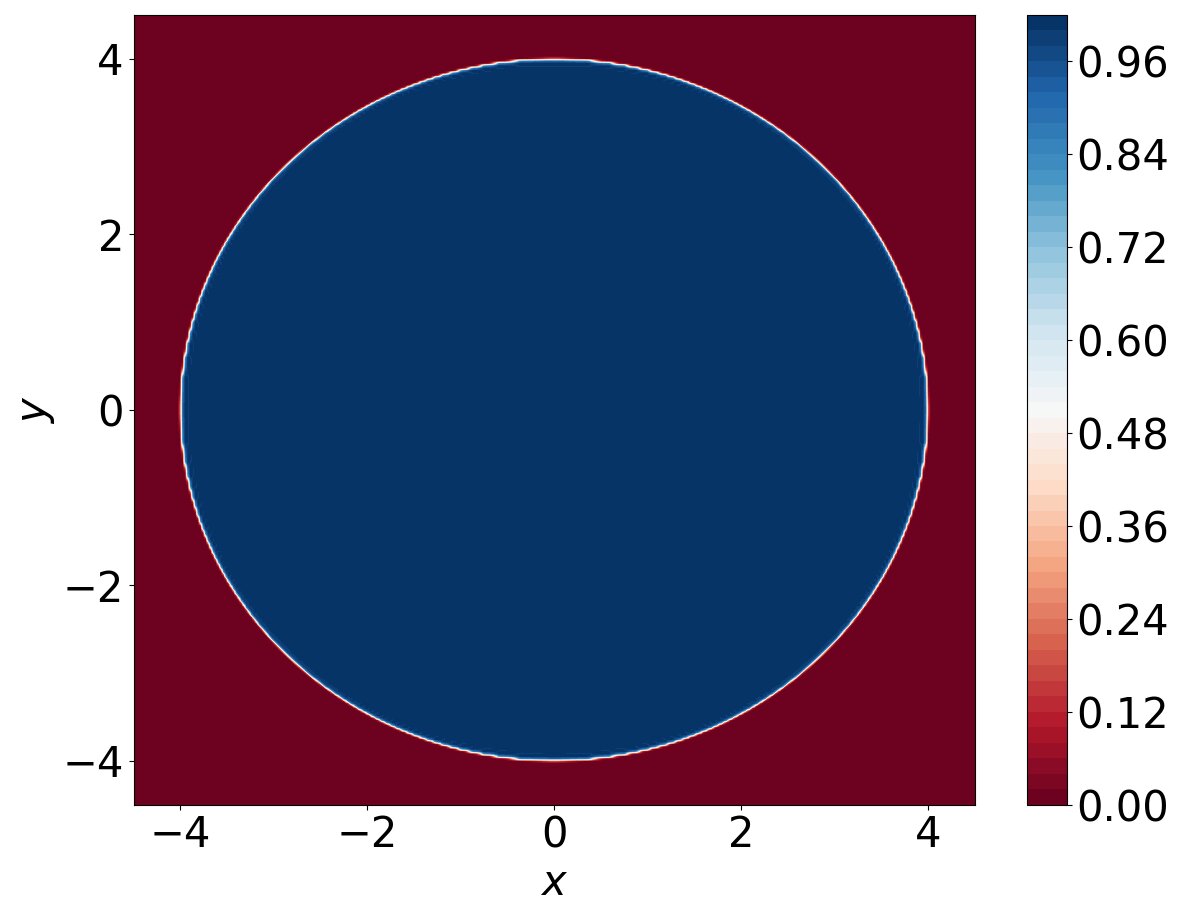}
        \subcaption{The squared modulus of a minimizer.}
    \end{subfigure}
\caption{Comparison between the non-negative part of $\rho$ (figure (A)) and the squared modulus of a minimizer for the energy $E_{\varepsilon}^\Delta$ (figure (B)) with no rotation ($\Omega=0$) and strong confinement ($\varepsilon=10^{-2}$) computed with $N=512$ and $K^\Delta=10^{-2}$.}
\label{fig:1_espece_om_0}
\end{figure}

\begin{figure}[ht]
    \centering
    \includegraphics[width=0.75\textwidth]{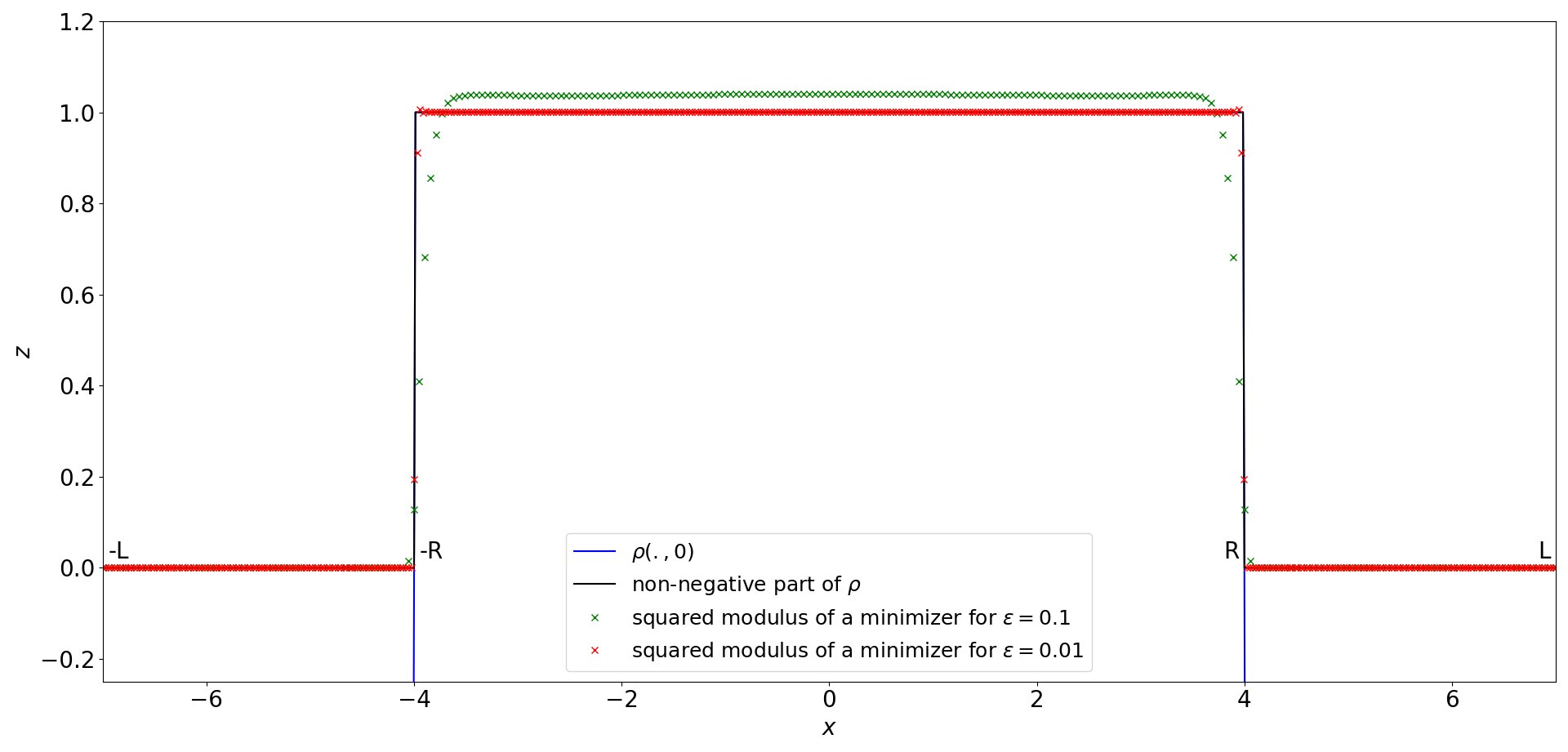}
    \caption{Comparison of $\rho$ and the squared modulus of a minimizer when $\varepsilon=0.1$
      and $\varepsilon=0.01$ along the $y$-axis.
    }
    \label{fig:rho_and_minimizer}
  \end{figure}

Figure \ref{fig:rho_and_minimizer} numerically illustrates the convergence of the squared modulus
of a minimizer for the energy $E^\Delta_{\varepsilon}$ to the non-negative part or $\rho$
when $\varepsilon$ tends to $0$ at fixed rotation speed $\Omega=0$.

\subsubsection{Moderate rotational speed ($\Omega_\varepsilon^1<\Omega_\varepsilon < \Omega_\varepsilon^2$)}

We consider the case of a moderate rotation speed ($\Omega=20$) and strong confinement
($\varepsilon=10^{-2}$), according to Figure \ref{fig:estimationzones}.
First, we compute a minimizer with $N=512$.
Then, we interpolate its real and imaginary parts and compute another minimizer with $N=1024$.
The EPG method converged due to the stopping criterion $K^\Delta$
which has a value of $5\times10^{-2}$.
We compute numerically the indices of the vortices using the algorithm described
in Section \ref{subsec:index}.
The results are shown in Figure \ref{fig:1_espece_om_20}.
As we can see, there are several vortices in the numerical minimizer and
there is no sign of a giant hole in the center.
The numerical index of most of the numerical vortices is equal to one.
This validates numerically that most zeros of the wave function have a phase circulation.
This is in accordance with the second case of the theory
presented in Section \ref{subsubsec:onecomponent}. 

\begin{figure}[ht]
    \centering
    \begin{subfigure}[t]{0.45\textwidth}
        \centering
        \includegraphics[width=0.8\textwidth]{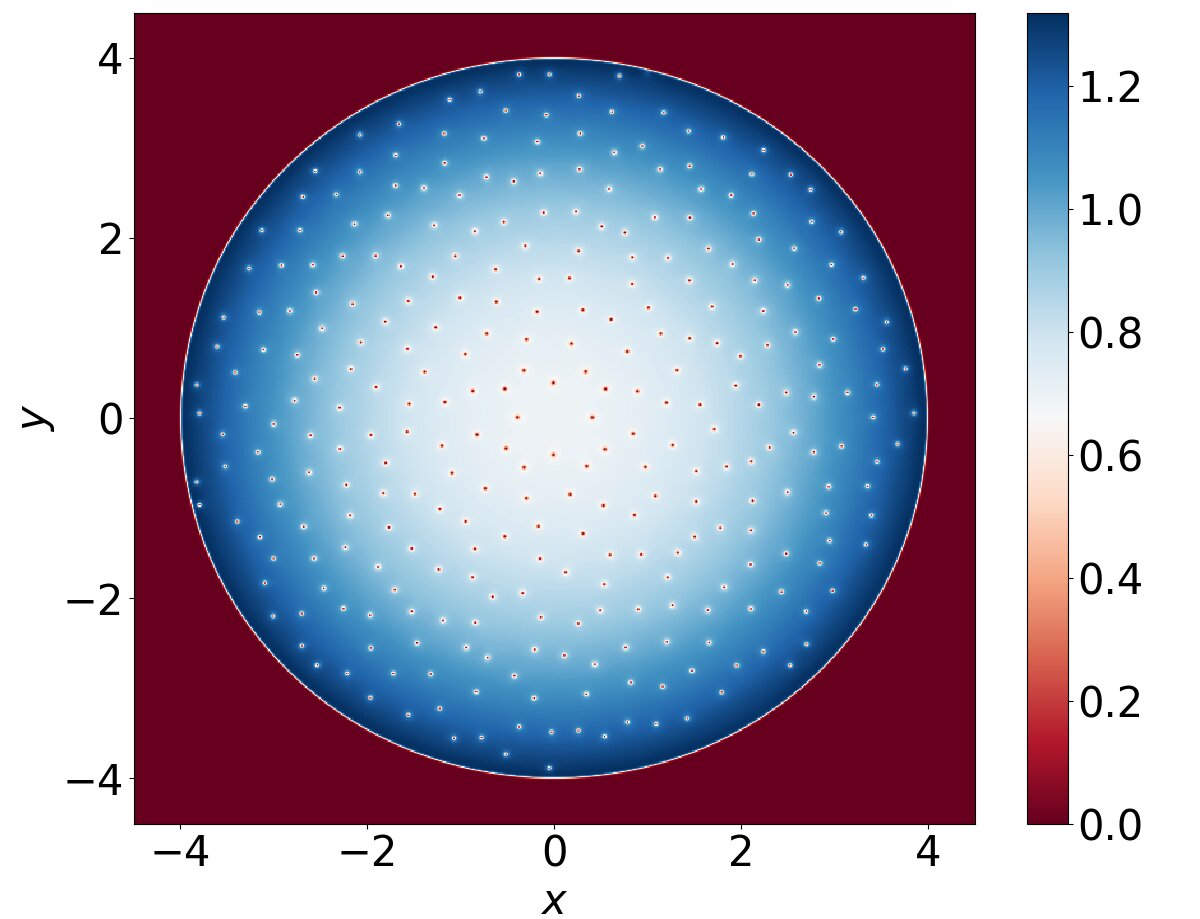}
    \end{subfigure}
    \hfill
    \begin{subfigure}[t]{0.45\textwidth}
        \centering
        \includegraphics[width=0.8\textwidth]{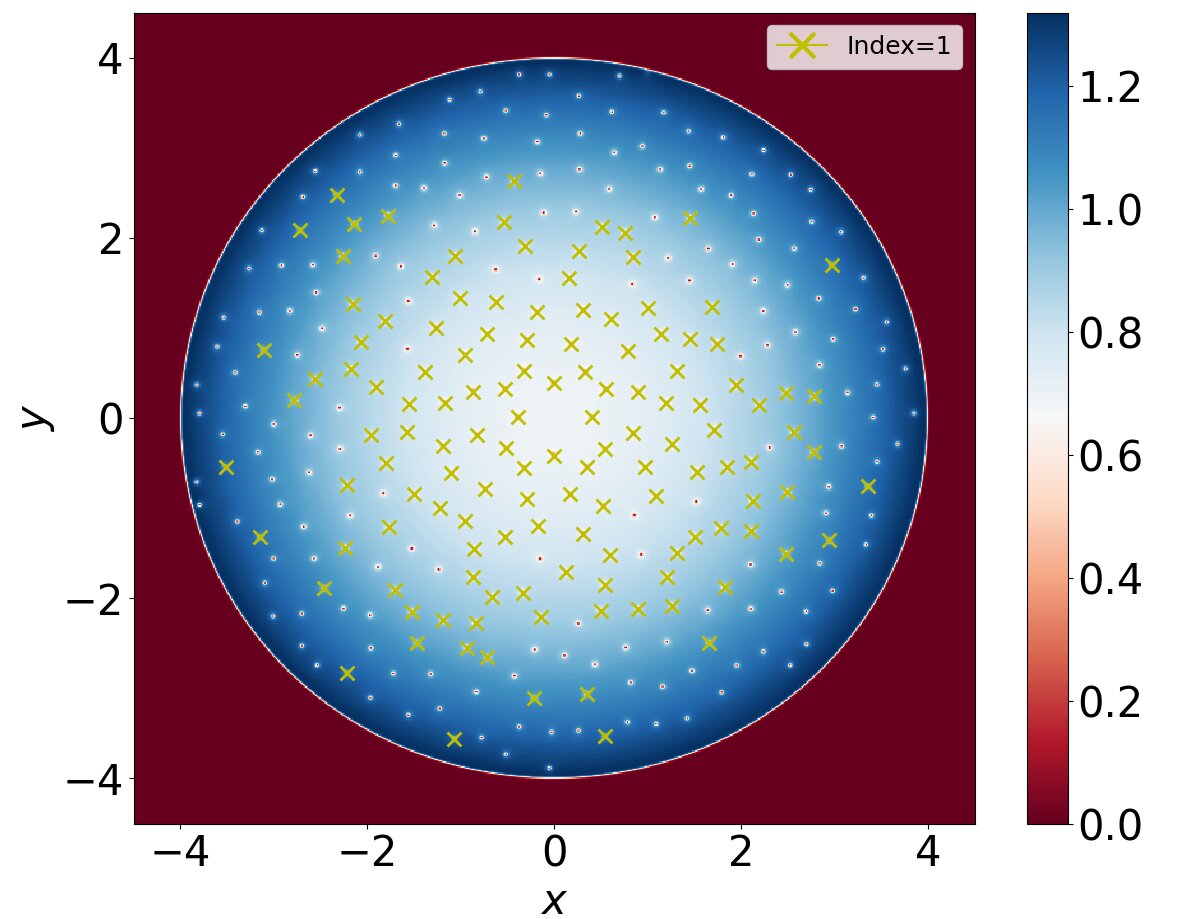}
    \end{subfigure}
\caption{The squared modulus of a minimizer for $\Omega=20$, $\varepsilon=10^{-2}$. In the right panel, the numerical computation of the indices of the vortices is carried out with $N_{min}=1$, $N_{max}=3$, $tol_1=0.1$ and $tol_2=0.05$.}
\label{fig:1_espece_om_20}
\end{figure}

In another simulation, we consider the case of a moderate yet higher rotation speed ($\Omega=30$)
and strong confinement ($\varepsilon=10^{-2}$).
Once again, we use an interpolation from $N=512$ to $N=1024$.
In this simulation, the EPG algorithm converged due to the stopping criterion $K^\Delta$
which has a value of $5\times10^{-2}$.
We also compute the indices of the vortices numerically using the algorithm described
in Section \ref{subsec:index}.
The results are shown in Figure \ref{fig:1_espece_om_30}.
As we can see, there are vortices in the numerical minimizer and there is a dark disk
in the center (not present in Figure \ref{fig:1_espece_om_20})
which indicates that we are close to the limit between zone $2$ and zone $3$
as predicted in Figure \ref{fig:estimationzones}.
The numerical index of most of the numerical vortices is equal to one.
This validates numerically that most zeros of the wave function have a phase circulation.
This 
is in accordance with the theory presented
in Section \ref{subsubsec:onecomponent}.

\begin{figure}[ht]
    \centering
    \begin{subfigure}[t]{0.45\textwidth}
        \centering
        \includegraphics[width=.8\textwidth]{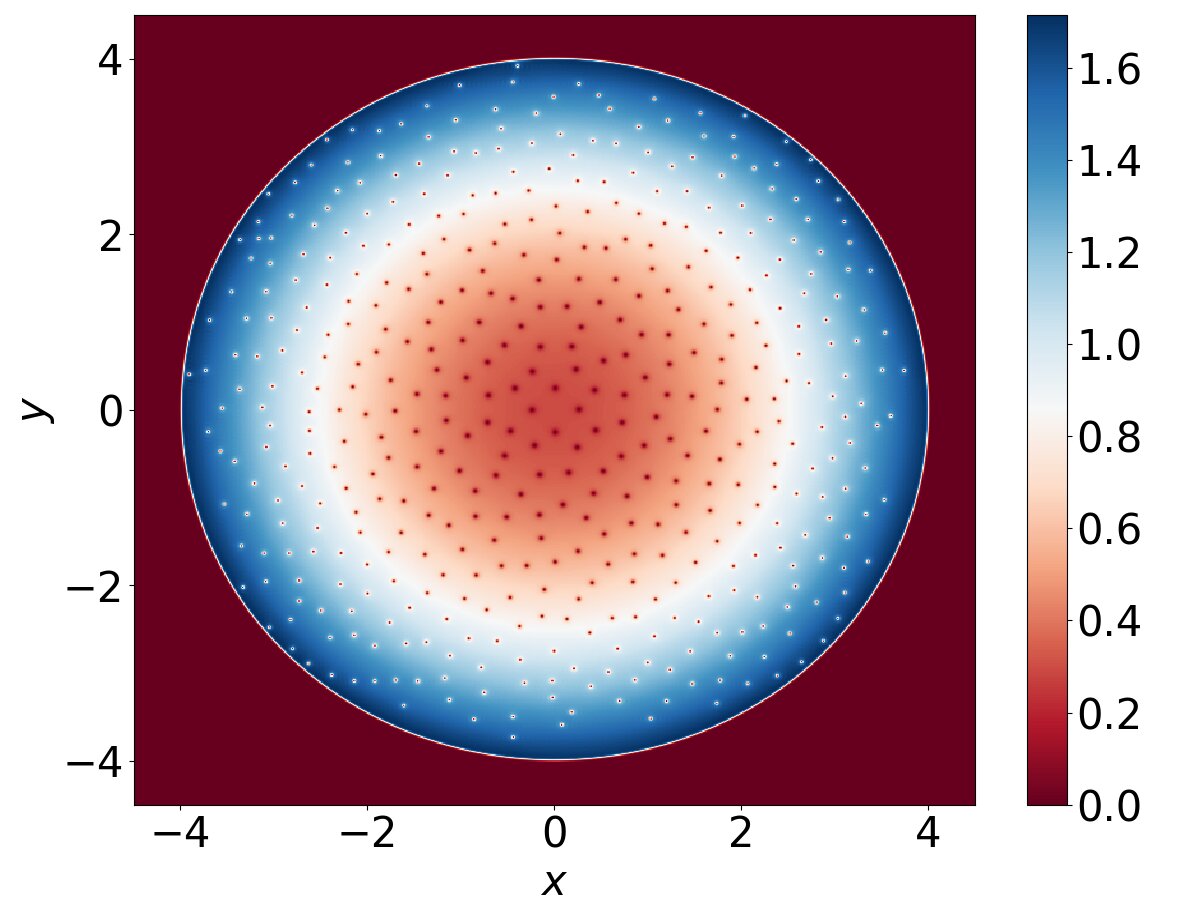}
    \end{subfigure}
    \hfill
    \begin{subfigure}[t]{0.45\textwidth}
        \centering
        \includegraphics[width=.8\textwidth]{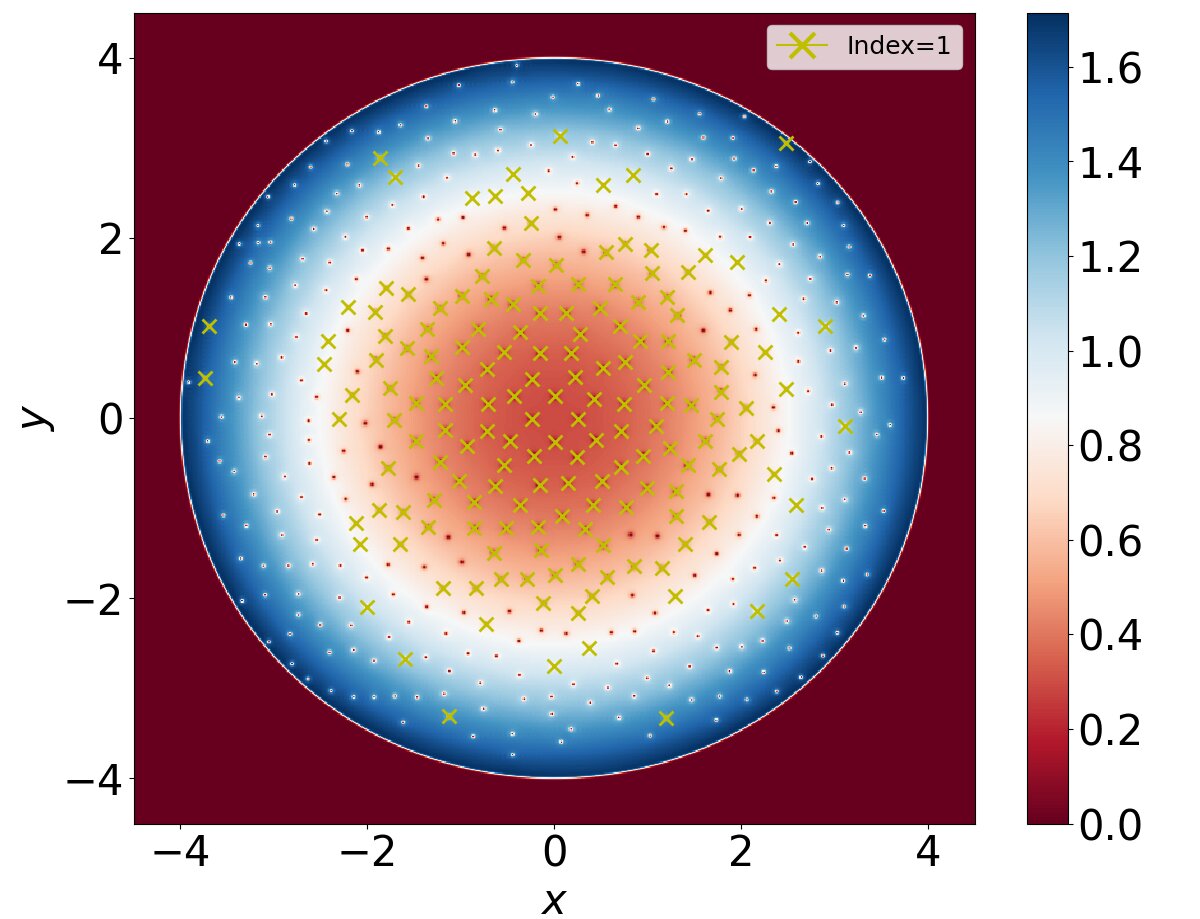}
    \end{subfigure}
\caption{The squared modulus of a minimizer for $\Omega=30$, $\varepsilon=10^{-2}$. In the right panel, the numerical computation of the indices of the vortices is carried out with $N_{min}=1$, $N_{max}=3$, $tol_1=0.05$ and $tol_2=0.02$. }
\label{fig:1_espece_om_30}
\end{figure}

\subsubsection{Big rotational speed ($\Omega_\varepsilon^2<\Omega_\varepsilon < \Omega_\varepsilon^3$)}

We consider the case of a fairly high rotation speed ($\Omega=40$) and strong confinement
($\varepsilon=10^{-2}$) in accordance with the regimes obtained in Figure \ref{fig:estimationzones}.
This simulation is carried out with $N=1024$.
The EPG algorithm converged due to the stopping criterion $K^\Delta \leq 5\times10^{-2}$.
The results are shown in Figure \ref{fig:1_espece_om_40}.
As we can see, there is a giant hole in the center,
surrounded by vortices on an annulus in the numerical minimizer.
This is due to the centrifugal force coming into play with the fairly high rotation speed
and is in accordance with the theory presented in section \ref{subsubsec:onecomponent}.

In Figures \ref{fig:1_espece_om_40_vortex_1} and \ref{fig:1_espece_om_40_vortex_2},
we compute the indices of the vortices of the numerical minimizer displayed
in Figure \ref{fig:1_espece_om_40}.
First, in Figure \ref{fig:1_espece_om_40_vortex_1}, we compute the indices of the vortices
in the annulus using the algorithm described in Section \ref{subsec:index}.
Then, in Figure \ref{fig:1_espece_om_40_vortex_2}, we compute the index of the giant hole
(see Remark \ref{rem:hole}).
The numerical parameters are indicated in the captions.
As we can see, all the indices of the numerical vortices in Figure \ref{fig:1_espece_om_40_vortex_1}
around the giant hole are equal to one which validates numerically that almost all the zeros
of the wave function have a phase circulation.
In Figure \ref{fig:1_espece_om_40_vortex_2}, the index of the giant hole is equal to $100$.
This is in accordance with the theory presented in section \ref{subsubsec:onecomponent}.

\begin{figure}[ht]
	\centering
	\begin{subfigure}[t]{0.3\textwidth}
	\includegraphics[width=\textwidth]{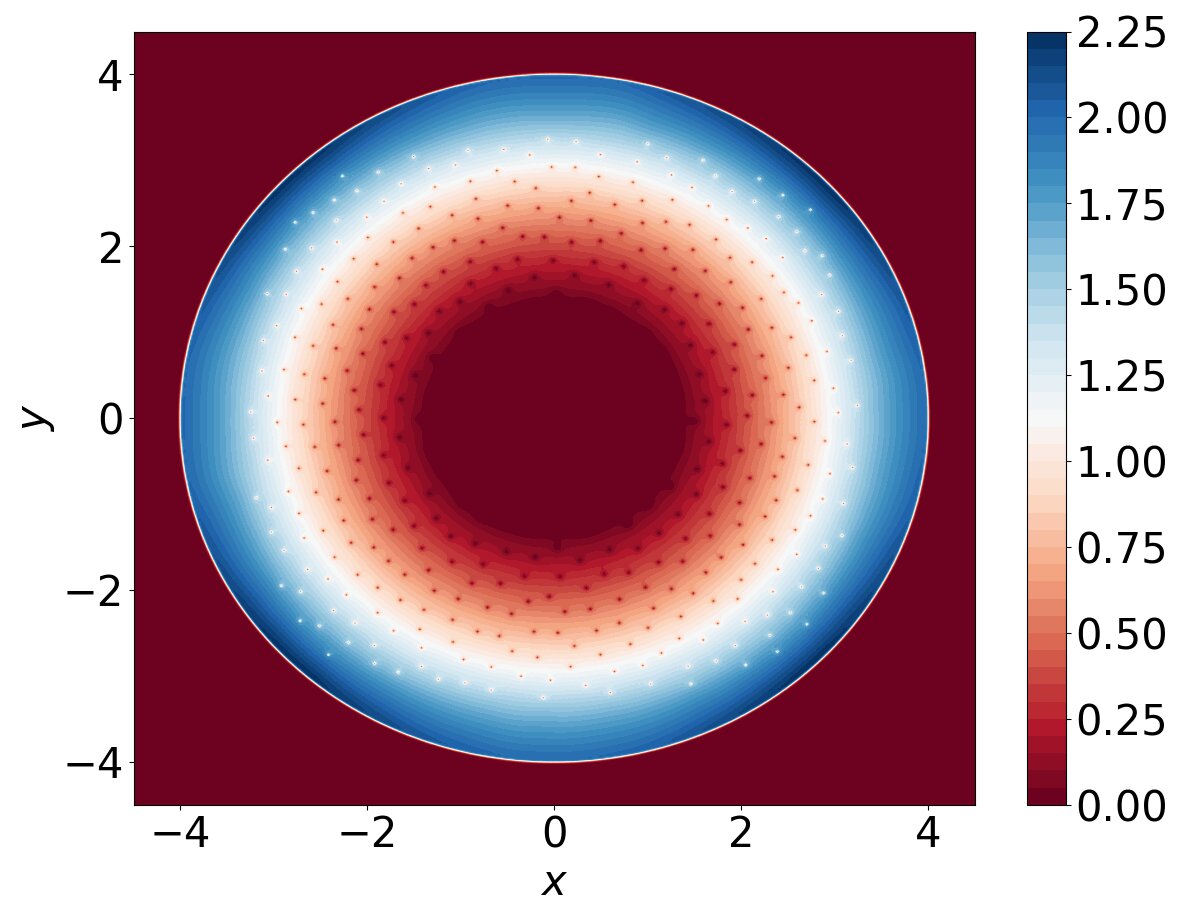}
	\subcaption{The squared modulus of a minimizer.}
	\label{fig:1_espece_om_40}
	\end{subfigure}
	\hfill
	\centering
	\begin{subfigure}[t]{0.3\textwidth}
		\centering
		\includegraphics[width=\textwidth]{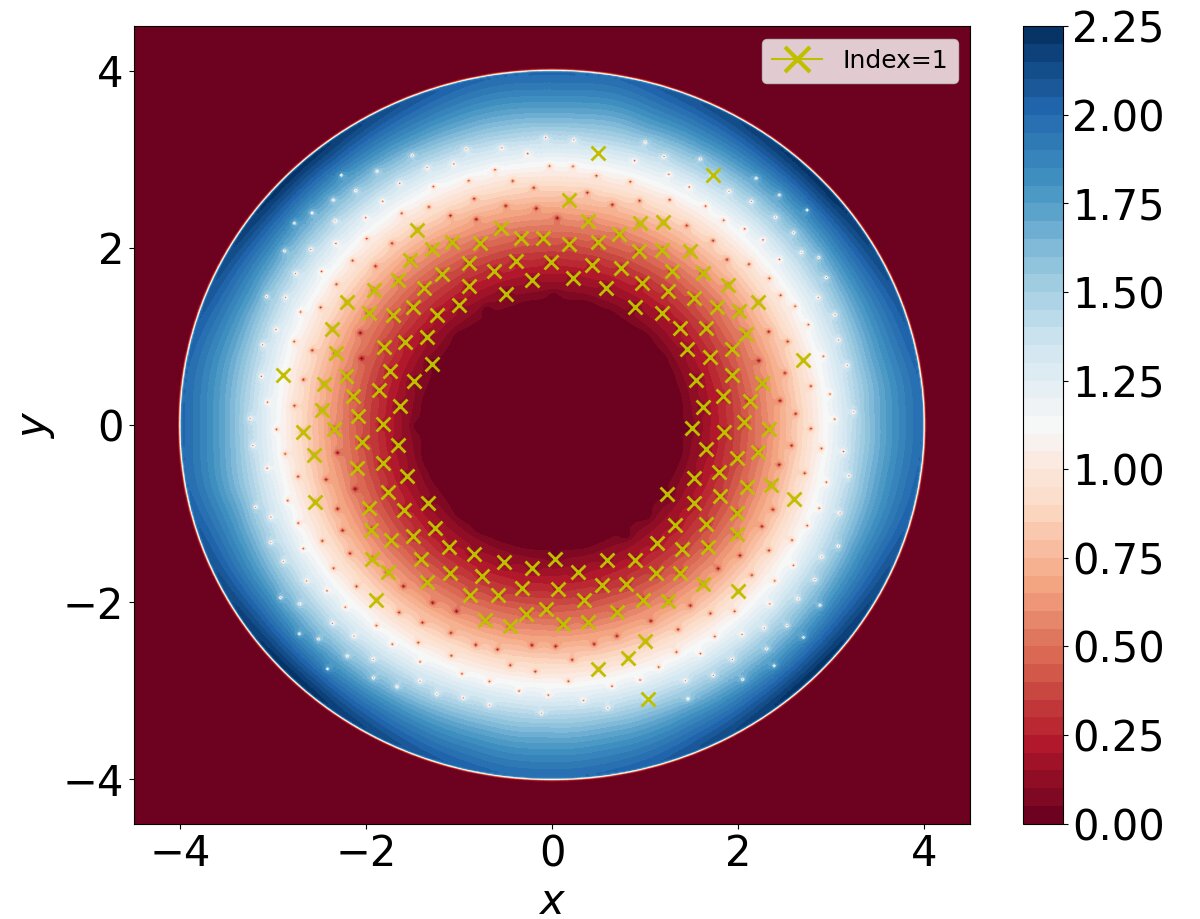}
		\subcaption{The vortices' indices of a minimizer.}
		\label{fig:1_espece_om_40_vortex_1}
	\end{subfigure}
	\hfill
	\begin{subfigure}[t]{0.3\textwidth}
		\centering
		\includegraphics[width=\textwidth]{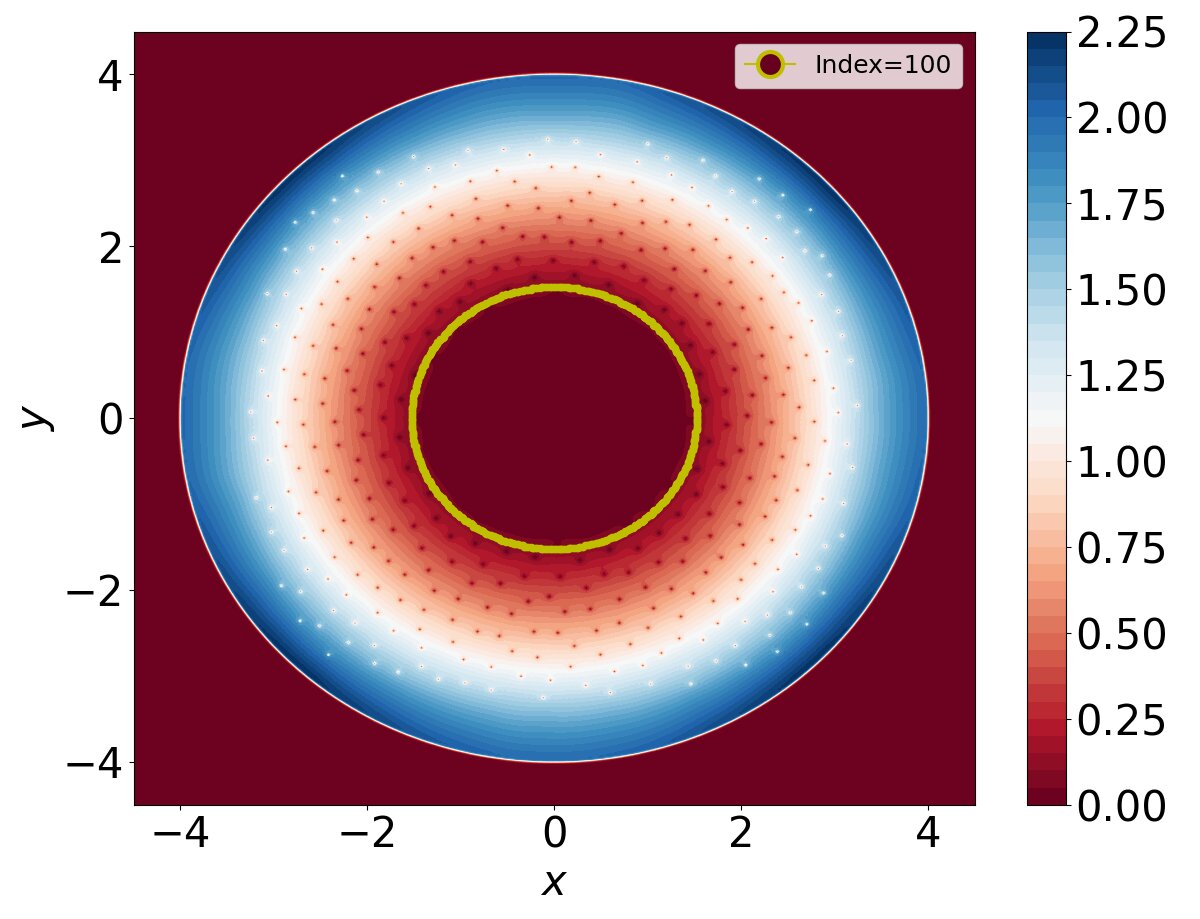}
		\subcaption{The index of the giant hole of a minimizer. 
		}
		\label{fig:1_espece_om_40_vortex_2}
	\end{subfigure}	
        \caption{The squared modulus (A) of a minimizer of the energy for $\Omega=40$
          and $\varepsilon=10^{-2}$.
          The vortices' indices of a minimizer (B) detected with $N_{min}=1$, $N_{max}=3$,
          $tol_1=0.05$ and $tol_2=0.02$.
          The giant hole's index (C) with $r=1.523$ (see remark \ref{rem:hole}).}
	\label{fig:1_espece_om_40_vortex}
\end{figure}

\subsubsection{Huge rotational speed ($\Omega_\varepsilon^3<\Omega_\varepsilon$)}

We consider the case of a huge rotation speed ($\Omega=70$) and strong confinement
($\varepsilon=10^{-2}$).
Once again, we use an interpolation from $N=512$ to $N=1024$.
The EPG algorithm converged due to the stopping criterion $K^\Delta \leq 5\times10^{-2}$.
The results are shown in Figure \ref{fig:1_espece_om_70}.
As we can see, the giant hole in the center is now bigger,
surrounded by less vortices on the annulus of the numerical minimizer than before.
This is due to the centrifugal force coming into play with a very high rotation speed
and is in accordance with the theory presented in section \ref{subsubsec:onecomponent}
(fourth case $\Omega^3_\varepsilon < \Omega$).

In Figures \ref{fig:1_espece_om_70_vortex_1} and \ref{fig:1_espece_om_70_vortex_2},
we compute the indices of the vortices of the numerical minimizer
displayed in Figure \ref{fig:1_espece_om_70}.
First, in Figure \ref{fig:1_espece_om_70_vortex_1}, we compute the indices of the vortices
in the annulus using the algorithm of Section \ref{subsec:index}.
Then, in Figure \ref{fig:1_espece_om_70_vortex_2}, we compute the index of the giant hole
(see Remark \ref{rem:hole}).
As we can see, most of the indices of the numerical vortices
in Figure \ref{fig:1_espece_om_70_vortex_1} around the giant hole are equal to one,
which validates numerically that most of the zeros of the function have a phase circulation.
The index of the giant hole is equal to $522$ (see Figure \ref{fig:1_espece_om_70_vortex_2}).
This is in accordance with the theory presented in section \ref{subsubsec:onecomponent}.

\begin{figure}[ht]
	\centering
	\begin{subfigure}[t]{0.3\textwidth}
	\includegraphics[width=\textwidth]{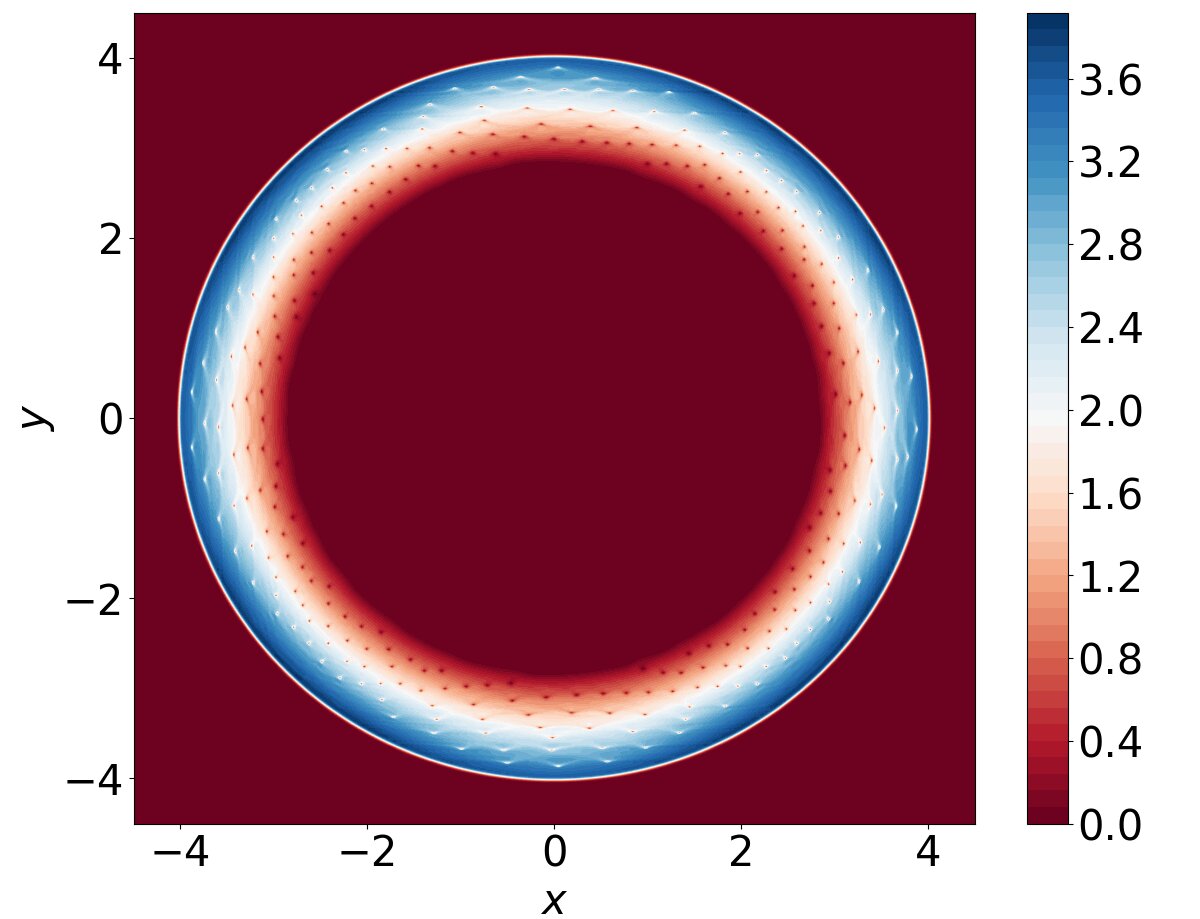}
	\subcaption{The squared modulus of a minimizer.}
	\label{fig:1_espece_om_70}
	\end{subfigure}
	\hfill
	\begin{subfigure}[t]{0.3\textwidth}
		\centering
		\includegraphics[width=\textwidth]{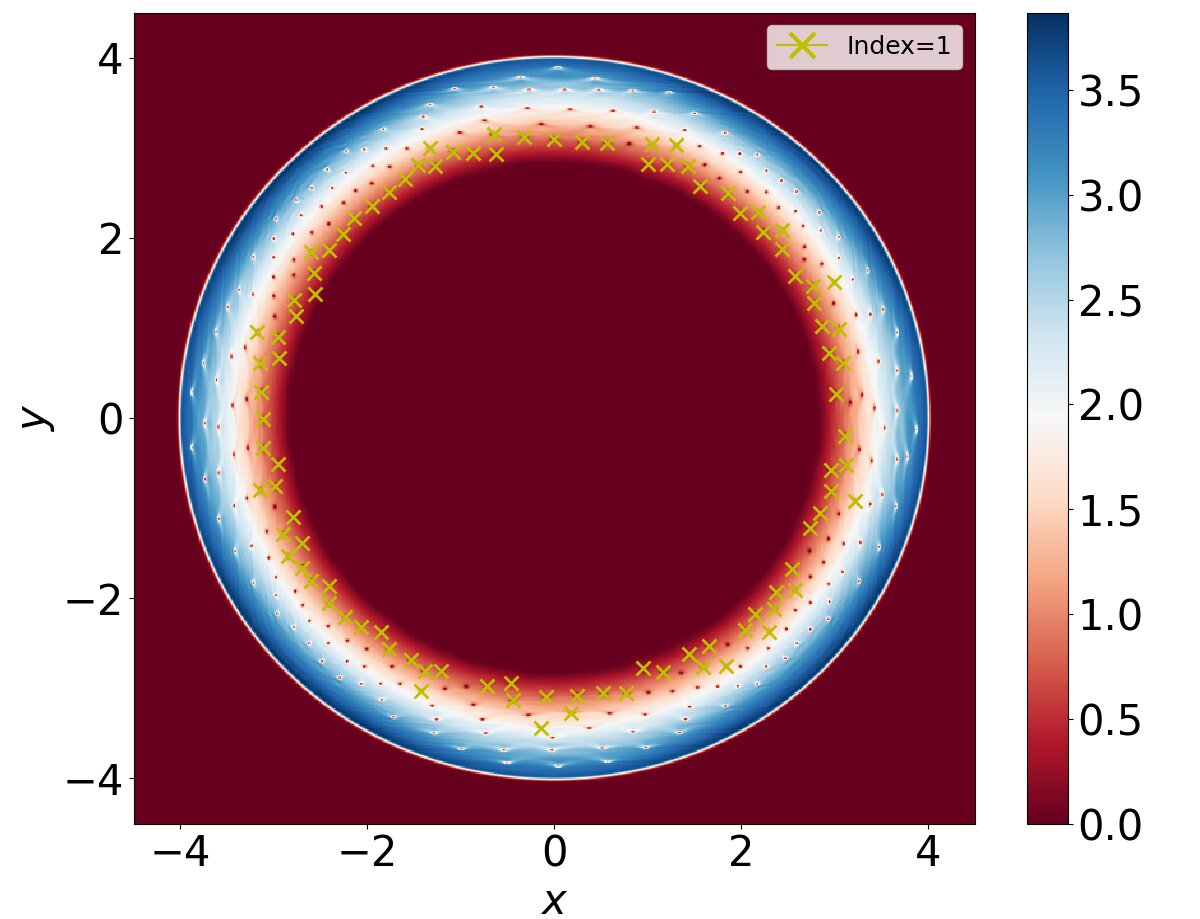}
		\subcaption{The vortices' indices of a minimizer.}
		\label{fig:1_espece_om_70_vortex_1}
	\end{subfigure}
	\hfill
	\begin{subfigure}[t]{0.3\textwidth}
		\centering
		\includegraphics[width=\textwidth]{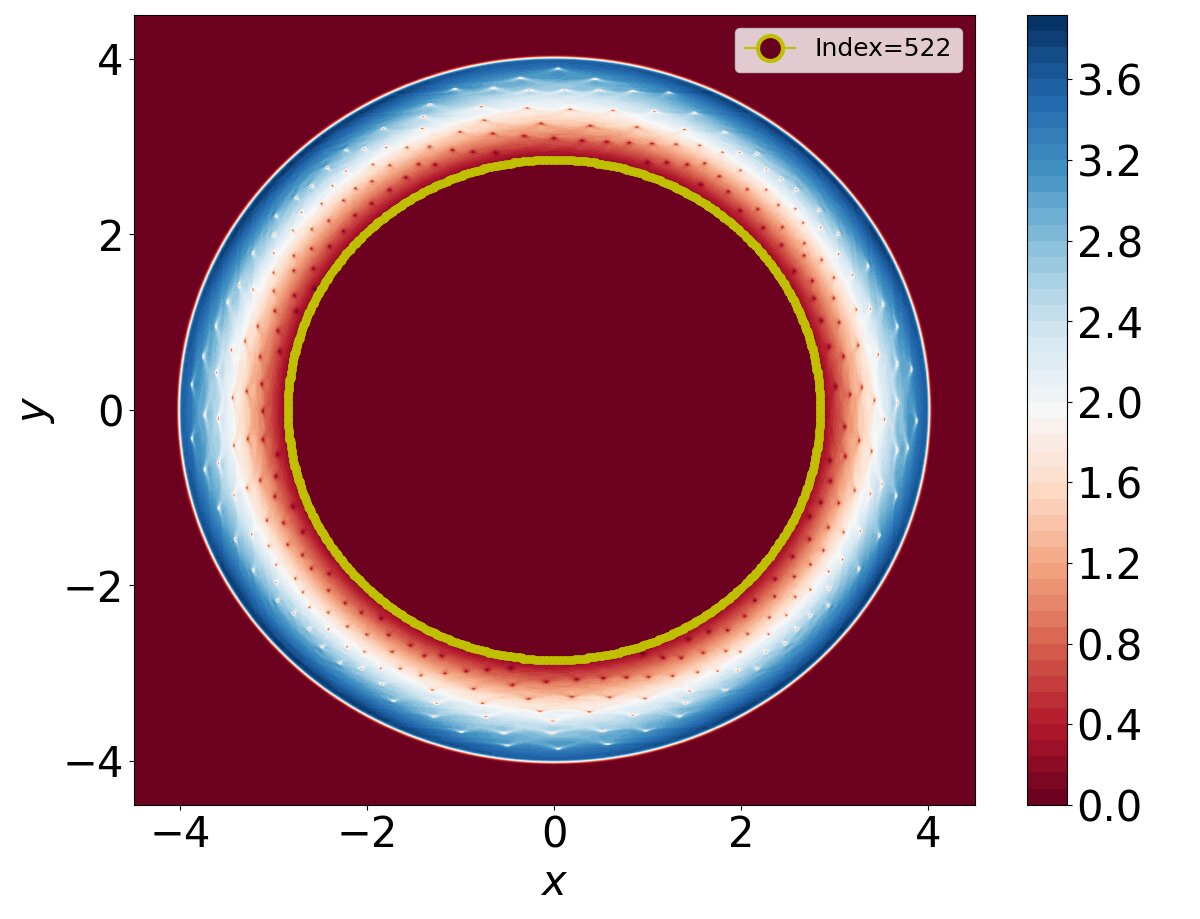}
		\subcaption{The index of the giant hole.}
		\label{fig:1_espece_om_70_vortex_2}
	\end{subfigure}	
	\caption{The squared modulus (A) of a minimizer of the energy for $\Omega=70$,
          $\varepsilon=10^{-2}$.
          The vortices' indices of a minimizer of the energy (B) detected with
          $N_{min}=1$, $N_{max}=3$, $tol_1=0.05$ and $tol_2=0.02$.
          The index of the giant hole (C) detected with $r=2.85$ (see Remark \ref{rem:hole}).}
	\label{fig:1_espece_om_70_vortex}
\end{figure}

\subsection{Two components condensate without rotation ($\Omega=0$)}
\label{subsec:simulation_twocomponent}

In this section, we move on the theoretical results described in the first part of Section
\ref{subsubsec:twocomponent} that deal with two components Bose--Einstein condensates,
without rotation, in different (strong, moderate and weak) segregation ($\delta>1$) regimes.

\subsubsection{Common numerical parameters}
\label{subsubsec:param_2species_norotation}

In this section, we use the following parameters.
The function $\rho$ is the same as in \eqref{eq:defrho}.
The physical parameters are $\varepsilon=5\times10^{-2}$, $N_1=0.55$ and $N_2=0.45$.
The initial data are given by $\psi^1=\psi^2=\exp({-10x^2-10y^2})/5$.
The discretization parameters are $L=7$, $R=4$, $N=256$, $h=0.1$ and $h_0=10^{-12}$.
The stopping criterion value for $K^\Delta$ is set to $10^{-2}$.
		
Note that for the simulations of Figures \ref{fig:2_espece_delta=4000}
and \ref{fig:2_espece_delta=1,02}),
we started with the minimizers obtained in Figure \ref{fig:2_espece_delta=1,5} as initial data.  

\subsubsection{The strong segregation regime: $\delta_\varepsilon\times\varepsilon^2 \to +\infty$ }
\label{subsubsec:strongconfin_2species}

We consider the case of strong segregation regime ($\delta_\varepsilon=4000$)
and strong confinement so that $\delta_\varepsilon \times\varepsilon^2=10$.
The results are shown in Figure \ref{fig:2_espece_delta=4000}.
As we can see in Figure \ref{fig:2_espece_delta=4000_1} and \ref{fig:2_espece_delta=4000_2},
the numerical support of the two components tend to {\it not} overlap,
thereby confirming that we are in the segregation regime.
In Figure \ref{fig:2_espece_delta=4000_12}, we can see that the sum of the squared modulus
of the two components has a minimum inside the disc (away from the boundary of the disc),
with $\min_{(x,y) \in D}|\psi^{1*}|^2+|\psi^{2*}|^2 \approx 0.14$.
This is in accordance with the theory presented in section \ref{subsubsec:twocomponent}.

\begin{figure}[ht]
	\centering
	\begin{subfigure}[t]{0.3\textwidth}
		\centering
		\includegraphics[width=\textwidth]{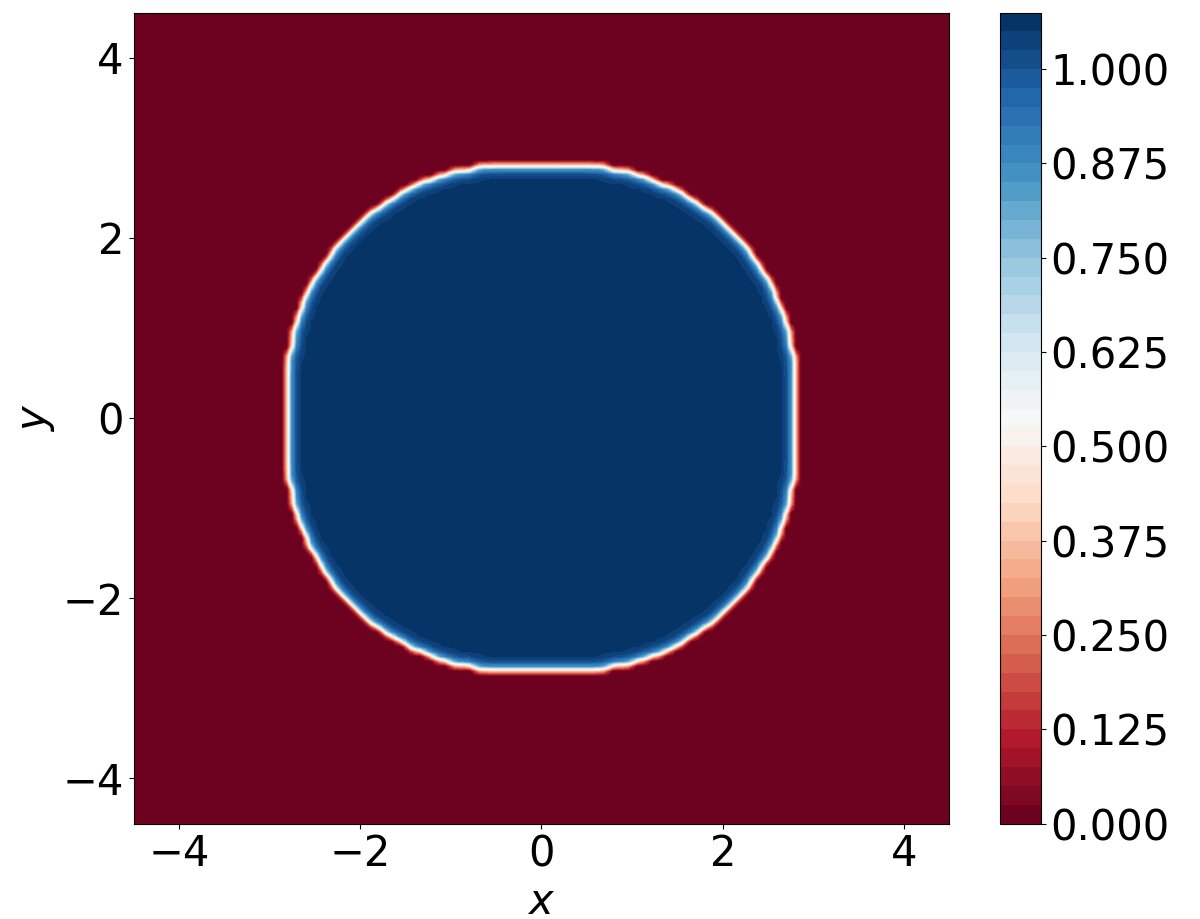}
		\subcaption{The squared modulus of the first component.}
		\label{fig:2_espece_delta=4000_1}
	\end{subfigure}
	\hfill
	\begin{subfigure}[t]{0.3\textwidth}
		\centering
		\includegraphics[width=\textwidth]{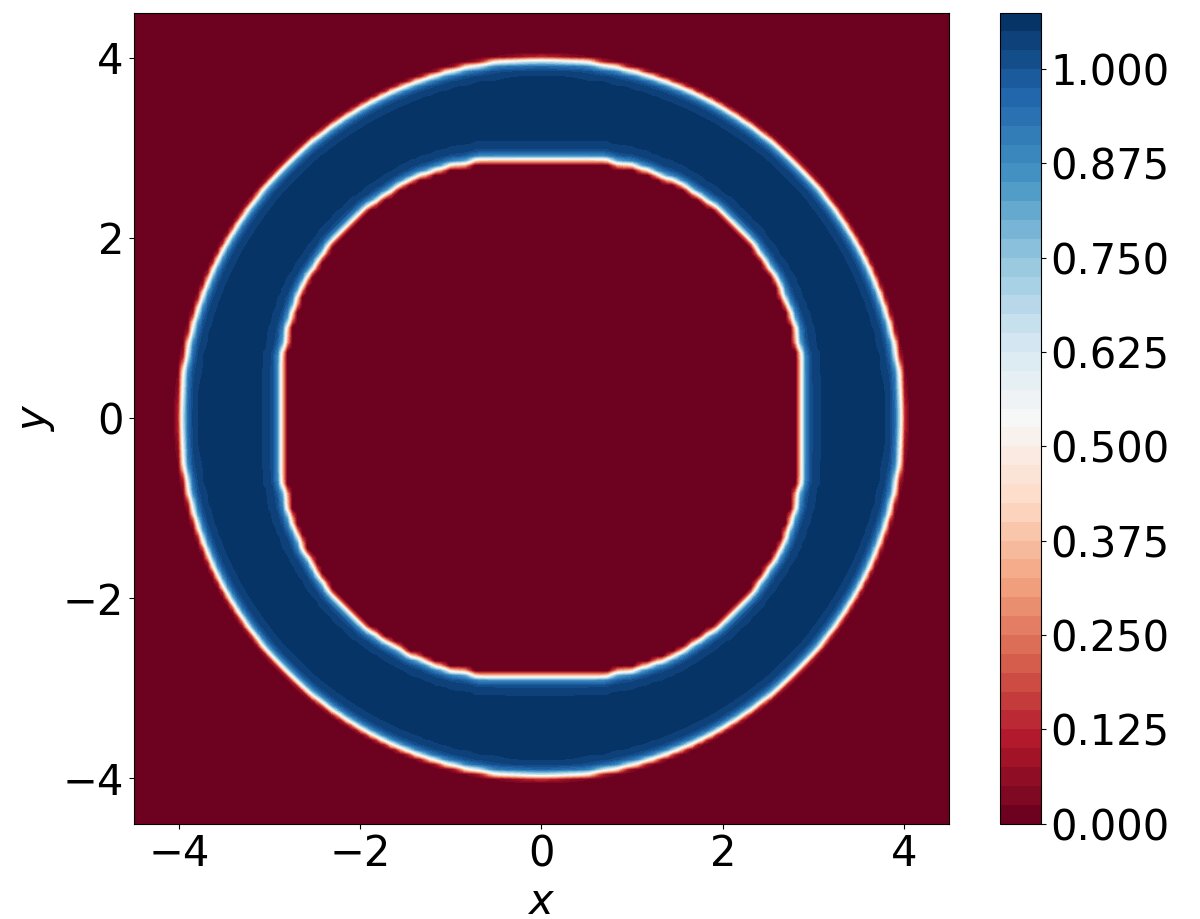}
		\subcaption{The squared modulus of the second component.}
		\label{fig:2_espece_delta=4000_2}
	\end{subfigure}
	\hfill
	\begin{subfigure}[t]{0.3\textwidth}
		\centering
		\includegraphics[width=\textwidth]{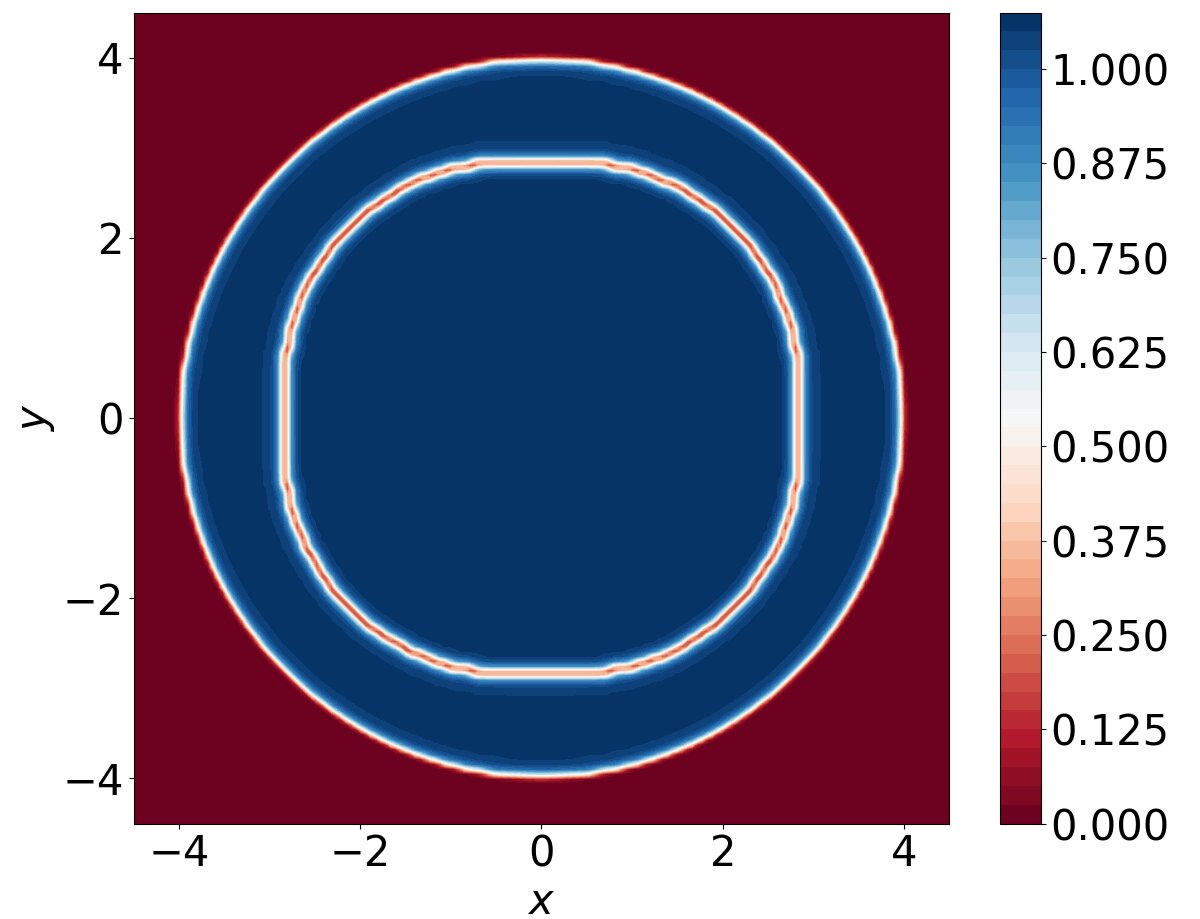}
		\subcaption{The sum of the squared moduli of both components.}
		\label{fig:2_espece_delta=4000_12}
	\end{subfigure}
	\caption{The squared modulus of the first component (A) and
          the second component (B) for a minimizer in the case of two components condensate with
          no rotation and $\delta_\varepsilon=4000$. Fig (C) displays their sum.}
	\label{fig:2_espece_delta=4000}
\end{figure}

\subsubsection{The moderate segregation regime: $\delta>1$ fixed and small $\varepsilon$}
\label{subsubsec:moderatesegregation_2species}      

We consider the case of moderate segregation regime ($\delta=1.5$) and strong confinement,
still without rotation ($\Omega=0$).
The results are shown in Figure \ref{fig:2_espece_delta=1,5}.
The results displayed in Figures \ref{fig:2_espece_delta=1,5_1} and \ref{fig:2_espece_delta=1,5_2}
confirm that we are in the segregation regime ($\delta>1$).
Figure \ref{fig:2_espece_delta=1,5_12} presents the sum of the two squared moduli
of the two components of the minimizer.
We can see a white curve corresponding to the separation between the two components
with $\min_{(x,y) \in D}|\psi^{1*}|^2+|\psi^{2*}|^2 \approx 0.78$.
This is in accordance with the theory presented in section \ref{subsubsec:twocomponent}
(second case with no rotation and fixed $\delta>1$).

\begin{figure}[ht]
	\centering
	\begin{subfigure}[t]{0.3\textwidth}
		\centering
		\includegraphics[width=\textwidth]{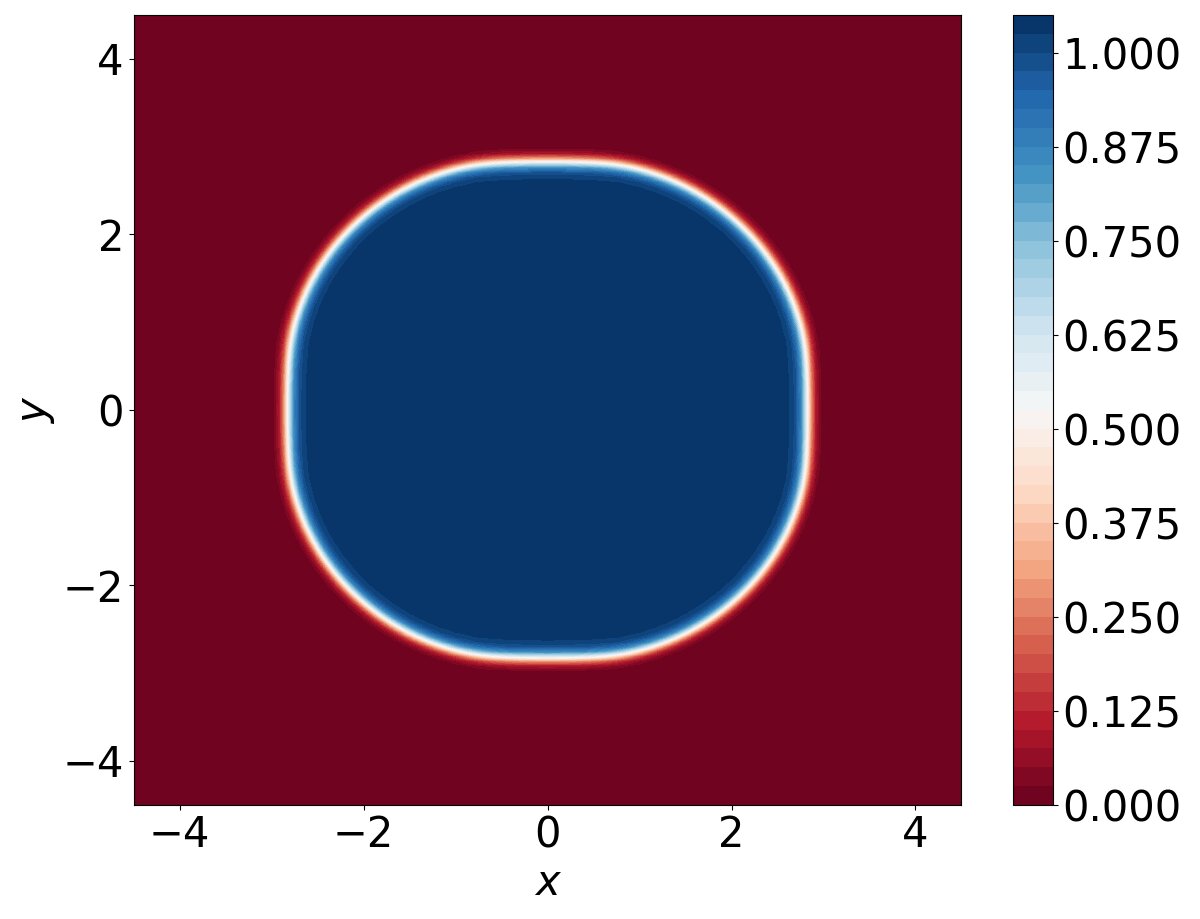}
		\subcaption{The squared modulus of the first component.}
		\label{fig:2_espece_delta=1,5_1}
	\end{subfigure}
	\hfill
	\begin{subfigure}[t]{0.3\textwidth}
		\centering
		\includegraphics[width=\textwidth]{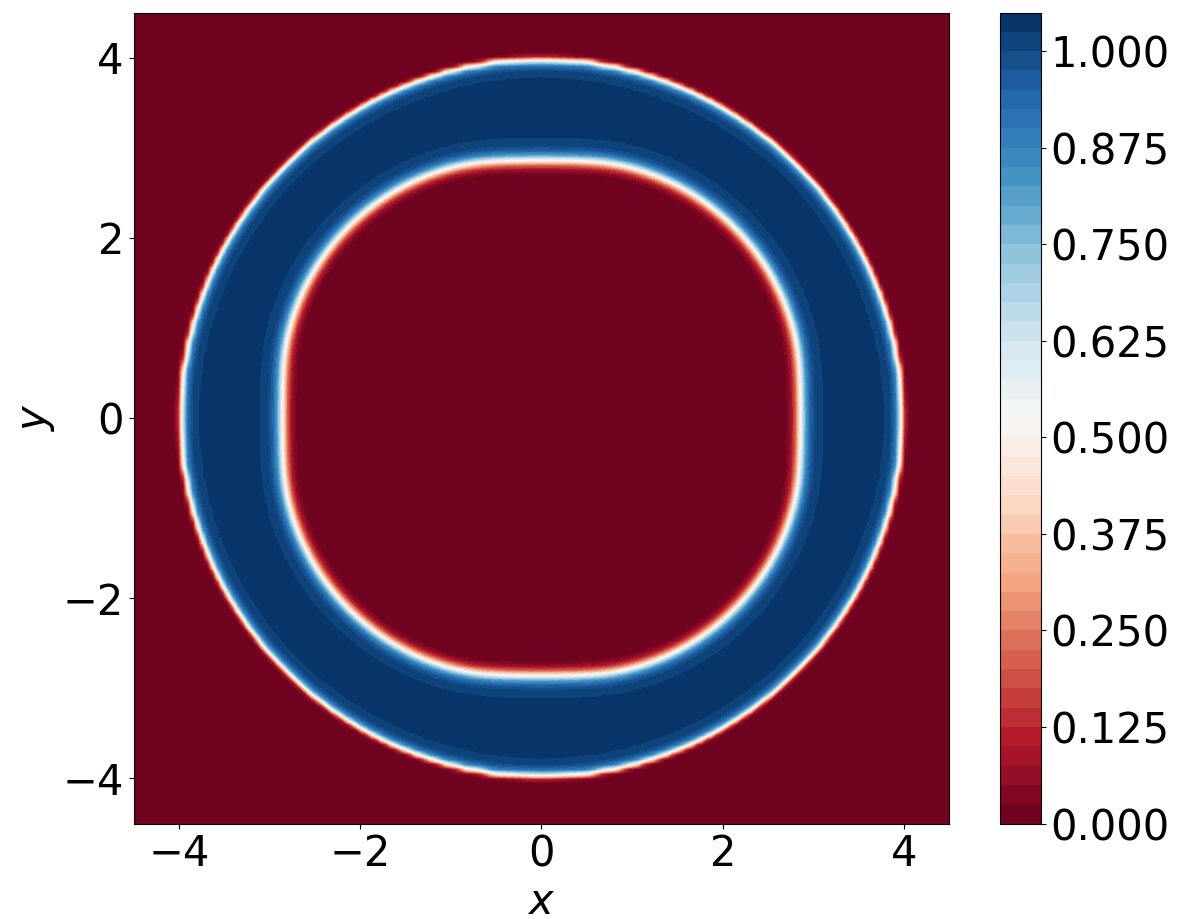}
		\subcaption{The squared modulus of the second component.}
		\label{fig:2_espece_delta=1,5_2}
	\end{subfigure}
	\hfill
	\begin{subfigure}[t]{0.3\textwidth}
		\centering
		\includegraphics[width=\textwidth]{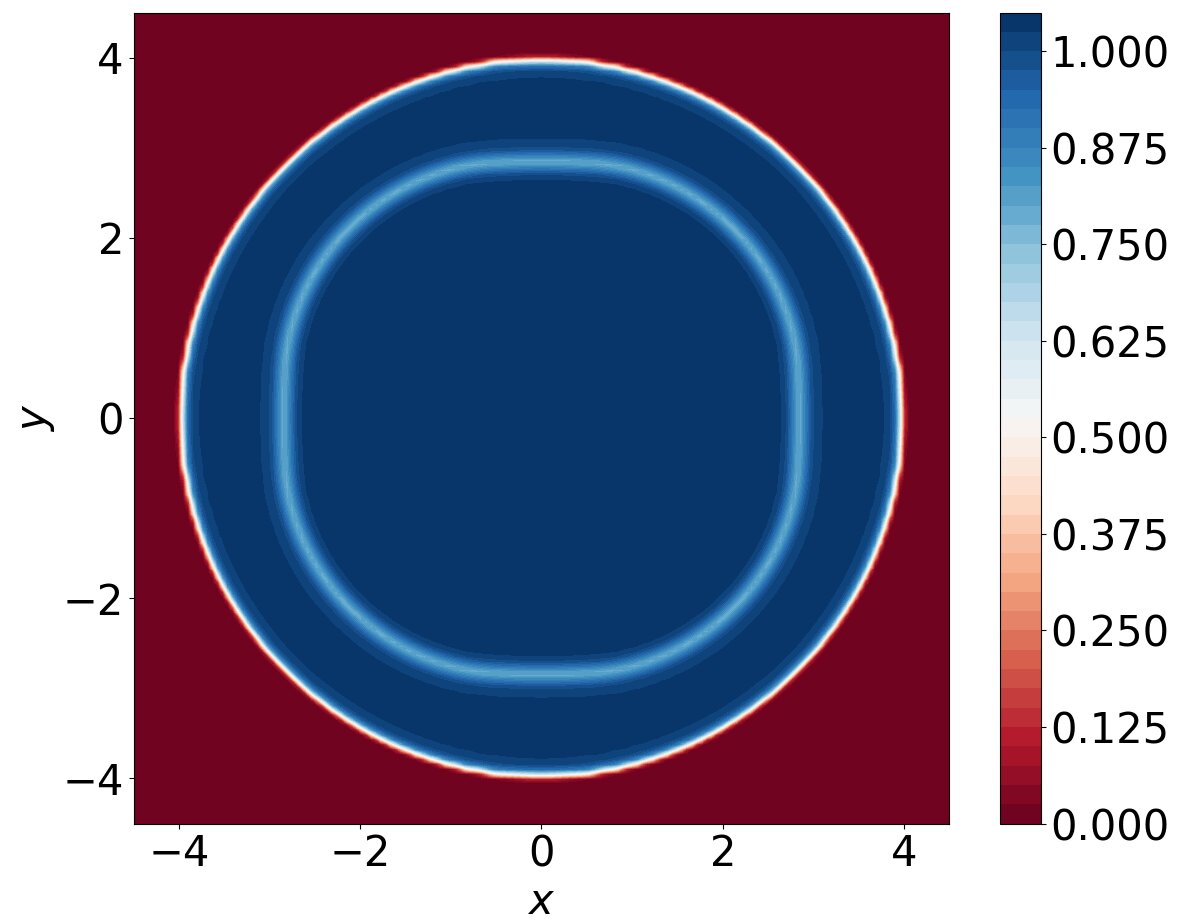}
		\subcaption{The sum of the squared moduli of both components.}
		\label{fig:2_espece_delta=1,5_12}
	\end{subfigure}
	
	\caption{The squared modulus of the first component (A) and
          of the second component (B) for a minimizer in the case of two components condensate
          with no rotation and $\delta=1.5$.
          Their sum is displayed in Fig. (C).}
	\label{fig:2_espece_delta=1,5}
\end{figure}

\subsubsection{The weak segregation regime: $\delta_\varepsilon \to 1$ with
  $\varepsilon/\sqrt{\delta_\varepsilon-1} \to 0$}
\label{subsubsec:weaksegregation_2species}        

We consider the case of segregation regime ($\delta_\varepsilon=1.02$) and strong confinement,
still without rotation ($\Omega=0$).
The results are shown in Figure \ref{fig:2_espece_delta=1,02}.
Figures \ref{fig:2_espece_delta=1,02_1} and \ref{fig:2_espece_delta=1,02_2} confirm
that we are still in the segregation regime.
In contrast to the two previous cases, we can see in figure \ref{fig:2_espece_delta=1,02_12}
that the sum of the two squared moduli does {\it not} present a seperation area between the components.
Indeed, apart from close to the boundary of the disc, the sum of the squared moduli
is almost constant, with an approximate value of $0.976$.
With the notations of Section \ref{subsubsec:twocomponent}, we have $\tilde{\varepsilon}=0.35$
in this third simulation.
The results are in accordance with the theory
(third case with no rotation, $\delta_\varepsilon$ close to $1$ and $\tilde{\varepsilon}$
close to $0$ in Section \ref{subsubsec:twocomponent}).

\begin{figure}[ht]
	\centering
	\begin{subfigure}[t]{0.3\textwidth}
		\centering
		\includegraphics[width=\textwidth]{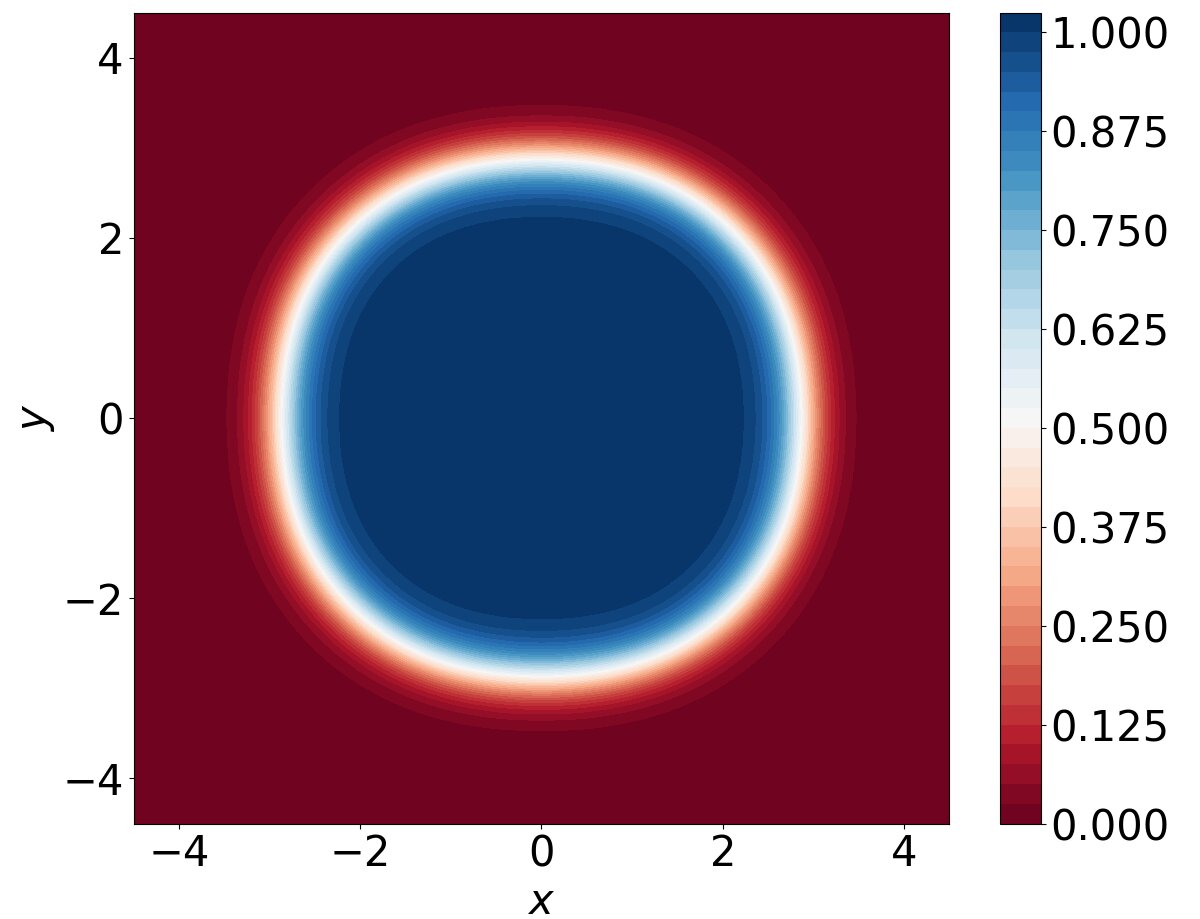}
		\subcaption{The squared modulus of the first component.}
		\label{fig:2_espece_delta=1,02_1}
	\end{subfigure}
	\hfill
	\begin{subfigure}[t]{0.3\textwidth}
		\centering
		\includegraphics[width=\textwidth]{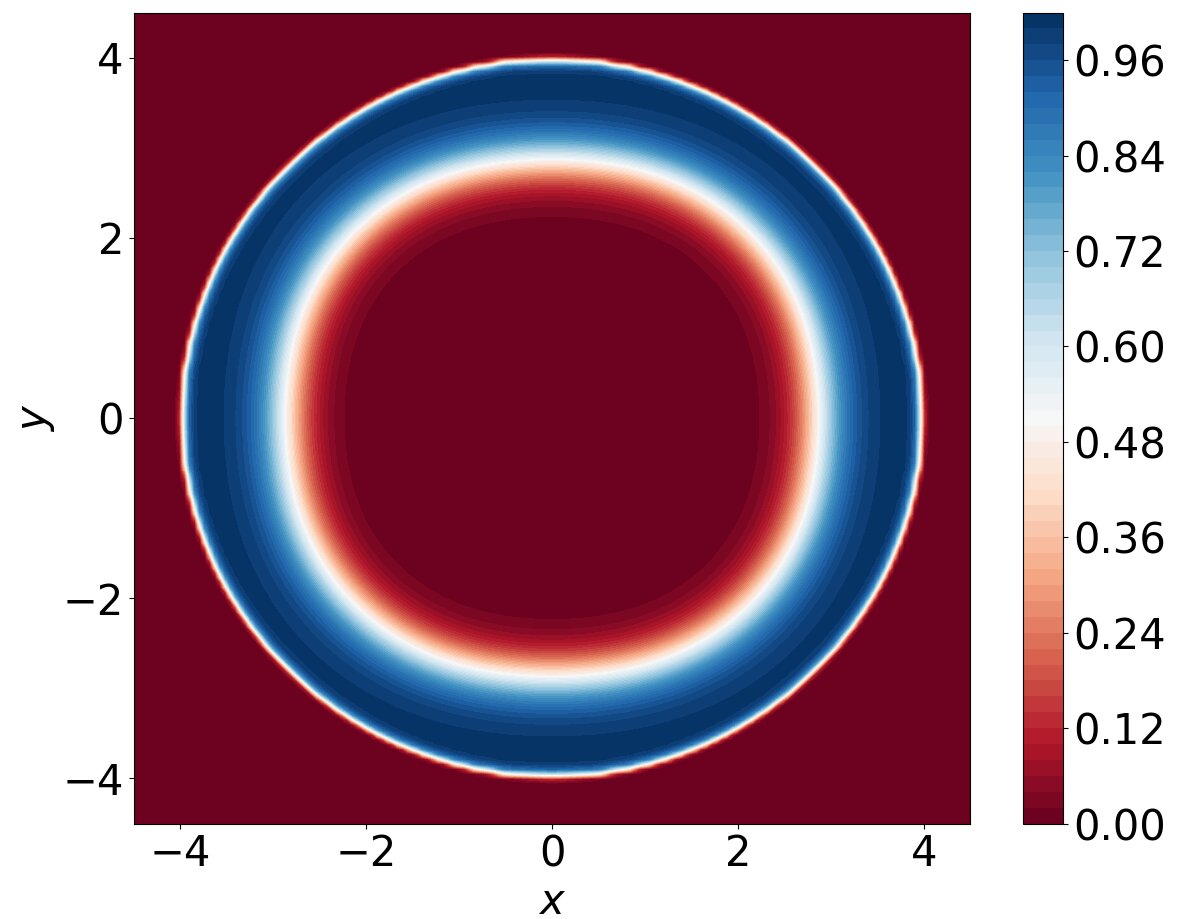}
		\subcaption{The squared modulus of the second component.}
		\label{fig:2_espece_delta=1,02_2}
	\end{subfigure}
	\hfill
	\begin{subfigure}[t]{0.3\textwidth}
		\centering
		\includegraphics[width=\textwidth]{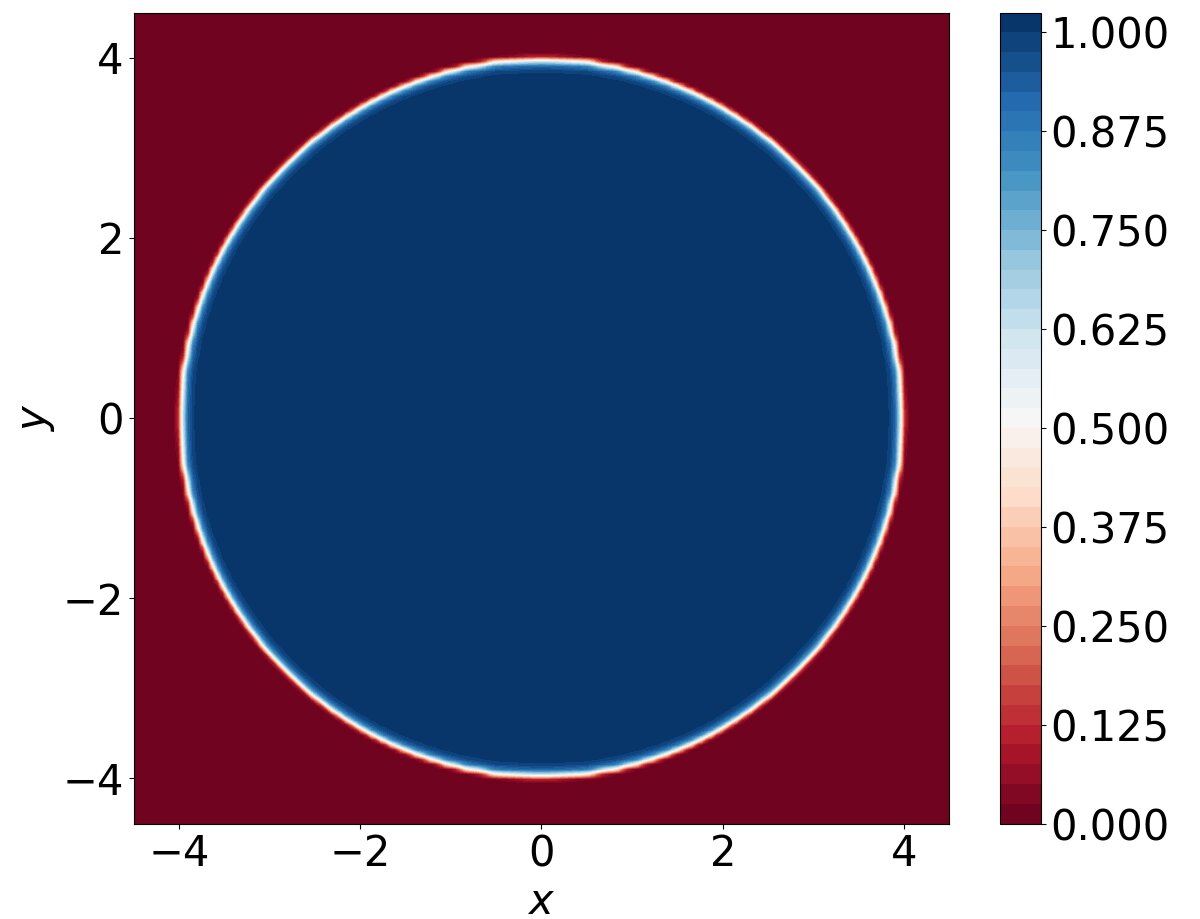}
		\subcaption{The sum of the squared moduli of both components.}
		\label{fig:2_espece_delta=1,02_12}
	\end{subfigure}
	
	\caption{The squared modulus of the first component (A)
          and the second component (B) for a minimizer in the case of two components condensate
          with no rotation and $\delta_\varepsilon=1.02$.
          Their sum is displayed in (C).}
	\label{fig:2_espece_delta=1,02}
\end{figure}

\subsection{Two components condensate with rotation ($\Omega\neq0$) in the segregation
  regime}
\label{subsec:twocompseg}

In this section, we consider physical parameters related to the theoretical results
described in the second part of Section \ref{subsubsec:twocomponent}, that deal with two components
Bose--Einstein condensates, in the segregation regime ($\delta>1$), {\it with} rotation.

\subsubsection{Common parameters}

In this section, we use the following parameters.
The function $\rho$ is the same as in \eqref{eq:defrho}.
The physical parameters are $\varepsilon=10^{-2}$,
$\delta_\varepsilon=1+\varepsilon=1.01$ (so that $\tilde{\varepsilon}=\sqrt{\varepsilon}=0.1$),
$N_1=0.55$ and $N_2=0.45$.
The initial data are given by $\psi^1=\psi^2=\exp({-10x^2-10y^2})/5$.
The discretization parameters are $L=7$, $R=4$, $h=0.1$ and $h_0=10^{-12}$.
All the numerical results in this Section \ref{subsec:twocompseg} are obtained by first
minimizing the discrete energy with $N=256$, then interpolating the real and imaginary parts
of each component to a grid of $N=512$ points in each direction, then minimizing the corresponding
discrete energy.
For all the experiments in this section, the EPG algorithm of Section \ref{subsec:gradient}
converged because of the stopping criterion on $K^\Delta$, with a value less or equal to
$2.0\times 10^{-2}$.
		
These parameters correspond to the segregation regime ($\delta$ greater yet close to $1$)
and $\varepsilon << \tilde{\varepsilon}$.
We consider the cases $\Omega\in\{1, 3, 6, 15\}$.

\subsubsection{Low rotation case: $\Omega=1$}

We first consider the case of low rotation and strong confinement.
The results are shown in Figure \ref{fig:2_espece_om_1_seg}.
The numerical experiment confirms that we are in a segregation regime (since $\delta>1$)
and the two components of the minimizer tend to {\it not} overlap.
Moreover, the small rotation speed ($\Omega=1$) is not big enough to produce vortices
in the minimizer.
This is in accordance with the theory presented in section \ref{subsubsec:twocomponent} (first bullet point in the rotational case).

	\begin{figure}[ht]
          \centering
	\begin{subfigure}[t]{0.42\textwidth}
		\centering
		\includegraphics[width=.8\textwidth]{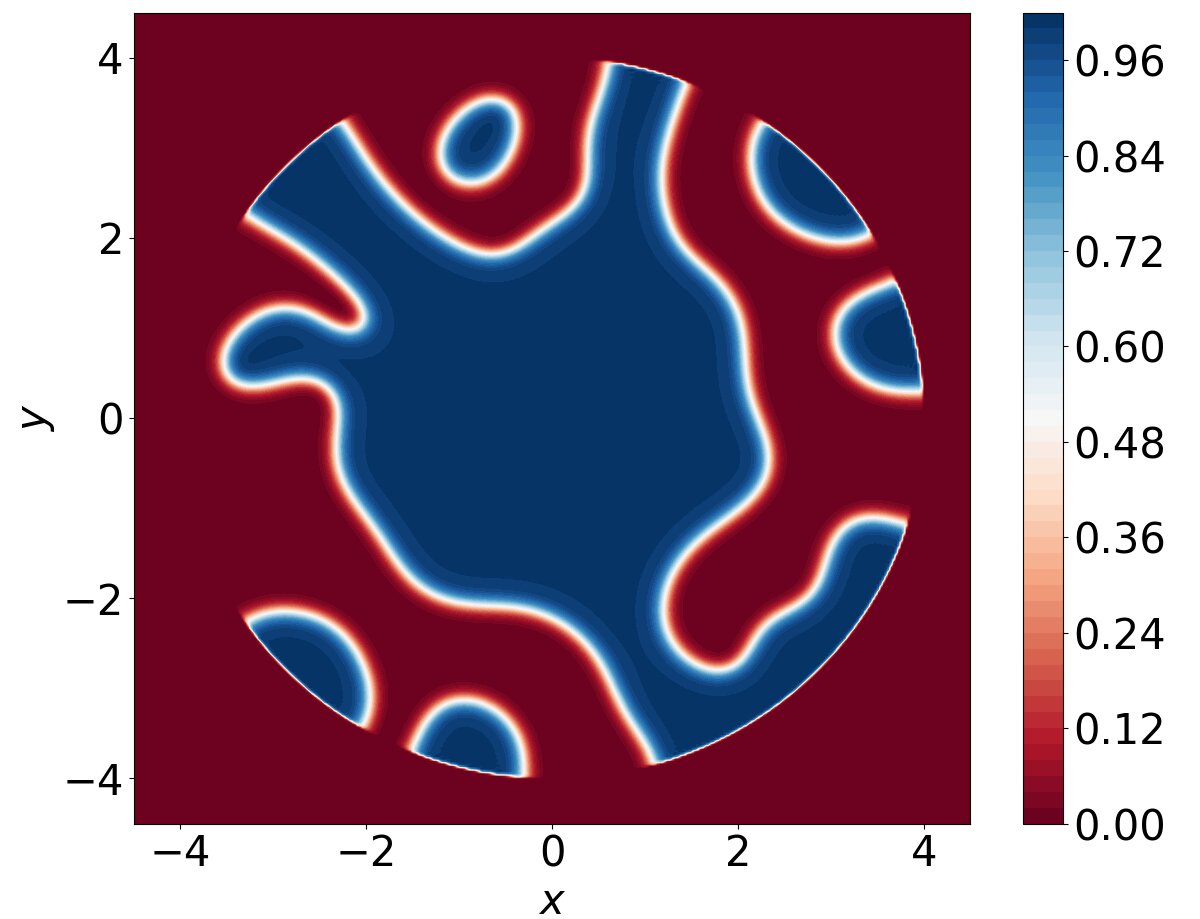}
		\subcaption{The squared modulus of the first component.}
		\label{fig:2_espece_om_1_1_seg}
	\end{subfigure}
	\hfill
	\begin{subfigure}[t]{0.42\textwidth}
		\centering
		\includegraphics[width=.8\textwidth]{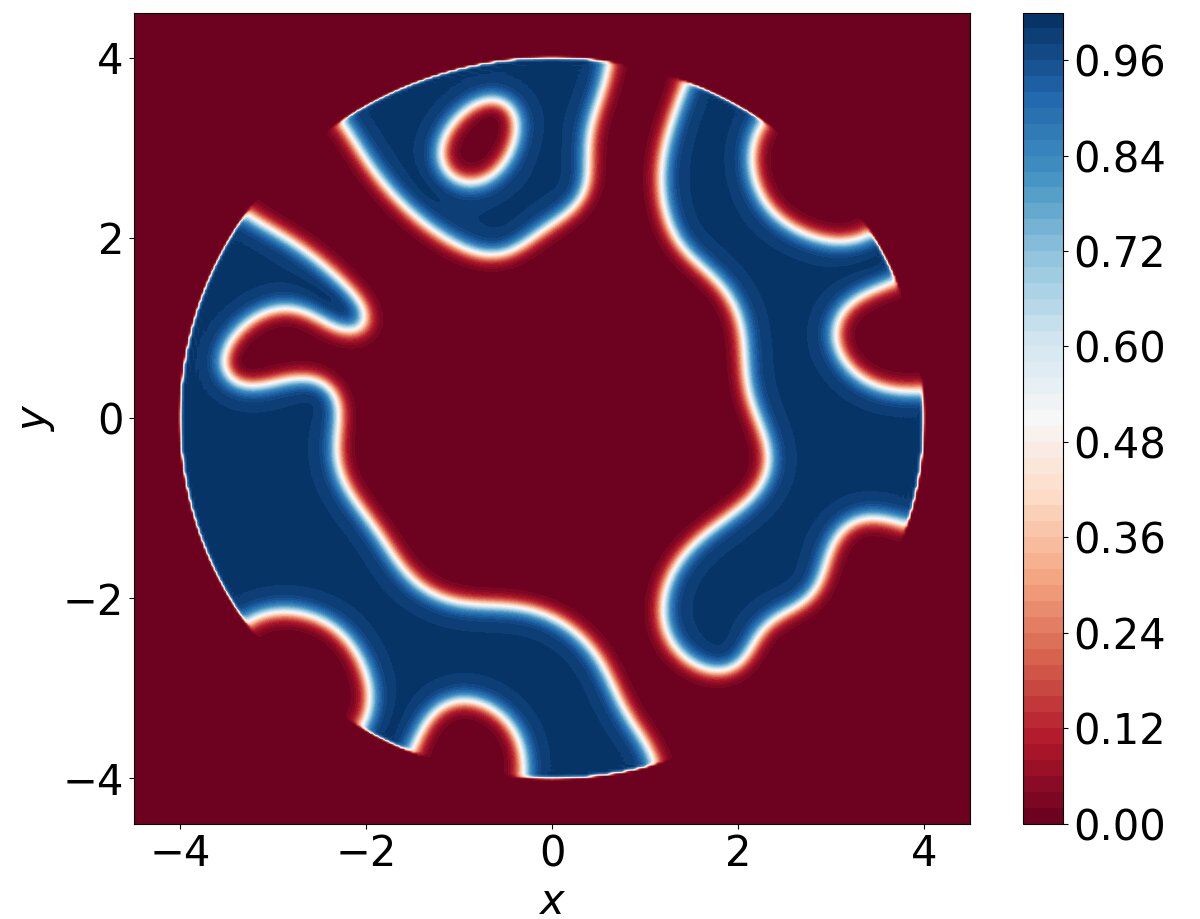}
		\subcaption{The squared modulus of the second component.}
		\label{fig:2_espece_om_1_2_seg}
        \end{subfigure}
	
	\caption{The squared modulus of the first component (A) and the second component (B)
          for a minimizer in the case of two components condensate with low rotation $\Omega=1$
          in the segregation regime.}
	\label{fig:2_espece_om_1_seg}
\end{figure}

\subsubsection{Moderate rotation case: $\Omega=3$}
We consider the case of moderate rotation speed and strong confinement.
The results are shown in Figure \ref{fig:2_espece_om_3_seg}.
The numerical experiment shows that segregation holds, and the moderate rotation speed
is big enough to produce singly quantized vortices in one of the components of the numerical minimizer
(see Fig. \ref{fig:2_espece_om_3_2_vortex}).
This is in accordance with the theory presented in section \ref{subsubsec:twocomponent}
(second bullet point of the rotational case).

	\begin{figure}[ht]
	\centering
	\begin{subfigure}[t]{0.3\textwidth}
		\centering
		\includegraphics[width=\textwidth]{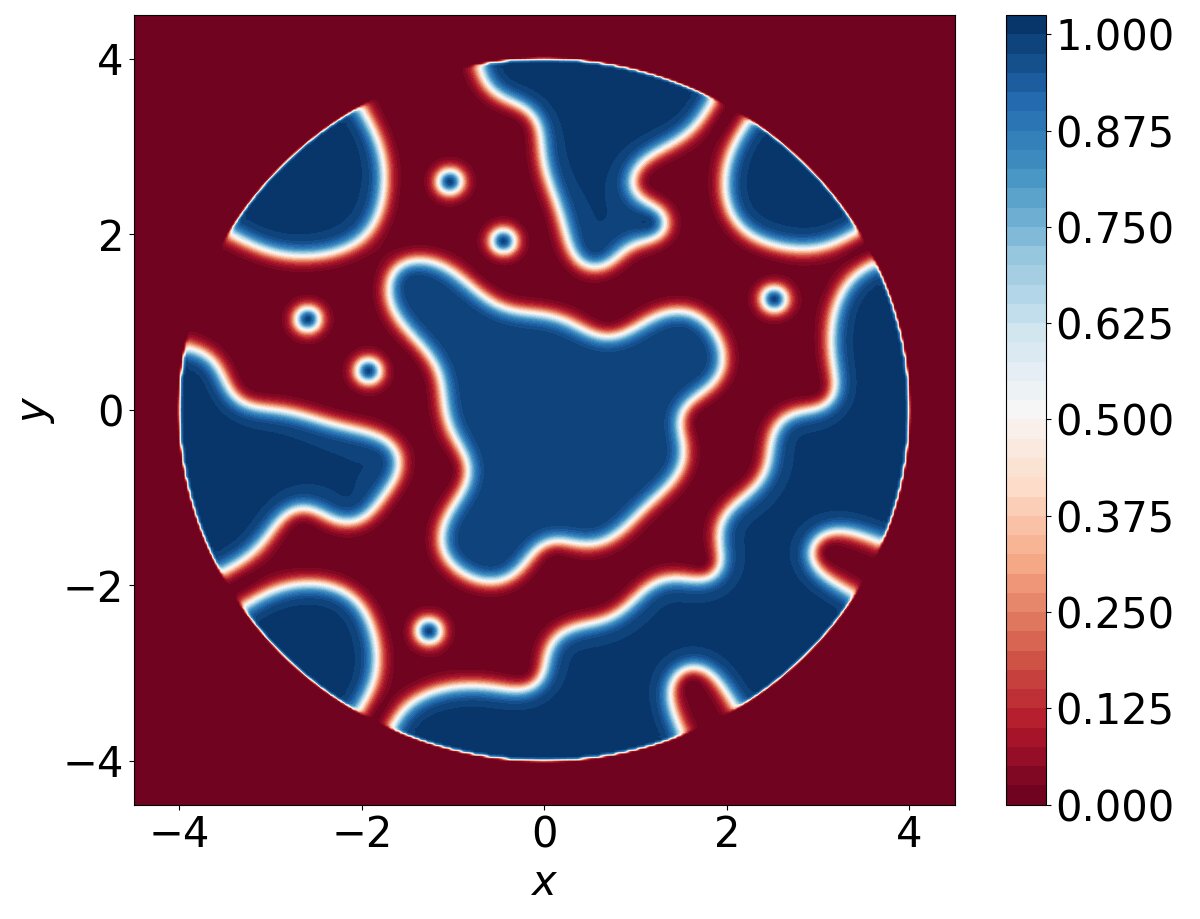}
		\subcaption{The squared modulus of the first component.}
		\label{fig:2_espece_om_3_1_seg}
	\end{subfigure}
	\hfill
	\begin{subfigure}[t]{0.3\textwidth}
		\centering
		\includegraphics[width=\textwidth]{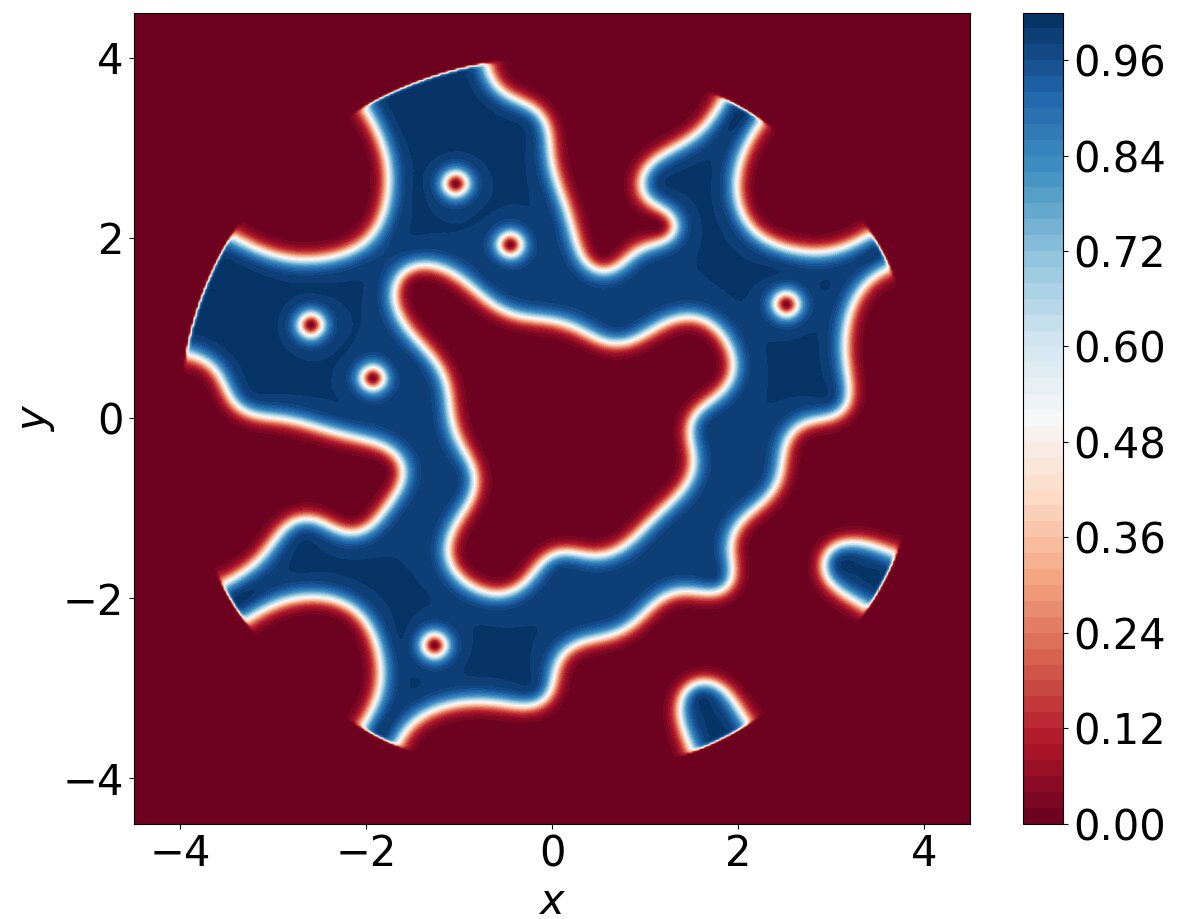}
		\subcaption{The squared modulus of the second component.}
		\label{fig:2_espece_om_3_2_seg}
	\end{subfigure}
	\hfill
	\begin{subfigure}[t]{0.3\textwidth}
		\centering
		\includegraphics[width=\textwidth]{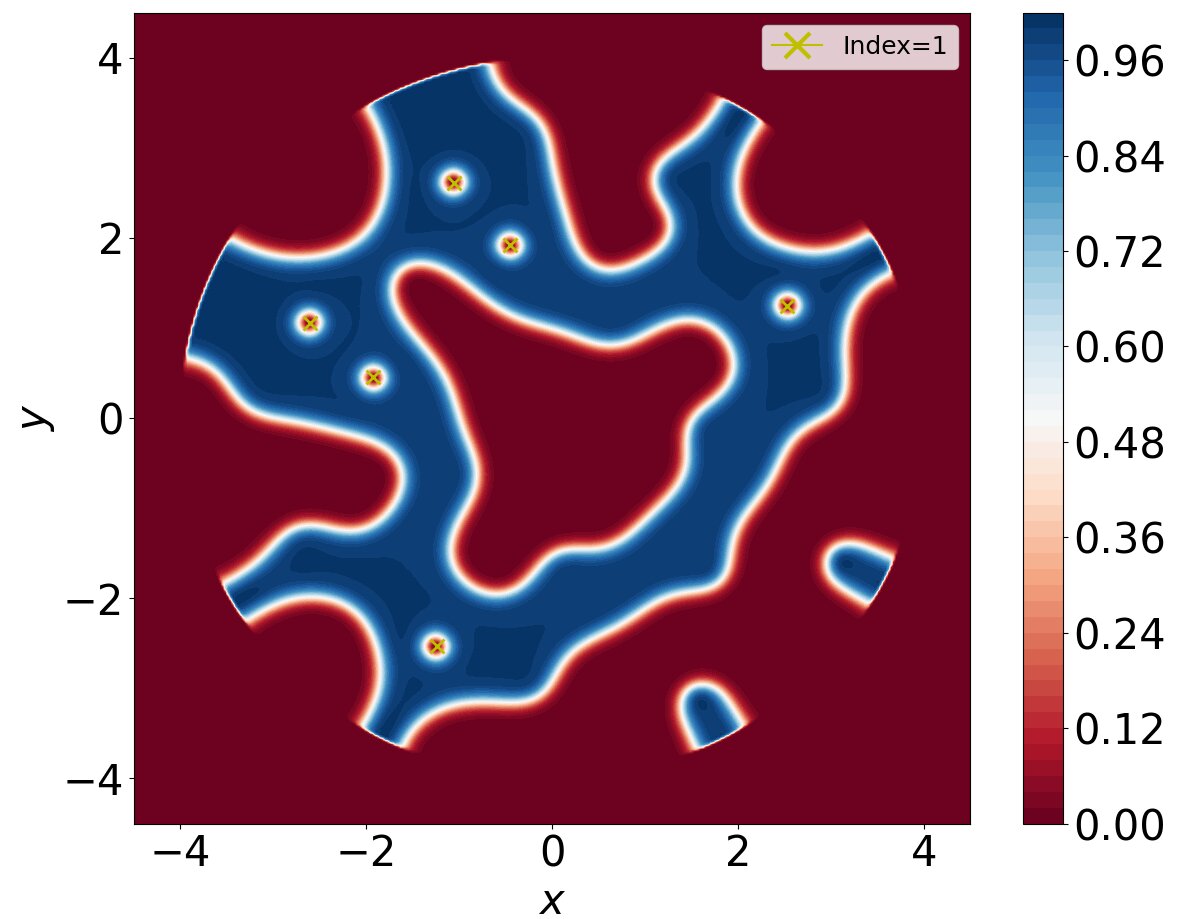}
		\subcaption{The vortices' indices of the second component
                  of a minimizer for the energy detected with
                  $N_{min}=1$, $N_{max}=5$, $tol_1=0.05$ and $tol_2=0.02$.}
		\label{fig:2_espece_om_3_2_vortex}
	\end{subfigure}
	\caption{The squared modulus of the first component (A) and the second component (B)
          of a minimizer in the case of two components condensate with moderate rotation
          $\Omega=3$.
          The vortices' indices of the second component are presented in Fig. (C).}
	\label{fig:2_espece_om_3_seg}
\end{figure}  

\subsubsection{Average rotation case: $\Omega=6$}
We consider the case of average rotation speed and strong confinement.
The results are shown in Figures \ref{fig:2_espece_om_6_seg}.
Numerically, the experiment yields vortices in both components of the numerical minimizer,
since the rotation is more important than in the two previous experiments.
Moreover, the computation of the indices of the vortices in both components of the minimizer,
using the algorithm described in Section \ref{subsec:index}, shows that the vortices
are singly quantized
(Fig. \ref{fig:2_espece_om_6_1_seg_vortex} and \ref{fig:2_espece_om_6_2_seg_vortex}).
This is in accordance with the theory presented in section \ref{subsubsec:twocomponent}
(second bullet point of the rotational case).

\begin{figure}[ht]
		\centering
		\begin{subfigure}[t]{0.45\textwidth}
			\centering
			\includegraphics[width=.8\textwidth]{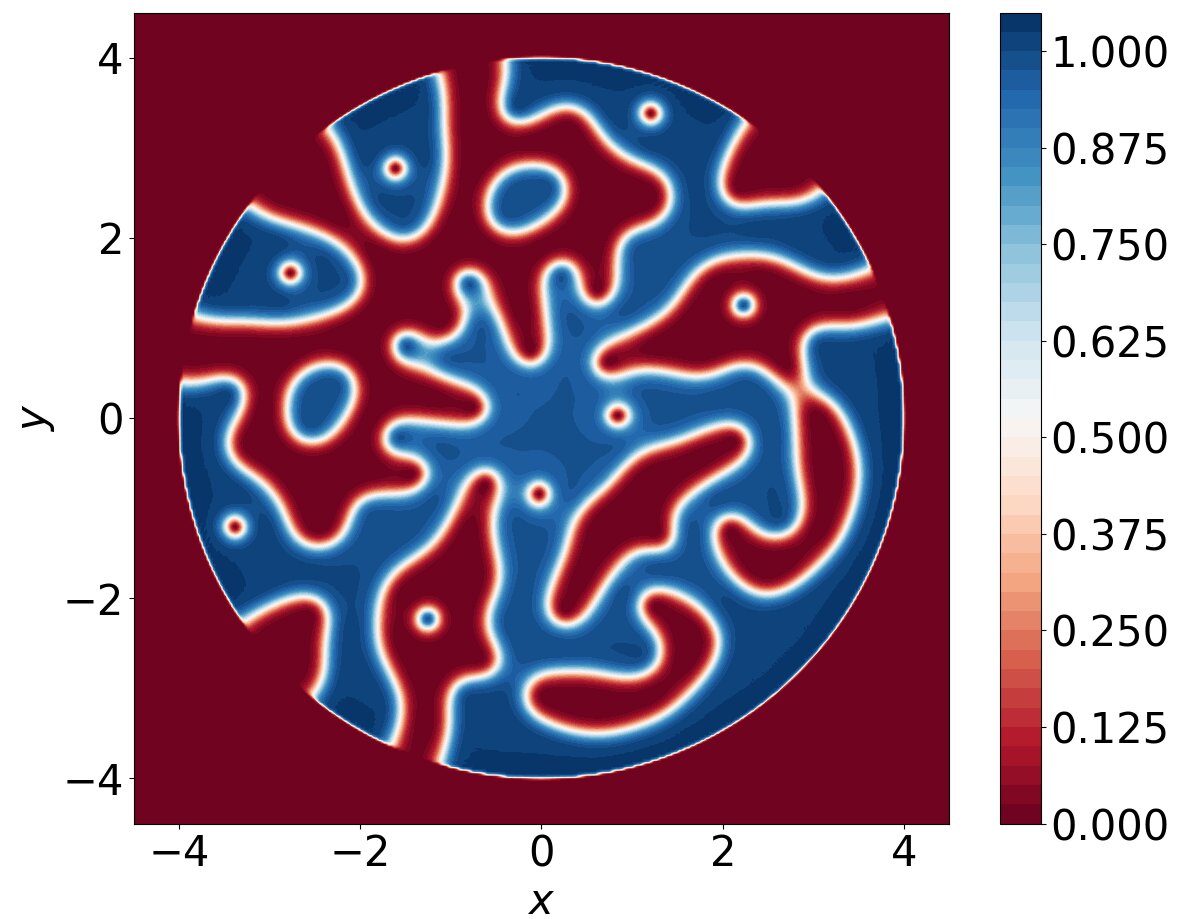}
			\subcaption{The squared modulus of the first component.}
			\label{fig:2_espece_om_6_1_seg}
		\end{subfigure}
		\hfill
		\begin{subfigure}[t]{0.45\textwidth}
			\centering
			\includegraphics[width=.8\textwidth]{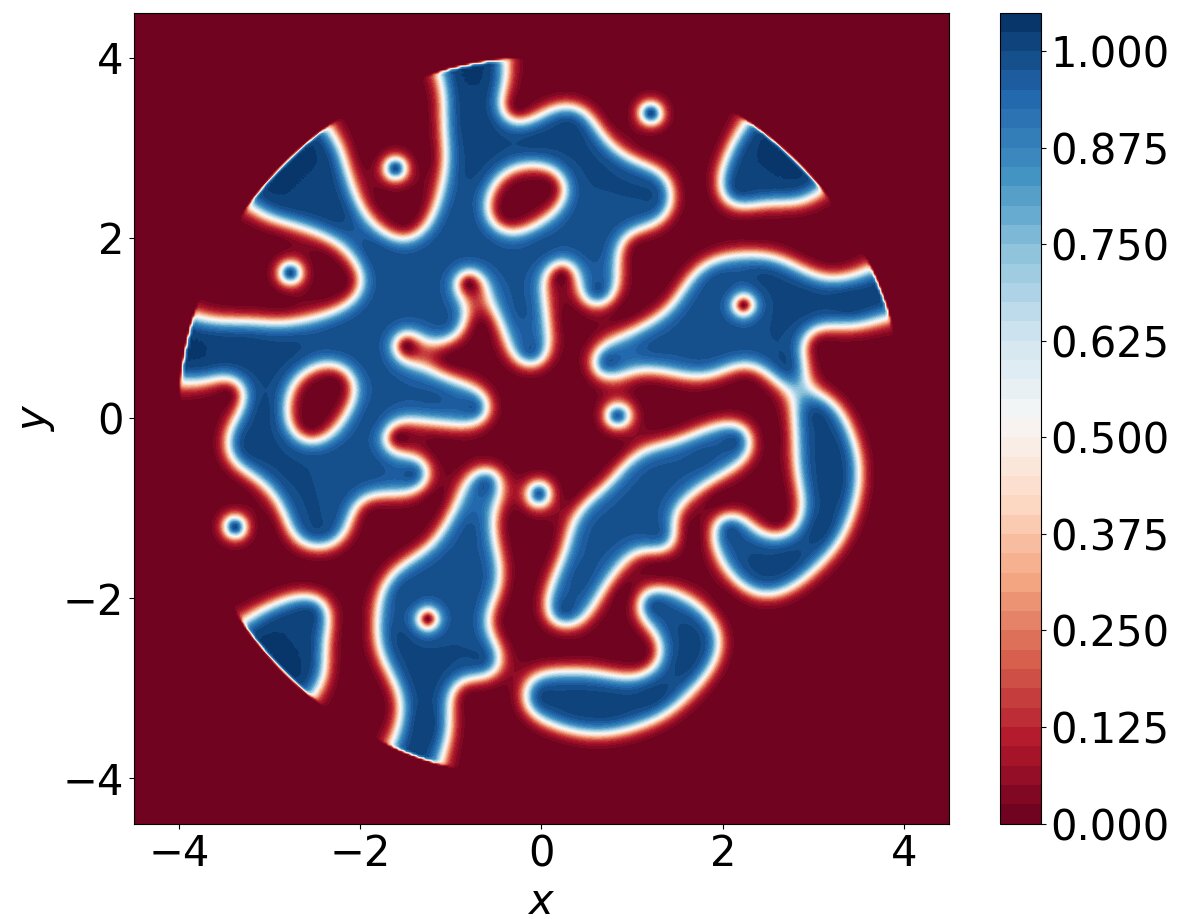}
			\subcaption{The squared modulus of the second component.}
			\label{fig:2_espece_om_6_2_seg}
		\end{subfigure}
		\hfill
        \vskip\baselineskip
		\begin{subfigure}[t]{0.45\textwidth}
			\centering
			\includegraphics[width=.8\textwidth]{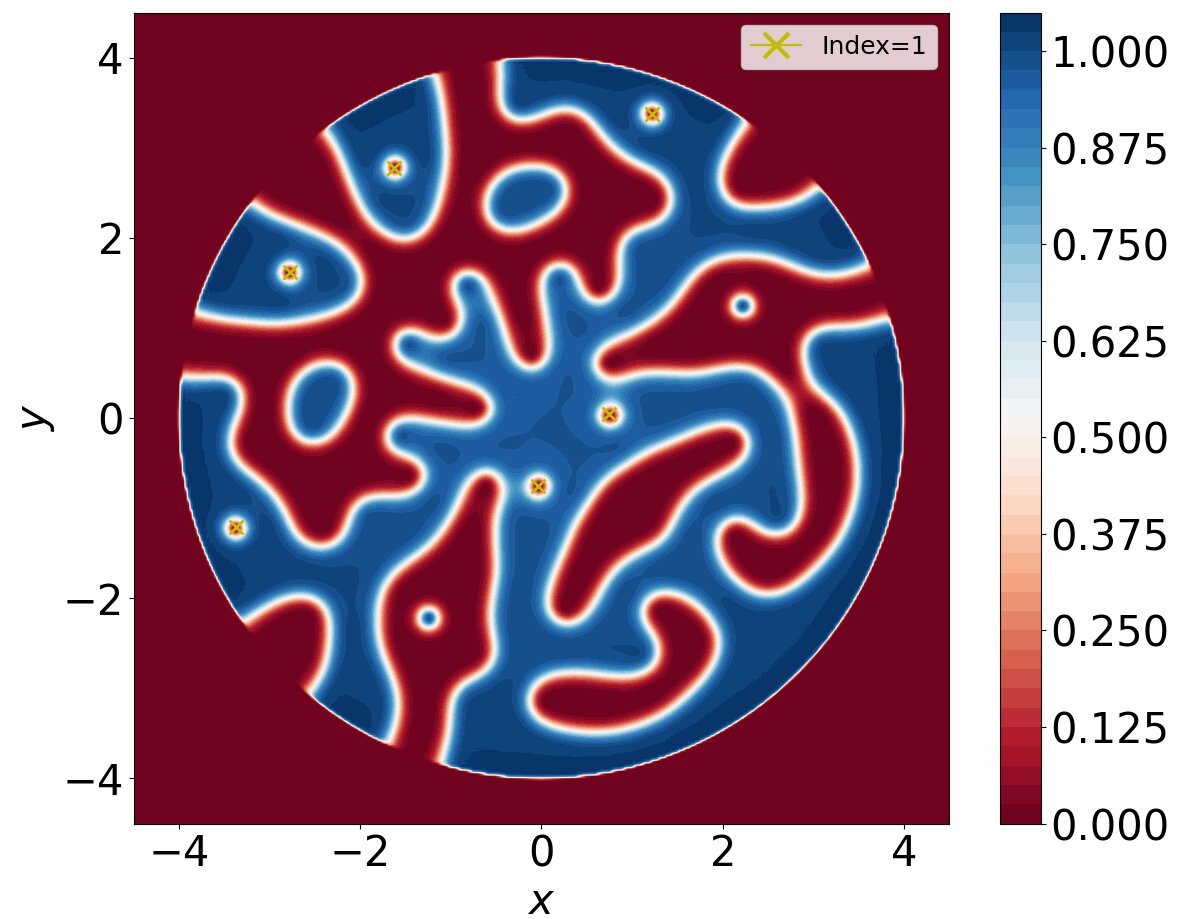}
			\subcaption{The vortices' indices of the first component.}
			\label{fig:2_espece_om_6_1_seg_vortex}
		\end{subfigure}
		\hfill
		\begin{subfigure}[t]{0.45\textwidth}
			\centering
			\includegraphics[width=.8\textwidth]{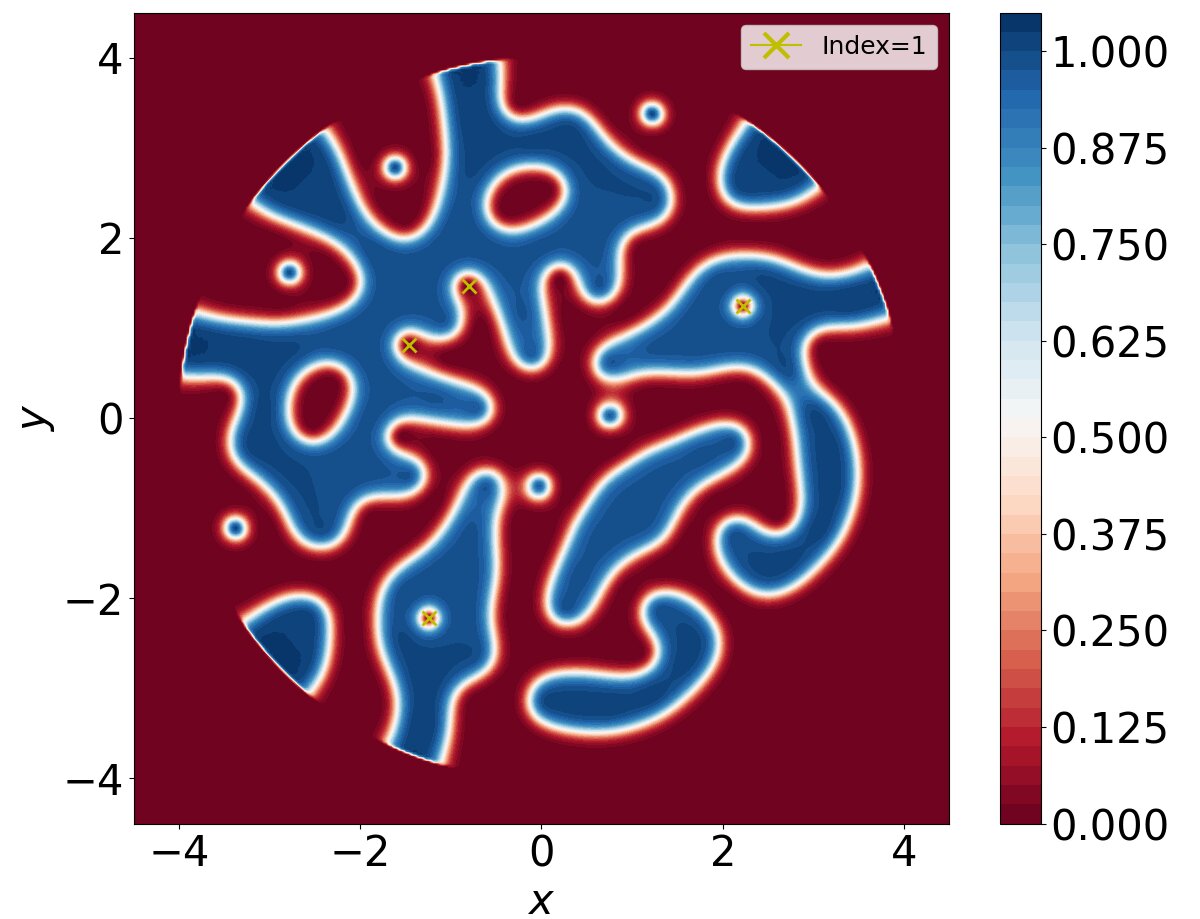}
			\subcaption{The vortices' indices of the second component.}
			\label{fig:2_espece_om_6_2_seg_vortex}
		\end{subfigure}
	\caption{The squared modulus of the first component (A) and the second component (B)
          of a minimizer in the case of two components condensate with moderate rotation $\Omega=6$.
          The vortices' indices of the first and second component are presented in Fig. (C) and (D)
          (respectively) with $N_{min}=1$, $N_{max}=5$, $tol_1=0.05$ and $tol_2=0.02$.}
        \label{fig:2_espece_om_6_seg}
\end{figure}

\subsubsection{High rotation case: $\Omega=15$}
We consider the case of high rotation velocity and strong confinement.
The results are shown in Figures \ref{fig:2_espece_om_15_seg}.
The numerical experiment shows that we still are in a segregation regime, since the supports
of the minimizers tend to {\it not} overlap.
Moreover, the rotation speed is sufficiently big to observe numerically
the formation of vortex sheets in each component of the minimizer.
This is in accordance with the theory presented in section \ref{subsubsec:twocomponent}
(last bullet point of the rotational case).
The computation of the indices of the vortex sheets is carried out using the algorithm
described in Section \ref{subsec:sheet_index}, with $m=0.4$, $M=0.6$ and $tol_3=0.3$ in each
component of the numerical minimizer.
The computed indices of the vortex sheets are positive,
which validates numerically the existence of vortex sheets with a phase circulation
when $\delta_{\varepsilon} \to 1$ and $\tilde{\varepsilon} \to 0$,
which was conjectured in \cite{AftalionSandier2020} (see also Section \ref{subsubsec:twocomponent}).


\begin{figure}[ht]
	\centering
	\begin{subfigure}[t]{0.45\textwidth}
		\centering
		\includegraphics[width=.8\textwidth]{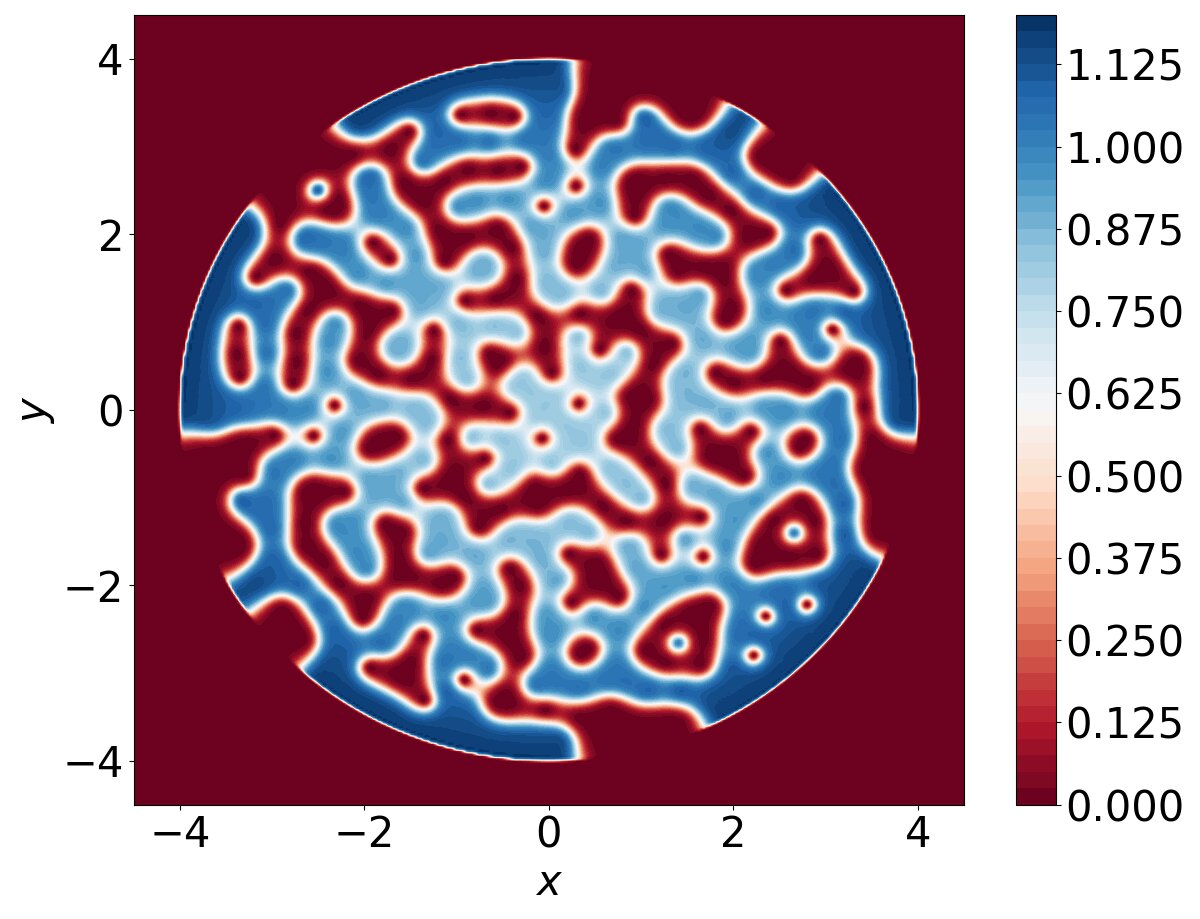}
		\subcaption{The squared modulus of the first component.}
		\label{fig:2_espece_om_15_seg_1}
	\end{subfigure}
	\hfill
	\begin{subfigure}[t]{0.45\textwidth}
		\centering
		\includegraphics[width=.8\textwidth]{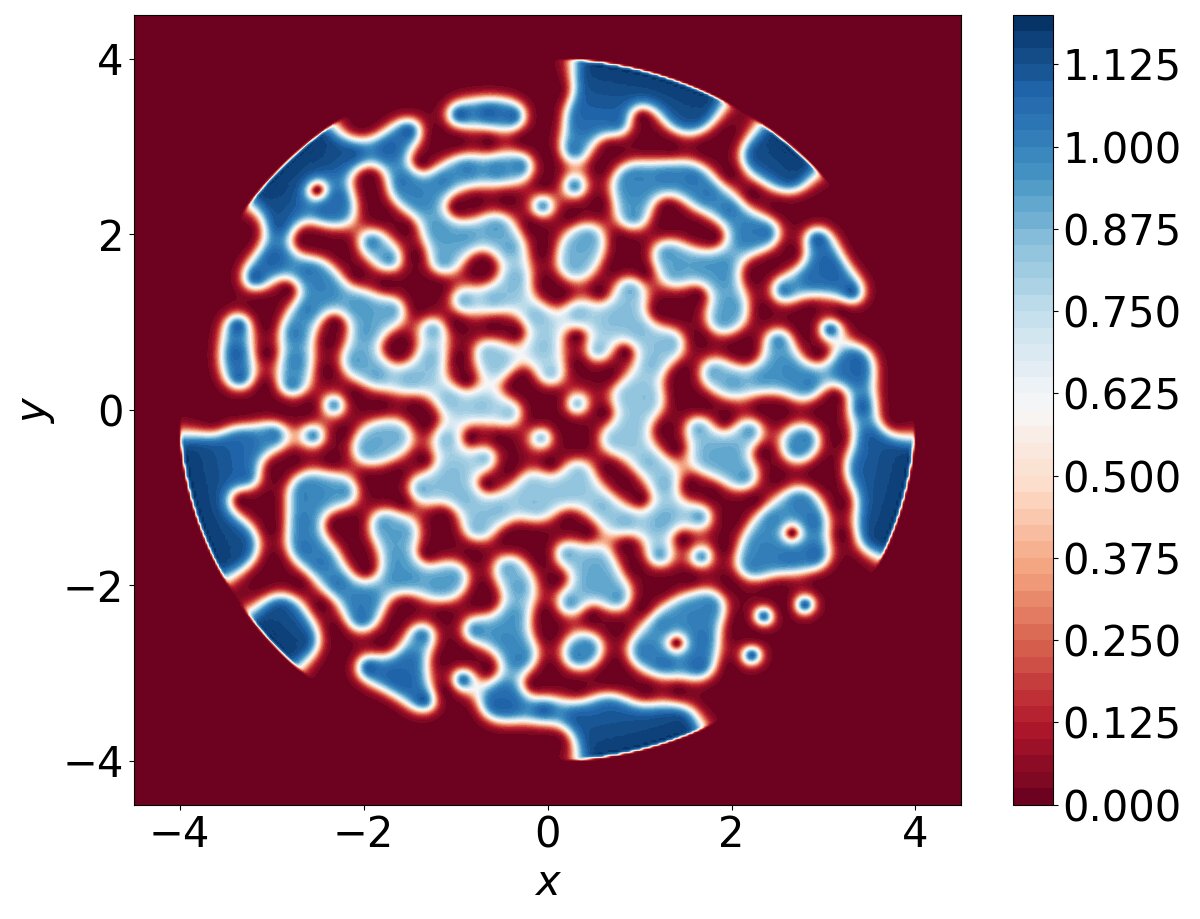}
		\subcaption{The squared modulus of the second component.}
		\label{fig:2_espece_om_15_seg_2}
	\end{subfigure}        
    \vskip\baselineskip
	\begin{subfigure}[t]{0.45\textwidth}
		\centering
		\includegraphics[width=.8\textwidth]{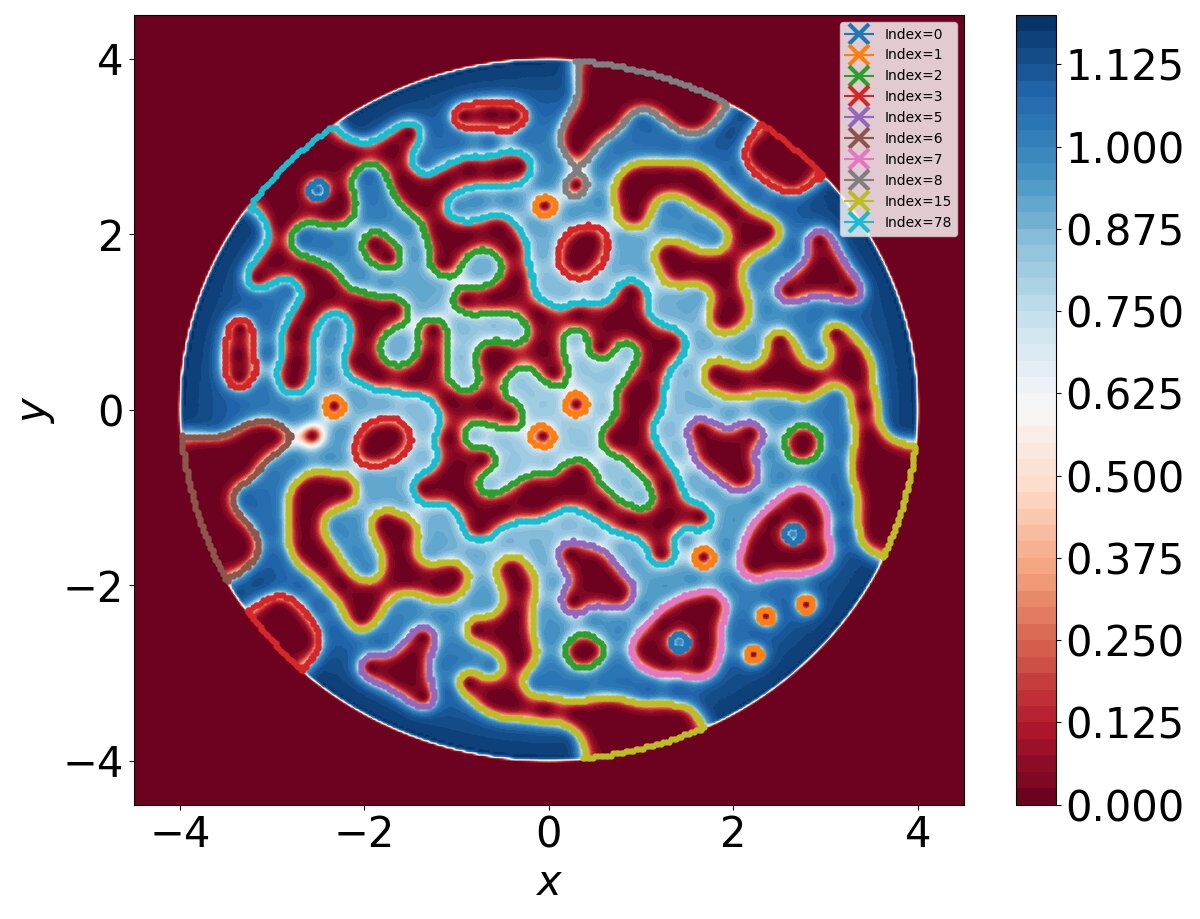}
		\subcaption{The vortex sheet's indices of the first component.}
		\label{fig:2_espece_om_15_seg_1_vortex}
	\end{subfigure}
	\hfill
	\begin{subfigure}[t]{0.45\textwidth}
		\centering
		\includegraphics[width=.8\textwidth]{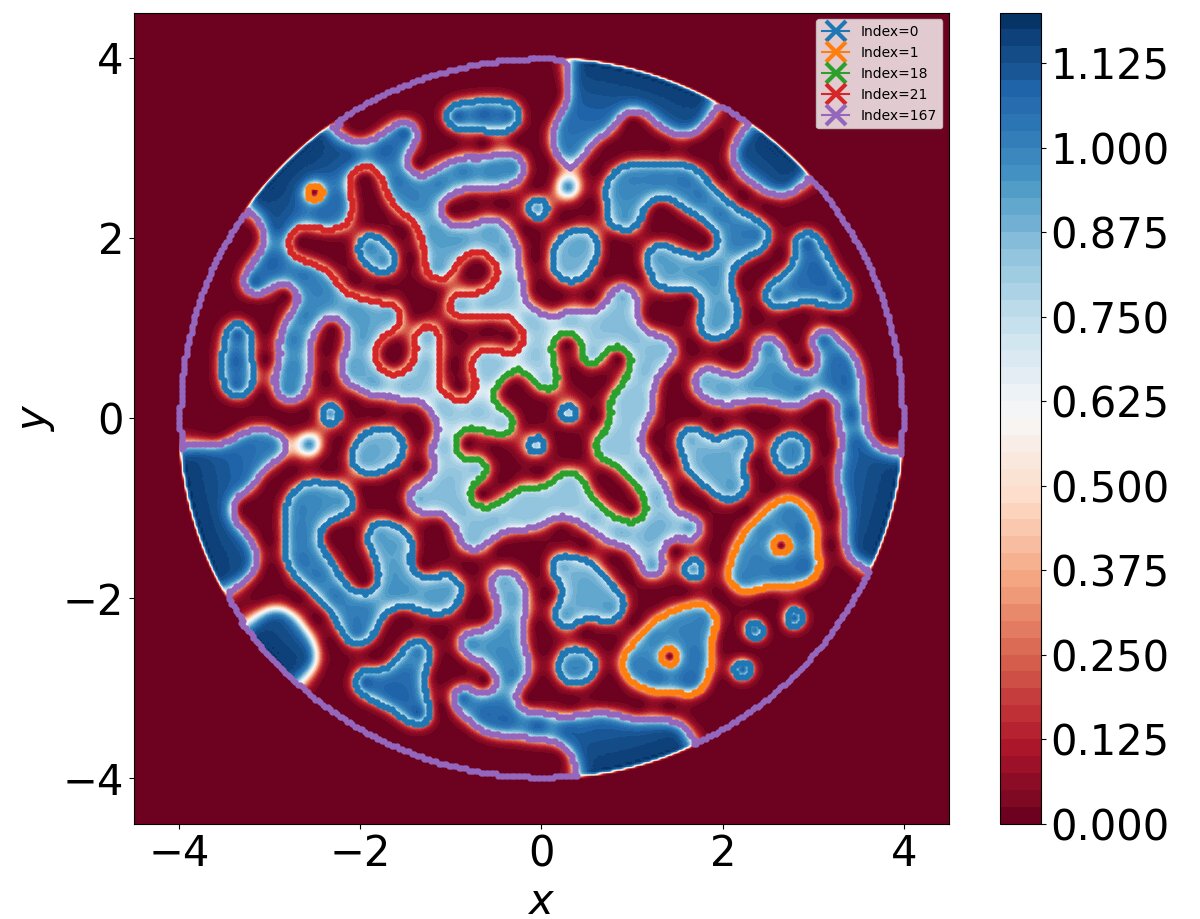}
		\subcaption{The vortex sheets' indices of the second component.}
		\label{fig:2_espece_om_15_seg_2_vortex}
	\end{subfigure}
	\caption{The squared modulus of the first component (A) and the second component (B)
          of a minimizer in the case of two components condensate with high rotation $\Omega=15$.
          The vortex sheets' indices of the first and second component are presented
          in Fig. (C) and (D) (respectively) with $m=0.4$, $M=0.6$ and $tol_3=0.3$.}
	\label{fig:2_espece_om_15_seg}
\end{figure}

\subsection{Two components condensate in the coexistence regime $\delta<1$}
\label{subsec:twocompcoex}

In this section, we consider the numerical behavior of two components Bose--Einstein condensates
in the coexistence regime ($\delta<1$) as $\varepsilon \to 0$.
As mentioned in \ref{subsubsec:twocomponent}, depending on the value of $\delta$ and $\Omega$,
we expect four different qualitative behaviours for the minimizers.

\subsubsection{Common parameters}
In this section, we use the following parameters.
The function $\rho$ is the same as in \eqref{eq:defrho}.
The physical parameters are $\varepsilon=5\times10^{-2}$, $N_1=0.55$ and $N_2=0.45$.
The discretization parameters are $L=7$, $R=4$, $h=0.1$ and $h_0=10^{-12}$.
All the numerical results in this Section \ref{subsec:twocompseg} are obtained by first
minimizing the discrete energy with $N=128$, then interpolating the real and imaginary parts
of each component to a grid of $N=256$ points in each direction, then minimizing the corresponding
discrete energy.
For all the experiments in this section, the EPG algorithm of Section \ref{subsec:gradient}
converged because of the stopping criterion on $K^\Delta$, with a value less or equal to $10^{-2}$.
These parameters correspond to the coexistence regime ($\delta < 1$).

\begin{remark}
  When minimizing $E^\Delta_{\varepsilon,\delta}$ with high rotation speed $\Omega$,
  giant holes are created in some of the next simulations.
  To circumvent this, from now on, we will refer to a discrete analogue of
  ${\mathcal E}_{\varepsilon,\delta}^\Omega$ (the energy {\it with} the centrifugal force) as
  $\mathcal{E}^\Delta_{\varepsilon,\delta}$.
  Of course, $E^\Delta_{\varepsilon,\delta}$ still denotes the discrete energy {\it without} the
  centrifugal force. See Remark \ref{rem:centrifuge}.
\end{remark}

\subsubsection{No vortices}
We first consider the case of no rotation ($\Omega=0$), strong confinement
and weak interaction strength $\delta=0.9$ ({\it i.e. $1-\delta=0.1$}).

We start with a symmetric numerical initial datum corresponding to that of Section
\ref{subsubsec:param_2species_norotation}.
The results are shown in Figure \ref{fig:coex_fft}.
The numerical experiment confirms that we are in a coexistence regime since
the two components of the minimizer are disk-shaped,
with almost constant moduli away from the boundary, with similar (hence overlapping) supports.
Of course, the values of the constant moduli for $\psi_1$ and $\psi_2$ depend
on the values of $N_1$ and $N_2$.
This is in accordance with the theory presented in Section \ref{subsubsec:twocomponentcoex}.

\begin{figure}[ht]
	\centering
	\begin{subfigure}[t]{0.45\textwidth}
		\centering
		\includegraphics[width=.8\textwidth]{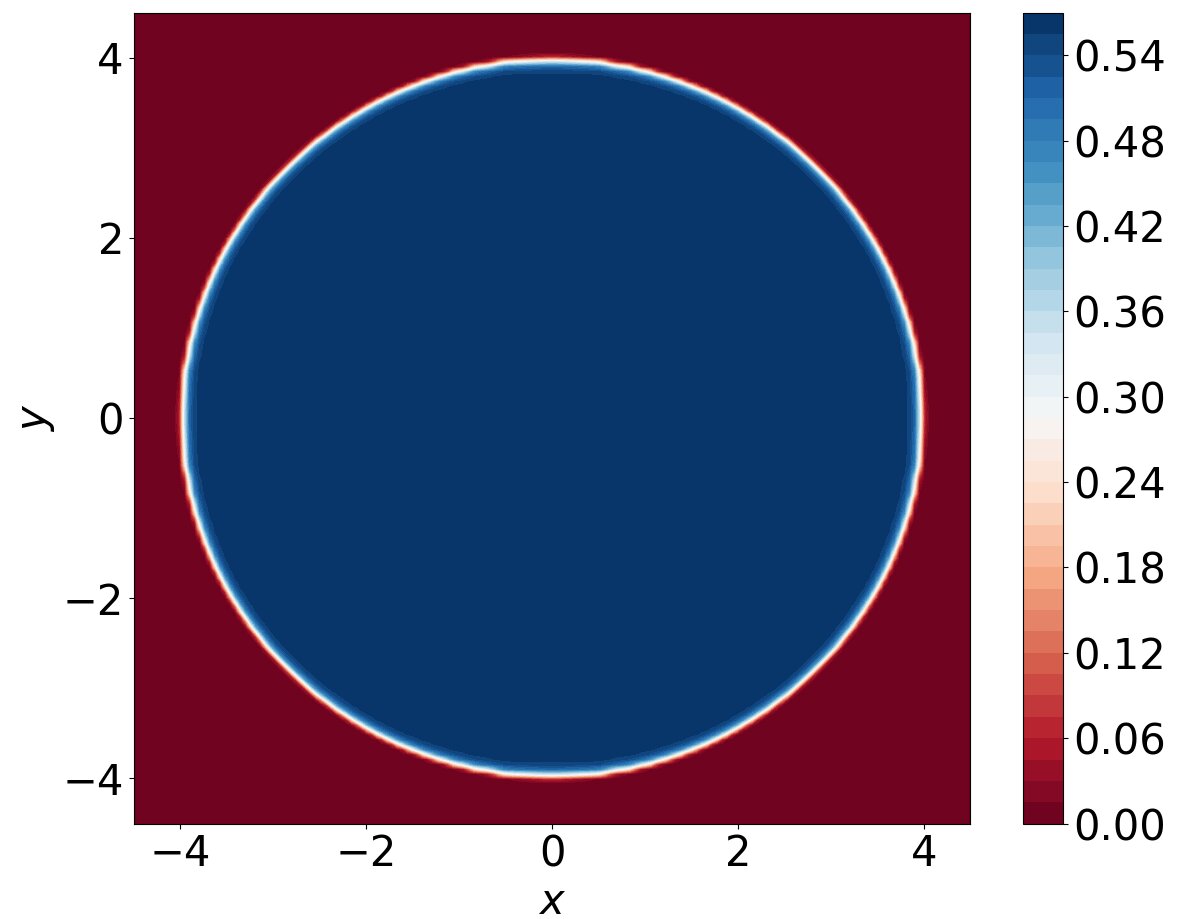}
		\subcaption{The squared modulus of the first component.}
		\label{fig:coex_fft_1}
	\end{subfigure}
	\hfill
	\begin{subfigure}[t]{0.45\textwidth}
		\centering
		\includegraphics[width=.8\textwidth]{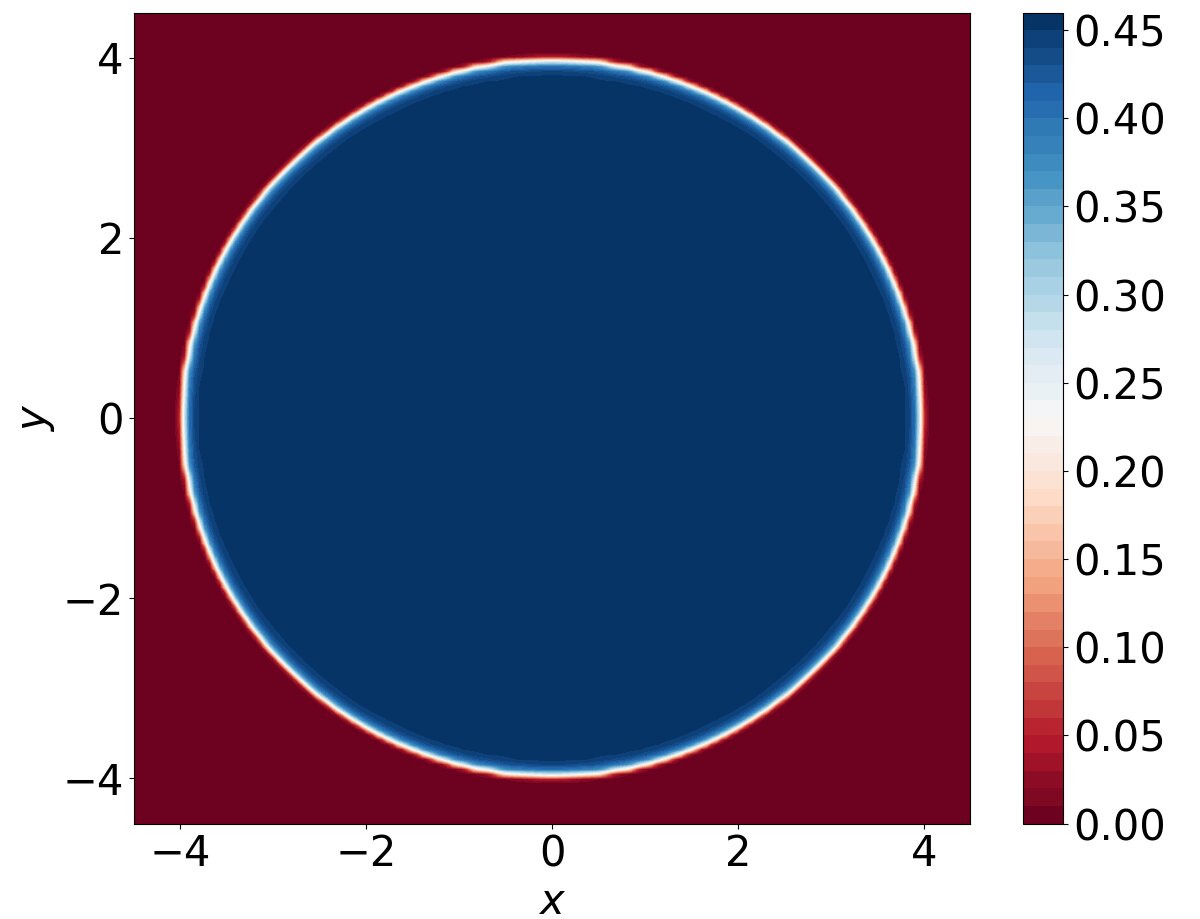}
		\subcaption{The squared modulus of the second component.}
		\label{fig:coex_fft_2}
	\end{subfigure}
	\caption{The squared modulus of the first component (A) and of the second component (B)
          of a minimizer of the energy $E_{\varepsilon,\delta}^\Delta$,
          in the case of no rotation ($\Omega=0$) and weak interaction strength ($\delta=0.9$).}
	\label{fig:coex_fft}
\end{figure}

\subsubsection{Triangular vortex lattices}
In this section, we consider the case of strong interaction strength ($\delta=0.2$,
{\it i.e. $1-\delta=0.8$})
and high rotation $(\Omega=7)$.
In order to break off the symmetry, we choose a decentered initial datum, given by 
\begin{eqnarray}
  \nonumber
\psi^1 (x,y)& = & \frac{1}{5} \exp({-10(x-0.5)^2-10(y+0.3)^2}), \\ \label{eq:donneedecentree}
\psi^2 (x,y) & = & \frac{1}{5} \exp({-10(x+0.7)^2-10(y-0.1)^2}).
\end{eqnarray}
The results are shown in Figure \ref{fig:triangular_lattice_fft}.
The numerical minimizer is consistent with the coexistence regime,
and the high rotation is big enough to produce singly quantized vortices
in both components forming a triangular vortex lattice.
This is in accordance with the second behaviour in the theory presented
in Section \ref{subsubsec:twocomponentcoex}.

\begin{figure}[ht]
	\centering
	\begin{subfigure}[t]{0.45\textwidth}
		\centering
		\includegraphics[width=.8\textwidth]{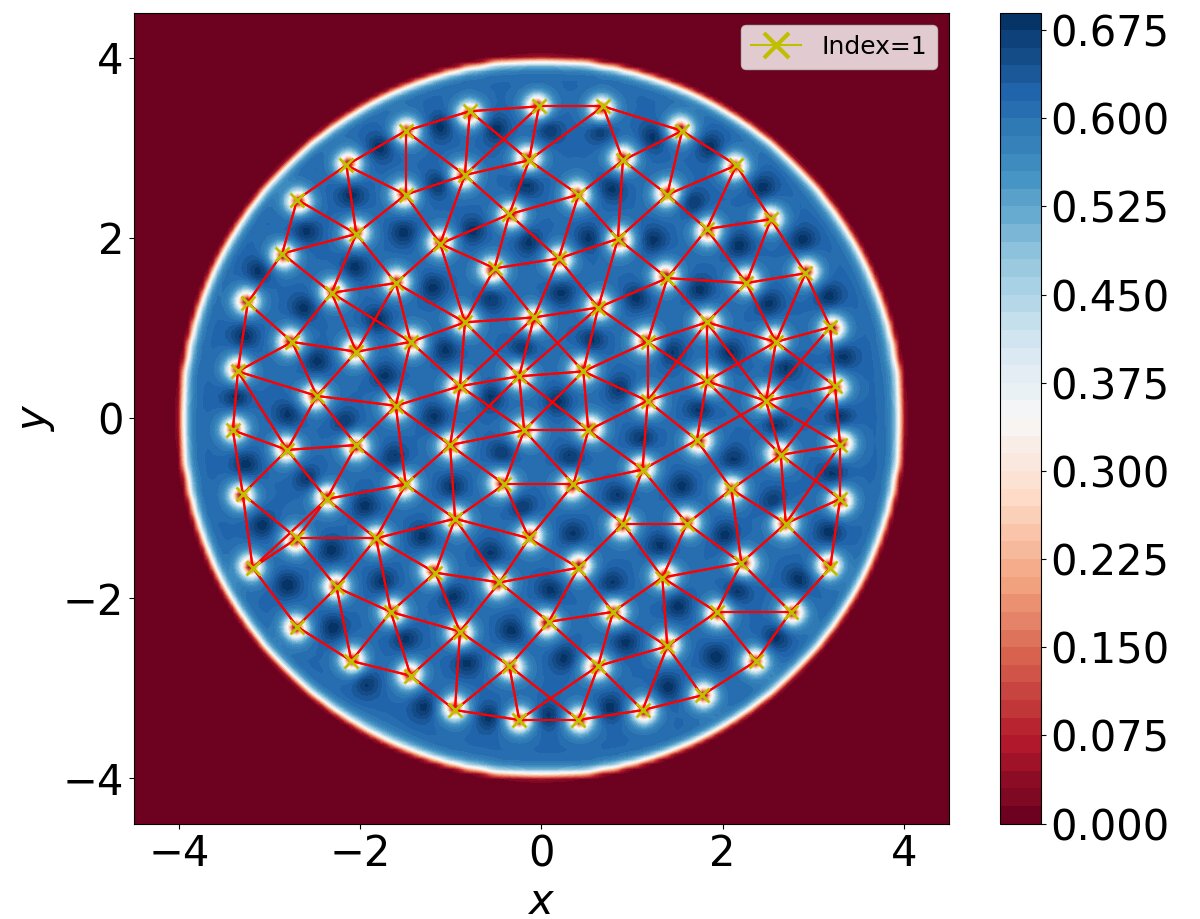}
		\subcaption{The vortices indexes of the first component.}
		\label{fig:triangular_lattice_fft_1}
	\end{subfigure}
	\hfill
	\begin{subfigure}[t]{0.45\textwidth}
		\centering
		\includegraphics[width=.8\textwidth]{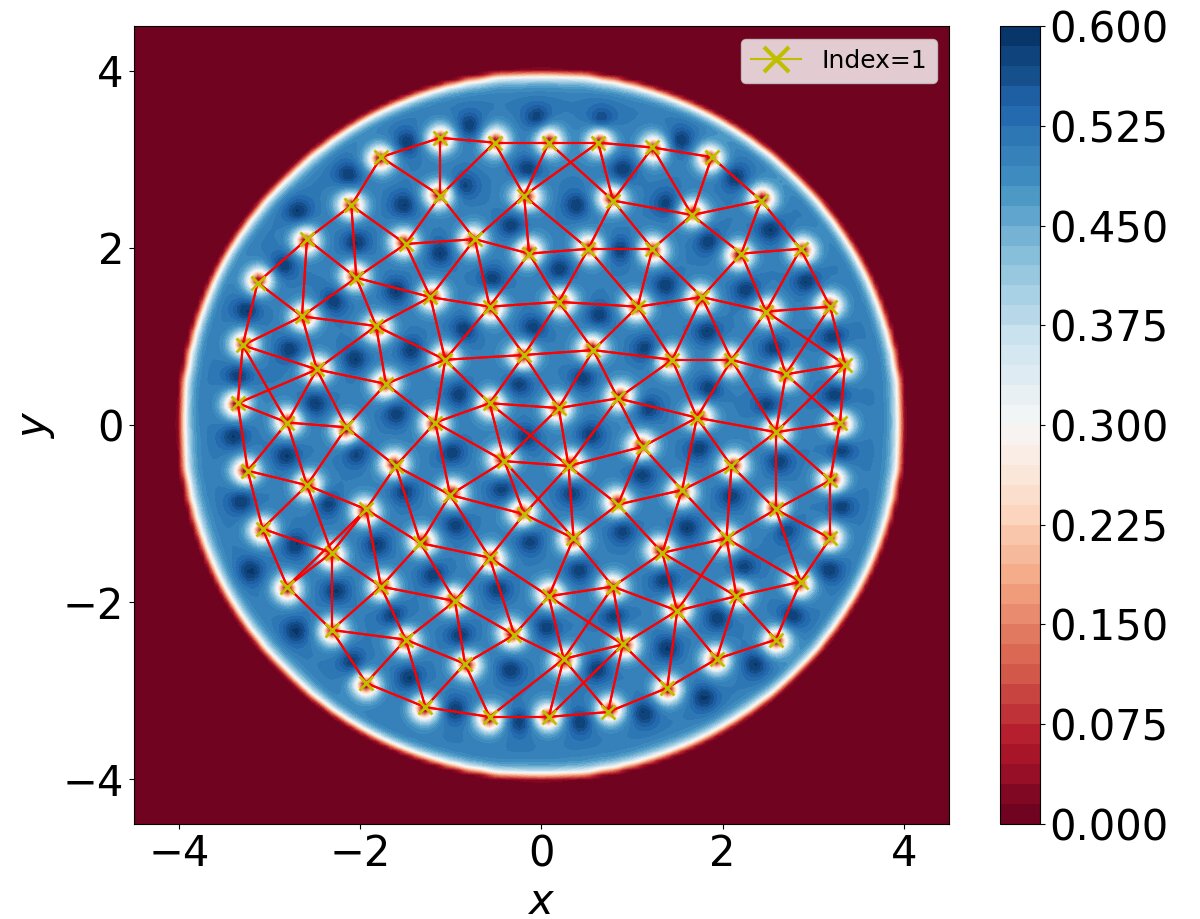}
		\subcaption{The vortices indexes of the second component.}
		\label{fig:triangular_lattice_fft_2}
	\end{subfigure}
	\caption{The vortices' indices of the first component (A) and second component (B) in the case of a two components condensate with high rotation ($\Omega=7$) and strong interaction strength ($\delta=0.2$) for the energy $\mathcal{E}_{\varepsilon,\delta}^\Delta$ followed by red lines highlighting the triangular lattices.}
	\label{fig:triangular_lattice_fft}
\end{figure}

\subsubsection{Square vortex lattices}
In this section, we consider the case of weak interaction strength ($\delta=0.8$,
{\it i.e. $1-\delta=0.2$})
and high rotation ($\Omega=8$).
We use the same non-symmetric initial data as before (see Equation \ref{eq:donneedecentree}).
The results are shown in Figure \ref{fig:square_lattice_fft}.
Numerically, the minimizer confirms that we are in a coexistence regime,
and the weak interaction strength with the high rotation produce a square vortex lattice
with singly quantized vortices in both components of the minimizer.
This is in accordance with the third behavior in the theory presented
in Section \ref{subsubsec:twocomponentcoex}.

\begin{figure}[ht]
	\centering
	\begin{subfigure}[t]{0.45\textwidth}
		\centering
		\includegraphics[width=.8\textwidth]{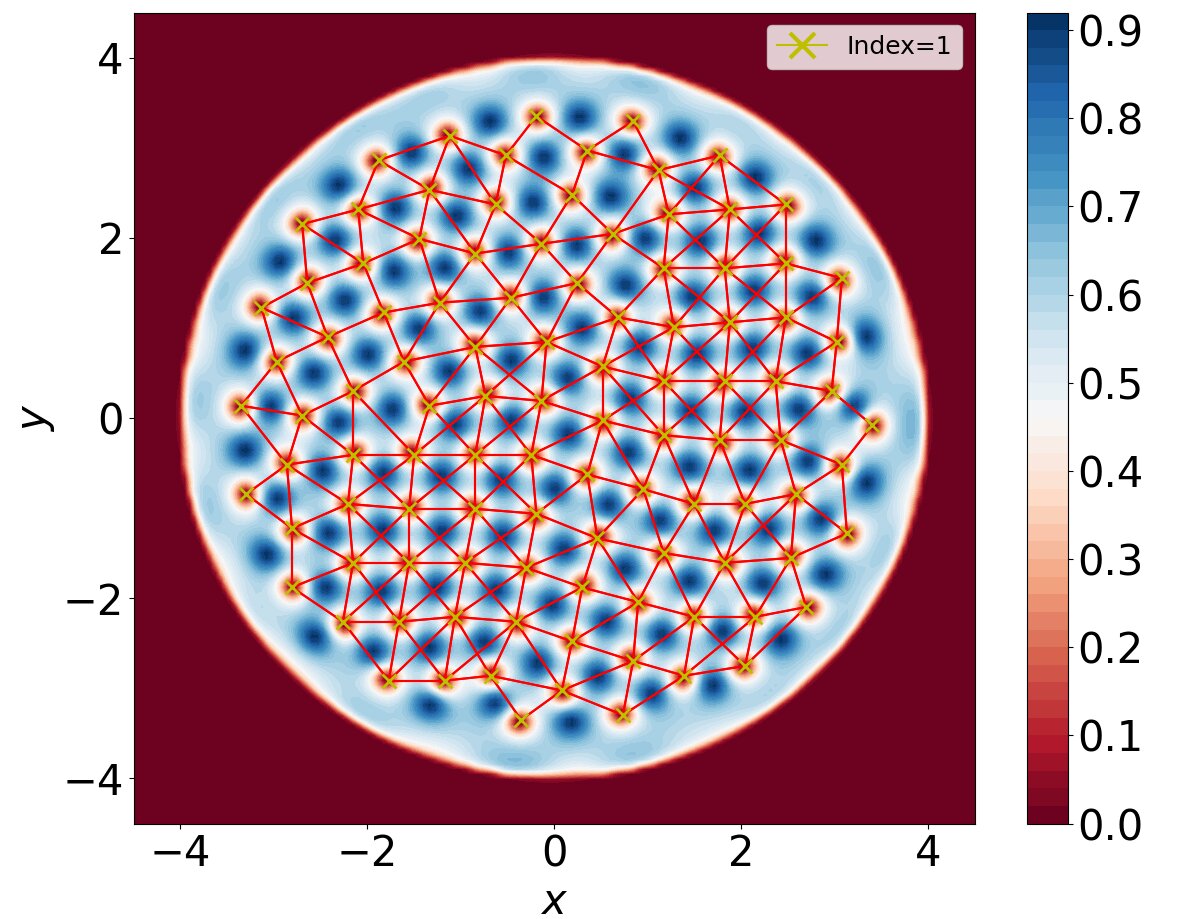}
		\subcaption{The vortices' indices of the first component.}
		\label{fig:square_lattice_fft_1}
	\end{subfigure}
	\hfill
	\begin{subfigure}[t]{0.45\textwidth}
		\centering
		\includegraphics[width=.8\textwidth]{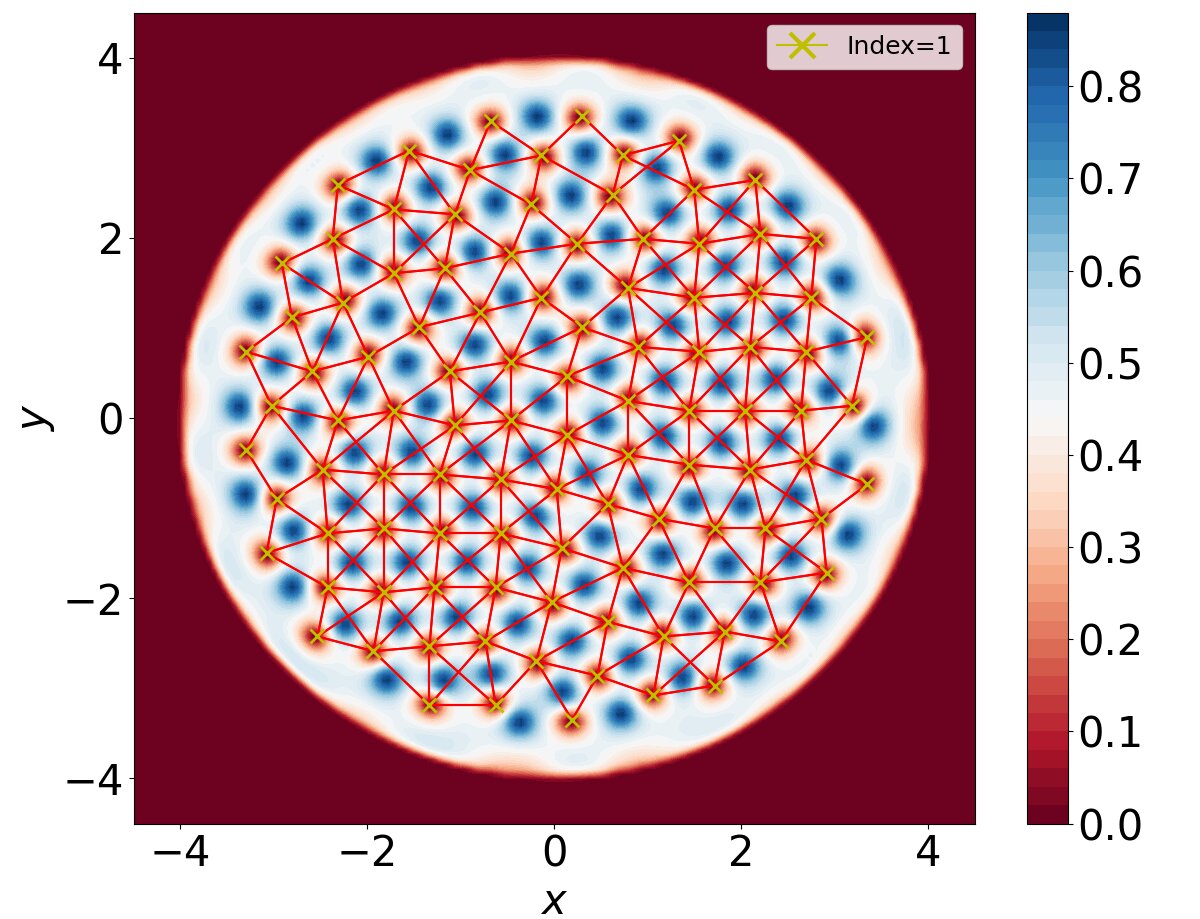}
		\subcaption{The vortices' indices of the second component.}
		\label{fig:square_lattice_fft_2}
	\end{subfigure}
	\caption{The vortices' indices of the first component (A) and second component (B)
          in the case of a two components condensate with high rotation ($\Omega=8$)
          and weak interaction strength ($\delta=0.8$)
          for the energy $\mathcal{E}_{\varepsilon,\delta}^\Delta$
          followed by red lines highlighting the square lattice.}
	\label{fig:square_lattice_fft}
\end{figure}

\subsubsection{Double core and Stripe vortex lattice}
In this last section, we fix $\delta$ close to $1$, and we hope to observe
either stripe vortex lattices or double core vortex lattices in the minimizers,
depending on the rotational speed $\Omega$.

In a first simulation, we use \eqref{eq:donneedecentree} as an initial datum as before.
We consider the case of fairly high rotation speed ($\Omega=4$) and weak interaction strength
($\delta=0.998$).
The numerical results obtained are displayed in Figure \ref{fig:double_core_vortex_cf}.
They confirm that we are in a coexistence regime where each component is disk-shaped.
Indeed, in both components, the points inside the disk of radius $R$
where the modulus of the minimizer is {\it very} small are isolated, and there is
at least a bit of mass in each component in all the rest of the disk.
This contrasts for example with all the numerical experiments carried out
in the segregation regime (see Section \ref{subsec:twocompseg}).

We compute the index of the detected vortices, and we display the results
in Figure \ref{fig:double_core_vortex_cf} as well.
As expected, almost all the vortices are paired up $2$ by $2$, forming double-core vortices.
Moreover, almost all the indices of the numerical vortices are equal to one, which validates
numerically that the zeros of the function have a non-trivial and singly quantized phase circulation.
The only exception is one vortex in the first component
(see Figure \ref{fig:double_core_cf_vortex_1}) which has an index of $2$.
An explanation for this is that it consists in two vortices that are too close to each other,
and hence cannot be detected as two different vortices by our vortex detection algorithm
(see Section \ref{subsec:index}).
This is in accordance with the last behaviour predicted by the theory presented
in Section \ref{subsec:twocompcoex}.

\begin{figure}[ht]
	\centering
	\begin{subfigure}[t]{0.45\textwidth}
		\centering
		\includegraphics[width=.8\textwidth]{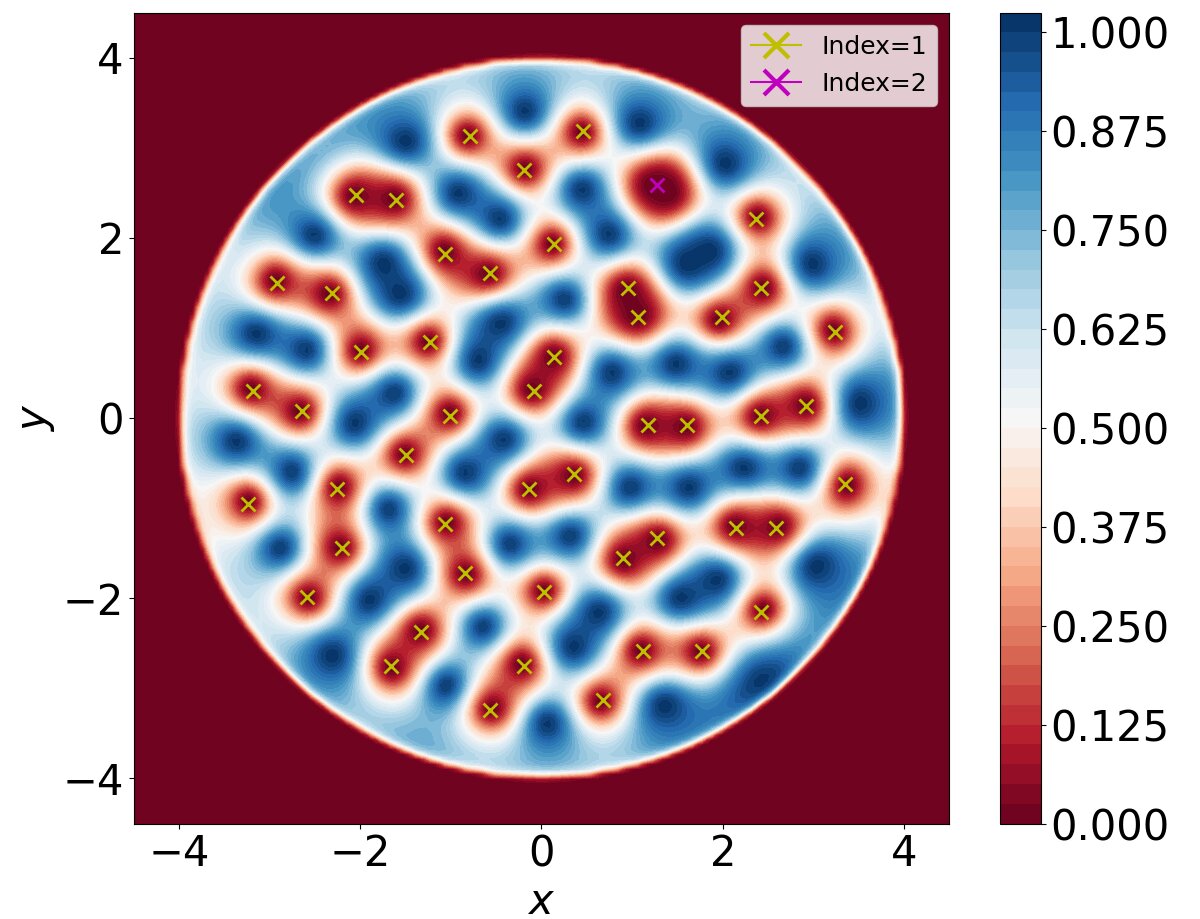}
		\subcaption{The squared modulus and the vortices' indices of the first component for $N_{min}=3$, $N_{max}=4$.}
		\label{fig:double_core_cf_vortex_1}
	\end{subfigure}
	\hfill
	\begin{subfigure}[t]{0.45\textwidth}
		\centering
		\includegraphics[width=.8\textwidth]{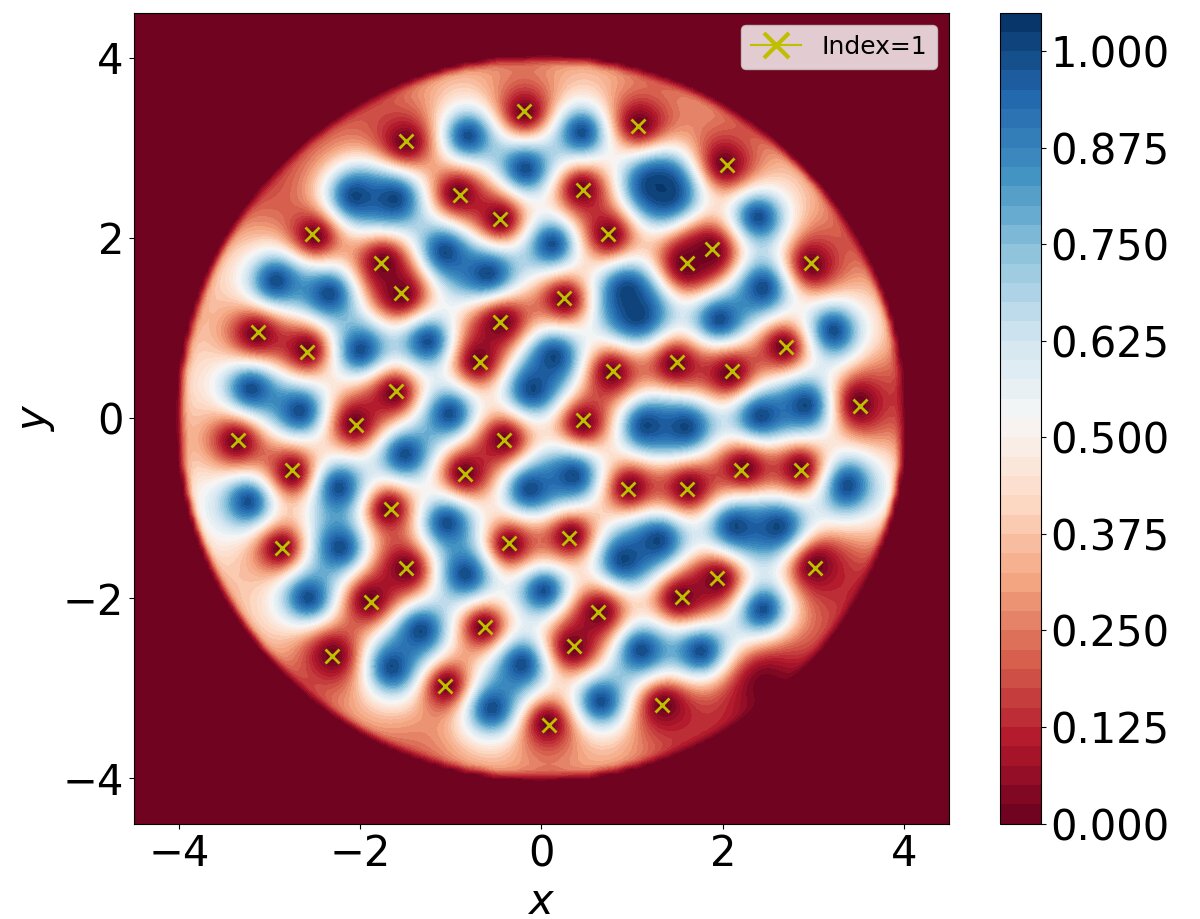}
		\subcaption{The squared modulus and the vortices' indices of the second component for $N_{min}=1$, $N_{max}=4$.}
		\label{fig:double_core_cf_vortex_2}
	\end{subfigure}
	\caption{The squared modulus of the first component (A) and the second component (B)
          of a minimizer (in the case of double-core pattern) for the energy
          $\mathcal{E}_{\varepsilon,\delta}^\Delta$ with $\Omega=4$ and $\delta=0.995$.
          The vortices' indices of each component are also presented in (A) and (B)
          (respectively) for $tol_1=0.05$ and $tol_2=0.01$.}
	\label{fig:double_core_vortex_cf}
\end{figure}

In a second simulation, we use the following initial datum
$\psi^1_{n,k} = \psi^2_{n,k} = \sin(x_{n,k}+y_{n,k})$,
still with the convention that $\psi^1$ and $\psi^2$ satisfy the Dirichlet boundary conditions.
We consider the case of fairly high rotation speed ($\Omega=5$)
and weak interaction strength ($\delta=0.98$).
The results are shown in Figure \ref{fig:stripe_cf_vortex}.

\begin{figure}[ht]
	\centering
	\begin{subfigure}[t]{0.45\textwidth}
		\centering
		\includegraphics[width=.8\textwidth]{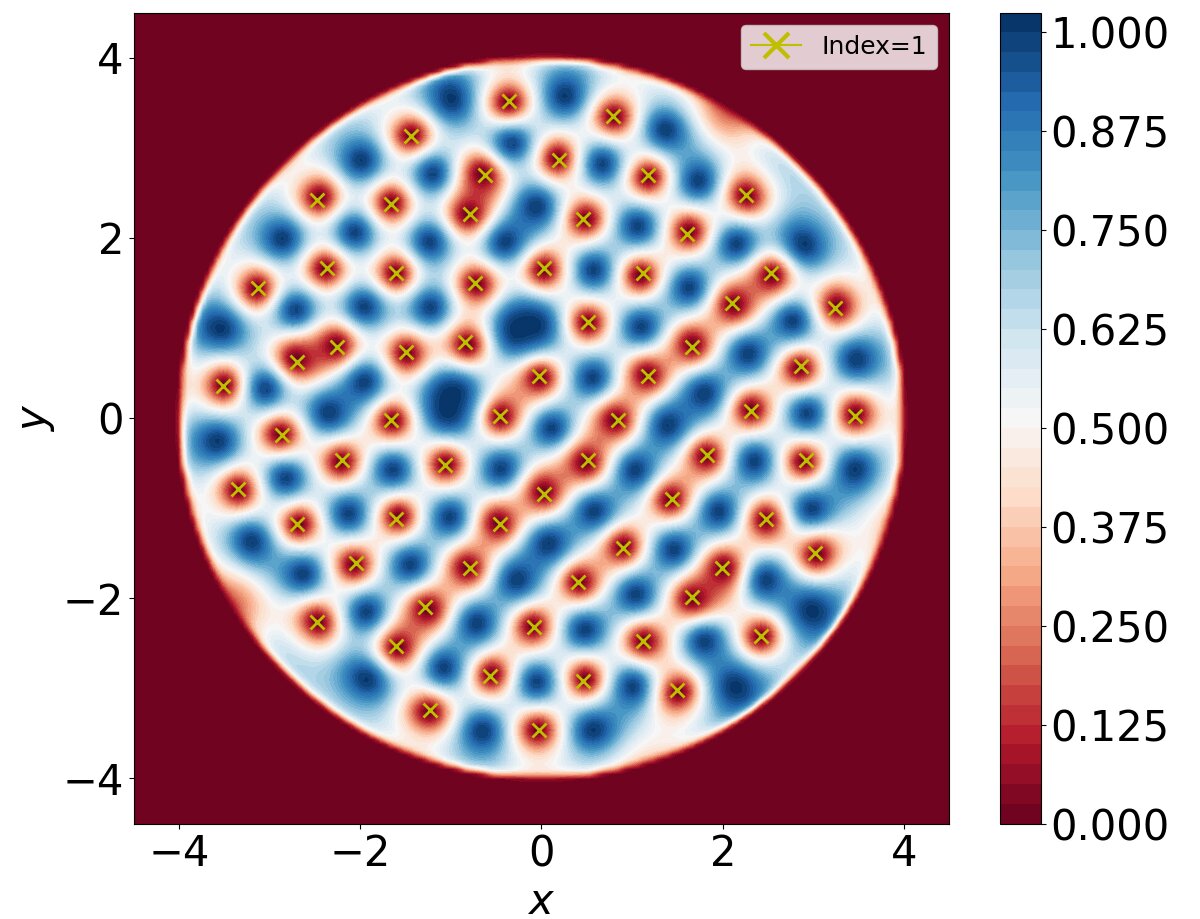}
		\subcaption{The squared modulus and the vortices' indices of the first component for $N_{min}=1$, $N_{max}=3$.}
		\label{fig:stripe_cf_1_vortex}
	\end{subfigure}
	\hfill
	\begin{subfigure}[t]{0.45\textwidth}
		\centering
		\includegraphics[width=.8\textwidth]{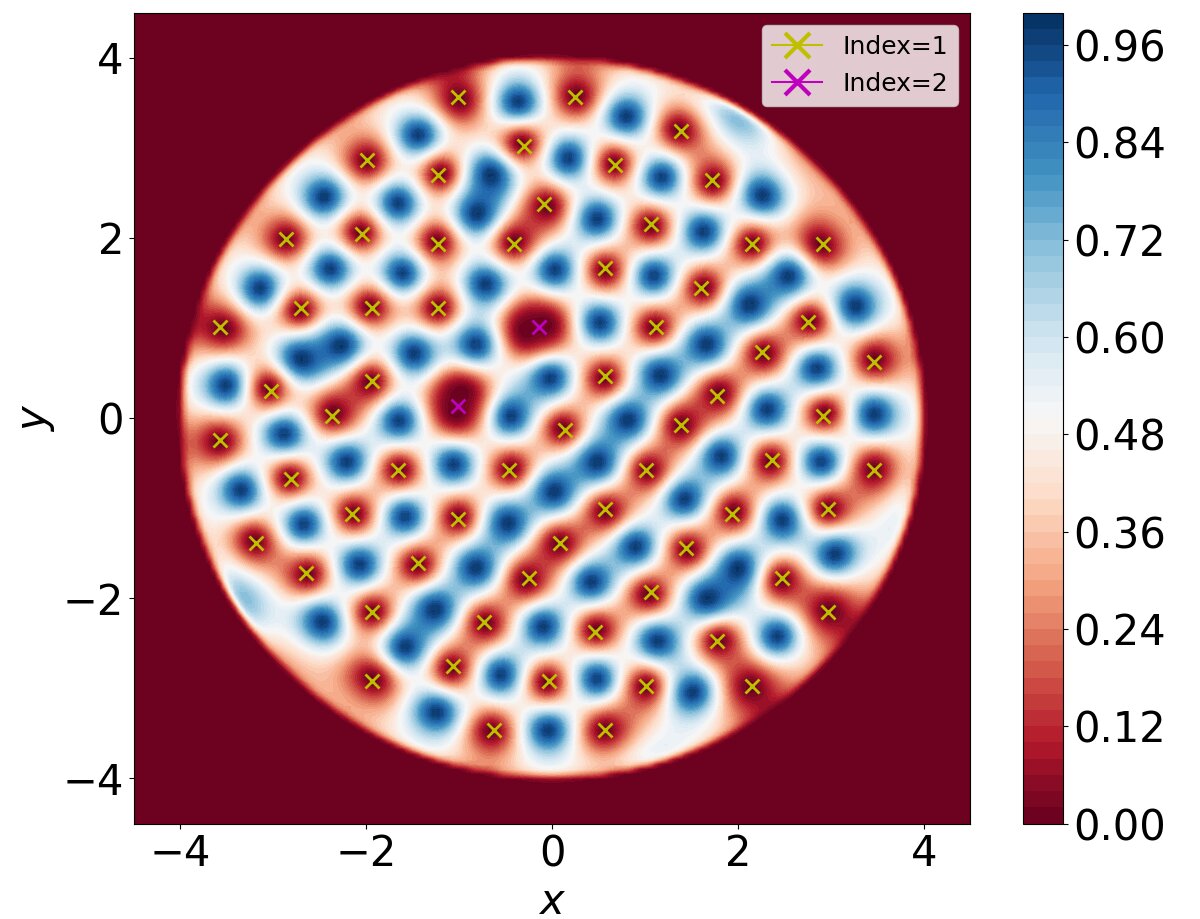}
		\subcaption{The squared modulus and the vortices' indices of the second component for $N_{min}=1$, $N_{max}=4$.}
		\label{fig:stripe_cf_2_vortex}
	\end{subfigure}
	\caption{The squared modulus of the first component (A) and the second component (B)
          of a minimizer (in the case of stripe pattern) for the energy
          $\mathcal{E}_{\varepsilon,\delta}^\Delta$ with $\Omega=5$ and $\delta=0.98$.
          The vortices' indices of each component are also presented in (A) and (B)
          (respectively) for $tol_1=0.05$ and $tol_2=0.01$.}
	\label{fig:stripe_cf_vortex}
\end{figure}

The numerical results confirm that we are in the coexistence regime,
where each component is disk-shaped.
In Figure \ref{fig:stripe_cf_vortex}, we also compute and display the index of the vortices detected.
As expected, almost all the indices of the numerical vortices are equal to $1$,
which validates numerically that the zeros of the function have a non-trivial and singly-quantized
phase circulation.
An explanation for the two exceptions (doubly-quantized vortices in the second component
displayed in Figure \ref{fig:stripe_cf_2_vortex}) is once again that our algorithm
simply fails to detect two separate vortices.
Moreover, we can see a stripe pattern in both components
(Figures \ref{fig:stripe_cf_1_vortex} and \ref{fig:stripe_cf_2_vortex}).
This is in accordance with the last behaviour in the theory presented in \ref{subsec:twocompcoex}.

\section{Comparison with GPELab}\label{sec:comparaisonGPE}
In this section, we study the efficiency of the EPG method
introduced in Section \ref{sec:discretmin} compared to that of GPELab's method (see \cite{GPE1,GPE2})
for minimizing the energy $E_{\varepsilon,\delta}^\Delta$.
At first sight, the EPG method is explicit whereas GPELab's method solves a linear system
at each time step. The time step is denoted by $\Delta_t$ in GPELab, while we denote it by
$h$ in this paper.
We display for each comparison test three convergence criteria (see below)
as well as the number of iterations, the energy and the execution time.
Let $\psi_{n}$ denote the $n^{\text{th}}$ iteration using either EPG or GPELab's algorithm.
We consider and compute the following three convergence criteria:
\begin{enumerate}
	\item The difference between two successive iterations $\|\psi_n-\psi_{n-1}\|^2_\infty$,
	\item The criterion $K^\Delta$ that we developed (see Section \ref{subsec:criterion}),
	\item The energy evolution $|E^\Delta_{\varepsilon,\delta}(\psi_n)-E^\Delta_{\varepsilon,\delta}(\psi_{n-1})|$.
\end{enumerate}
Even though we compute all these quantities in all the numerical experiments in this Section,
we use as a stopping criterion the first criterion above (which is native in the code) for GPELab's
method, and the first two criteria for the EPG method
(in contrast to the previous numerical simulations of Section \ref{sec:numericalresults}
where only the criterion on $K^\Delta$ is used).
For each test case, we compute the final value of the three criteria and we compare them.

\begin{remark}
  \label{rem:GPELab}
  The energy computed by GPELab's method is different from the one given in \eqref{eq:defNRJ}.
  The continuous GPELab's energy reads as follows (see \cite{GPE1,GPE2} for more details):
	\begin{equation*}
	\mathcal{E}_{GPE}(u_1,u_2)=\int_D \sum_{\ell=1}^{2}\bigg[ \frac12 |\nabla u_\ell|^2 + V_\ell|u_\ell|^2
	-\Omega u_\ell^*L_zu_\ell \bigg] + \int_D u^*\frac{\beta}{2} \textbf{F}u,
	\end{equation*}
        where $V_\ell=-\frac{\rho}{2\varepsilon^2}$
        is the confinement function associated to component number $\ell$,
        $\beta=\frac{1}{2\varepsilon^2}$
         is the intra-component interaction, $u^*$ is defined as
        $(\widebar{u_1}, \widebar{u_2})$, the matrix \textbf{F} is given by 
$$\textbf{F}=\begin{pmatrix}
	\beta_{1,1}|u_1|^2+\beta_{1,2}|u_2|^2 & 0 \\
	0 &\beta_{2,1}|u_1|^2+\beta_{2,2}|u_2|^2 
\end{pmatrix},$$
where $\begin{pmatrix}
		\beta_{1,1} & \beta_{1,2}\\
		\beta_{2,1} & \beta_{2,2}
	\end{pmatrix} = \begin{pmatrix}
		1 & \delta\\
		\delta & 1
              \end{pmatrix}$,
and the rotation operator is defined by $L_z = -i\big(x\partial_x - y\partial_x\big)$. 
The relation between the discrete analogue to $\mathcal{E}_{GPE}$ and $E_{\varepsilon,\delta}^\Delta$
(see \eqref{energie_discr}) is the following:
\begin{equation*}
	\mathcal{E}_{GPE}^\Delta=E^\Delta_{\varepsilon,\delta}-\frac{\delta_x\delta_y}{4\varepsilon^2}\sum_{n=0}^{N}\sum_{k=0}^{K}\big( \rho_{n,k}^2 \big)_{\rho>0}.
\end{equation*}
Observe that the difference between the two energies does not depend on $u=(u_1,u_2)$.

In the case of one component condensate (by taking $u_2\equiv0$),
we take $\beta_{1,1}=1$ and $\beta_{1,2}=\beta_{2,1}=\beta_{2,2}=0$,
so that the last term of $\mathcal{E}_{GPE}$ reads
$\int_D u^*\frac{\beta}{2} \textbf{F}u = \int_D \frac{\beta}{2}|u_1|^4$.
\end{remark}

\subsection{One component comparison}
\label{subsec:onecompcomparison}
In this Section, we consider a one component Bose--Einstein condensate.
We compare both algorithms (GPELab's method and EPG) on the same computer.
Each test is done separately.
\subsubsection{Parameters for the EPG method}\label{subsec:com_par_gpe}
We take the following parameters for all the tests.
The function $\rho$ is the same as before (see \eqref{eq:defrho}).
We take $\varepsilon=5\times 10^{-2}$.
For the initial datum, we choose $\psi^1=\frac{1}{5} \exp({-10x^2-10y^2})$.
For the discretization parameters, we take the initial step $h=0.256$, the tolerance  $h_0=10^{-12}$,
the length of the box $2L=14$, the radius of the disk $R=4$, the species ratios $N_1=1$ and $N_2=0$,
the number of points per direction $N=256$ and the convergence tolerance $K_0=10^{-2}$.

\subsubsection{Parameters for GPELab's method}
The equivalent of these parameters in GPELab (with the notations of \cite{GPE1,GPE2}) are
\begin{itemize}
	\item Ncomponents$=1$,
	\item Type='BESP' (for Backward Euler pseudoSPectral scheme),
	\item Delta$=0.5$ (the coefficient in front of the kinetic energy),
	\item Beta$=200$,  which corresponds to 
          $\frac{1}{2\varepsilon^2}$ for $\varepsilon=0.05$,
	\item $V(x,y)=-\frac{1}{2\varepsilon^2}\min[1,10 (R^2-x^2-y^2)]$ (we had to add our own function),
	\item Initial datum: $\psi^1=\frac{1}{5} \exp({-10x^2-10y^2})$ (we had to add it too),
	\item $x_{min}=y_{min}=-7$ and $x_{max}=y_{max}=7$.
\end{itemize}
We had to modify the normalization step after each time step $\Delta_t$.
\subsubsection{Numerical results}
The GPELab algorithm converges whenever the infinity norm of the difference
between two successive iterations is less than a certain value $G_0=\Delta_t \times \text{Stop\_crit}$.
It stops without convergence if the number of iterations done exceeds $10^6$.
In this Section, the EPG method converges if either one of the following statements is true:
at an iteration step $n$, $\|\psi^1_{n}-\psi^1_{n-1}\|_\infty<G_0$ or $K^\Delta < K_0$.

For a first comparison, we set the rotation speed to $\Omega=1$. We also set $\Delta_t= 2\times10^{-3}$ and Stop\_crit$ =  10^{-2}$ so that $G_0=2\times10^{-5}$.
We display the results in Figure \ref{fig:comparison_1_256}.
For a second comparison, we set the rotation speed to $\Omega=3$.  We also set $\Delta_t= 2\times10^{-3}$ and Stop\_crit$ =  10^{-2}$ so that $G_0=2\times10^{-5}$.
We display the results in Figure \ref{fig:comparison_3_256}.
For a third and final comparison, we set the rotation speed to $\Omega=6$. We also set $\Delta_t= 10^{-3}$ and Stop\_crit$ =  10^{-2}$ so that $G_0=10^{-5}$.
The results are displayed in Figure \ref{fig:comparison_6_256}.

\begin{figure}[ht]
	\centering
	\begin{subfigure}[t]{0.45\textwidth}
		\centering
		\includegraphics[width=.8\textwidth]{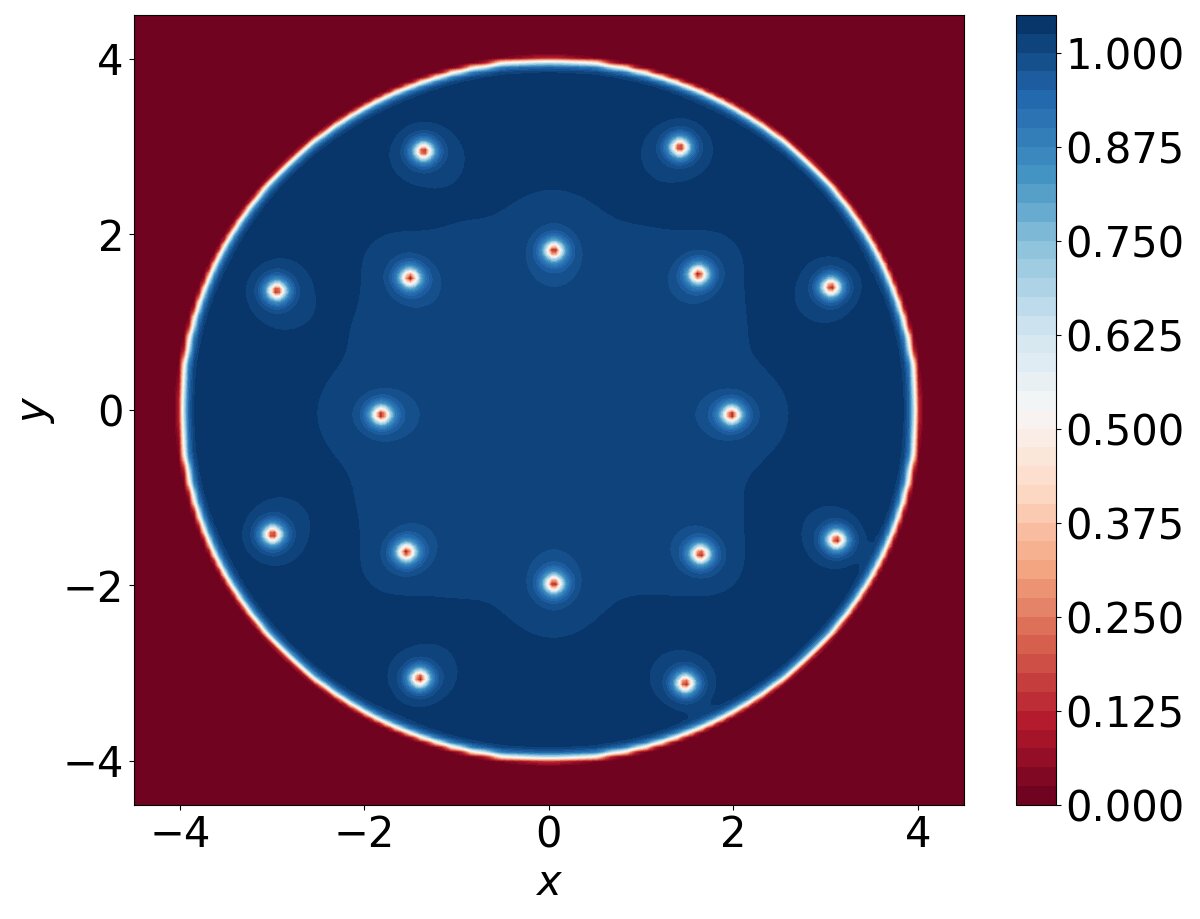}
		\subcaption{A minimizer of the energy $E_{\varepsilon}^\Delta$ using the EPG method.}
	\end{subfigure}
	\hfill
	\begin{subfigure}[t]{0.45\textwidth}
		\centering
		\includegraphics[width=.8\textwidth]{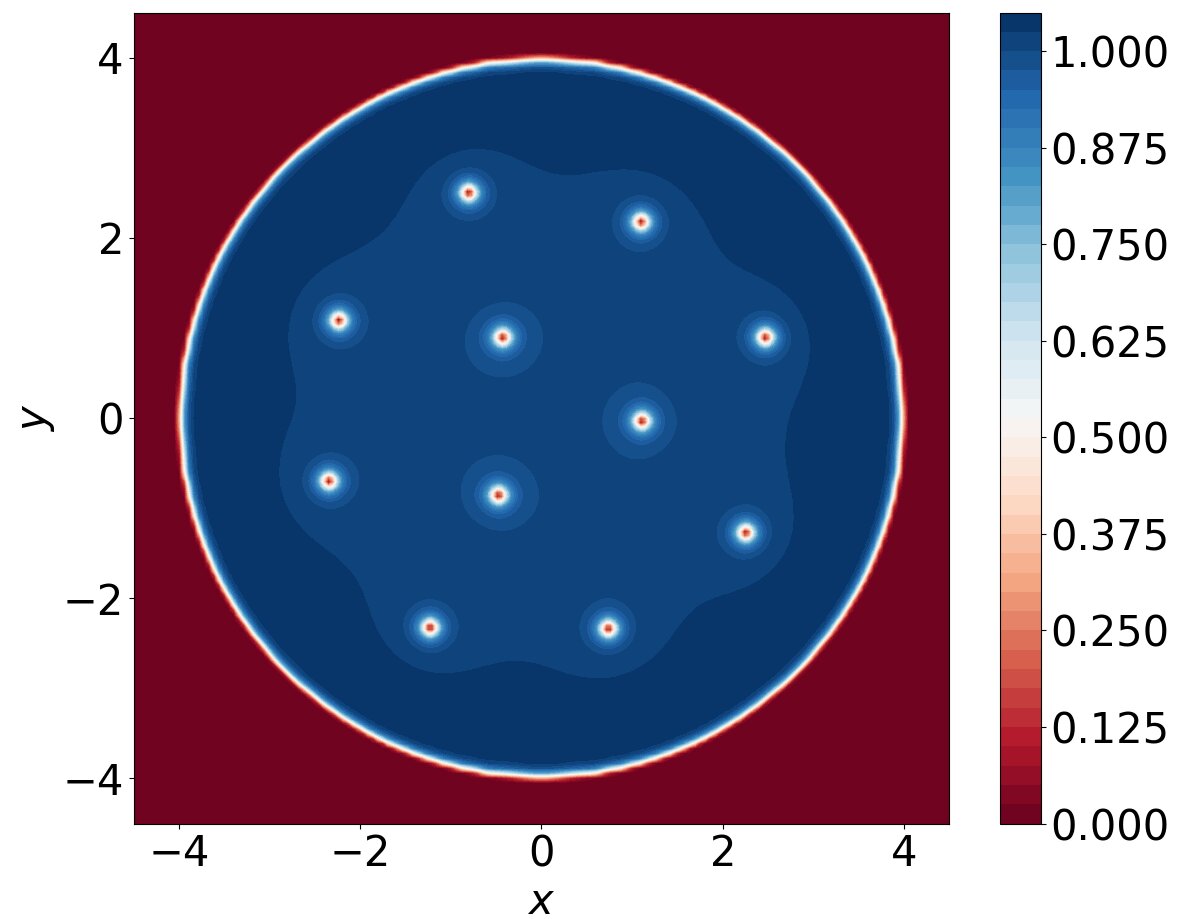}
		\subcaption{A minimizer of the energy $\mathcal{E}_{GPE}$ using the GPELab method.}
	\end{subfigure}
	\vskip\baselineskip
	\begin{center}	
		\begin{tabular}{|| c || c | c | c | c | c | r||}
			\cline{2-7}
			\multicolumn{1}{c|}{} & $E^\Delta_{\varepsilon,\delta}$ & $G_0$ & $K^\Delta$ & \# of iterations & $E_{\varepsilon,\delta}^\Delta(u_n)-E_{\varepsilon,\delta}^\Delta(u_{n-1})$ & Time (s) \\
			\cline{2-7} \hline
			GPELab & $38$ & $2\times10^{-5}$ & $0.579$ & $129569$ & $-1.4\times10^{-5}$ & $95\thinspace189$ \\
			\hline 
			EPG & $83$ & $2\times10^{-5}$ & $0.011$ & $271041$ & $-4.5\times10^{-5}$ & $24\thinspace688$\\ 
			\hline
		\end{tabular}
	\end{center}
	\caption{Comparison between GPELab and EPG results for $\Omega=1$.}
	\label{fig:comparison_1_256}
\end{figure}

\begin{figure}[ht]
	\centering
	\begin{subfigure}[t]{0.45\textwidth}
		\centering
		\includegraphics[width=.8\textwidth]{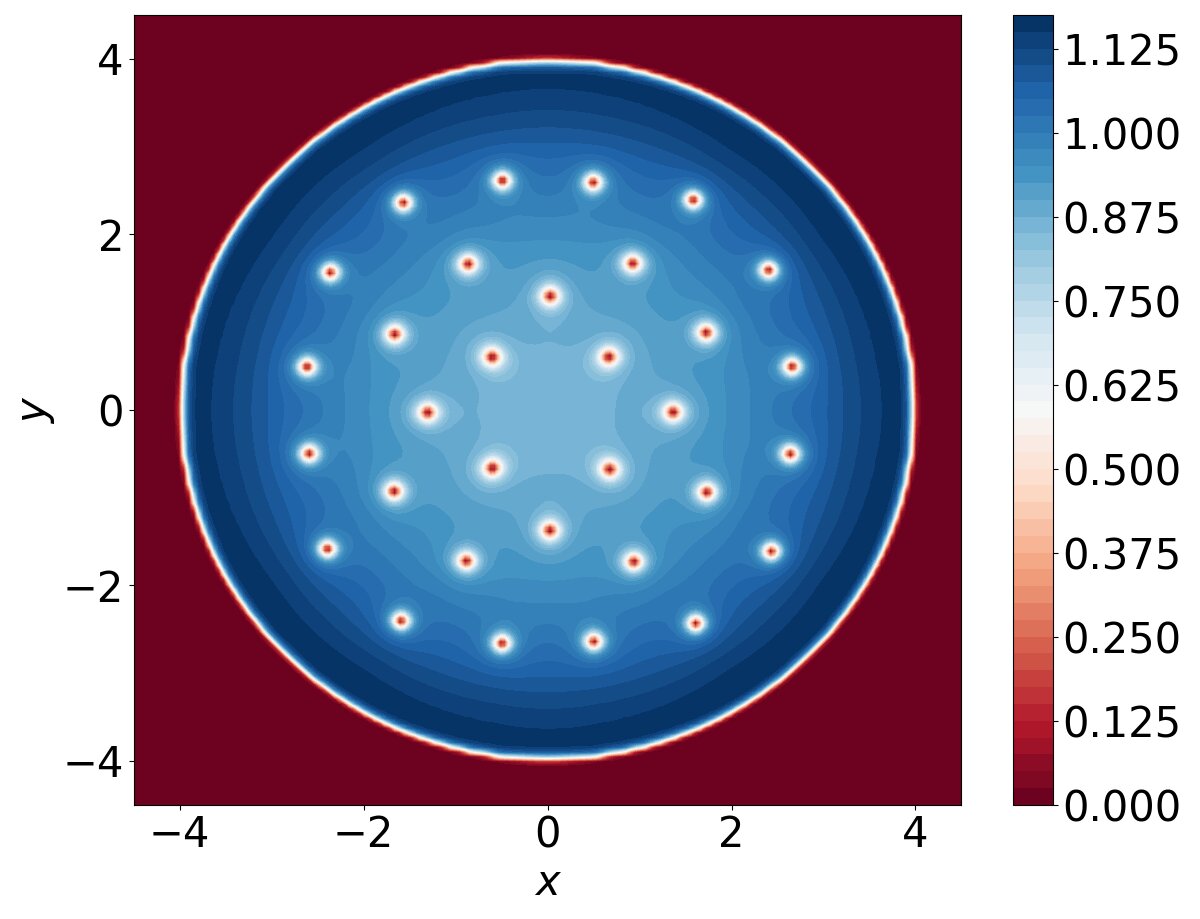}
		\subcaption{A minimizer of the energy $E_{\varepsilon}^\Delta$ using the EPG method.}
	\end{subfigure}
	\hfill
	\begin{subfigure}[t]{0.45\textwidth}
		\centering
		\includegraphics[width=.8\textwidth]{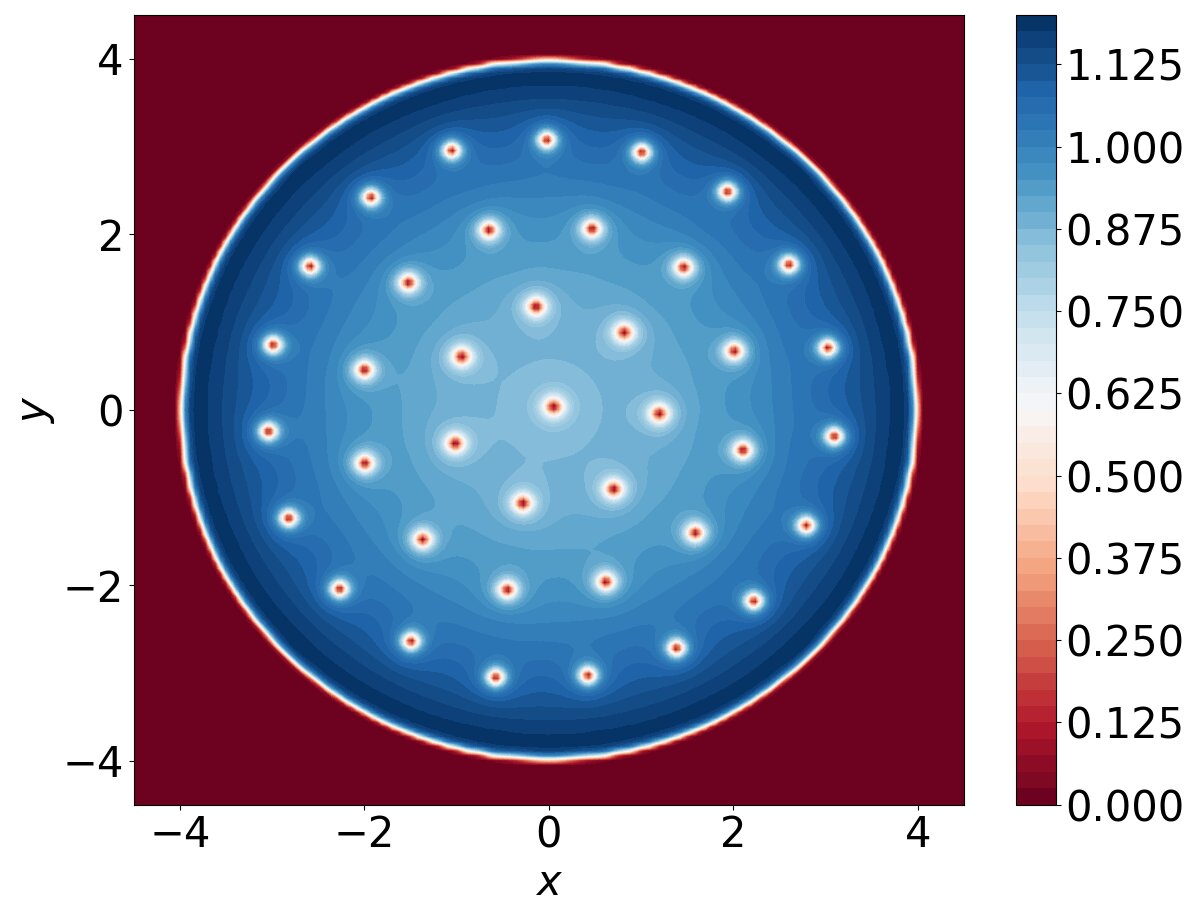}
		\subcaption{A minimizer of the energy $\mathcal{E}_{GPE}$ using the GPELab method.}
	\end{subfigure}
	\vskip\baselineskip	
	\begin{center}
		\begin{tabular}{|| c || c | c | c | c | c | r||}
			\cline{2-7}
			\multicolumn{1}{c|}{} & $E^\Delta_{\varepsilon,\delta}$ & $G_0$ & $K^\Delta$ & \# of iterations & $E_{\varepsilon,\delta}^\Delta(u_n)-E_{\varepsilon,\delta}^\Delta(u_{n-1})$ & Time(s) \\
			\cline{2-7} \hline
			GPELab & $-1419$ & $2\times 10^{-5}$ & $0.05$ & $56507$ & $-1.3\times10^{-5}$ & $60\thinspace209$ \\
			\hline 
			EPG & $-1428$ & $2.4\times10^{-5}$ & $0.01$ & $182346$ & $-3.5 \times 10^{-5}$ & $16\thinspace609$\\ 
			\hline
		\end{tabular}
	\end{center}	
	\caption{Comparison between GPELab and EPG results for $\Omega=3$.}
	\label{fig:comparison_3_256}
\end{figure}

\begin{figure}[ht]
	\centering
	\begin{subfigure}[t]{0.45\textwidth}
		\centering
		\includegraphics[width=.8\textwidth]{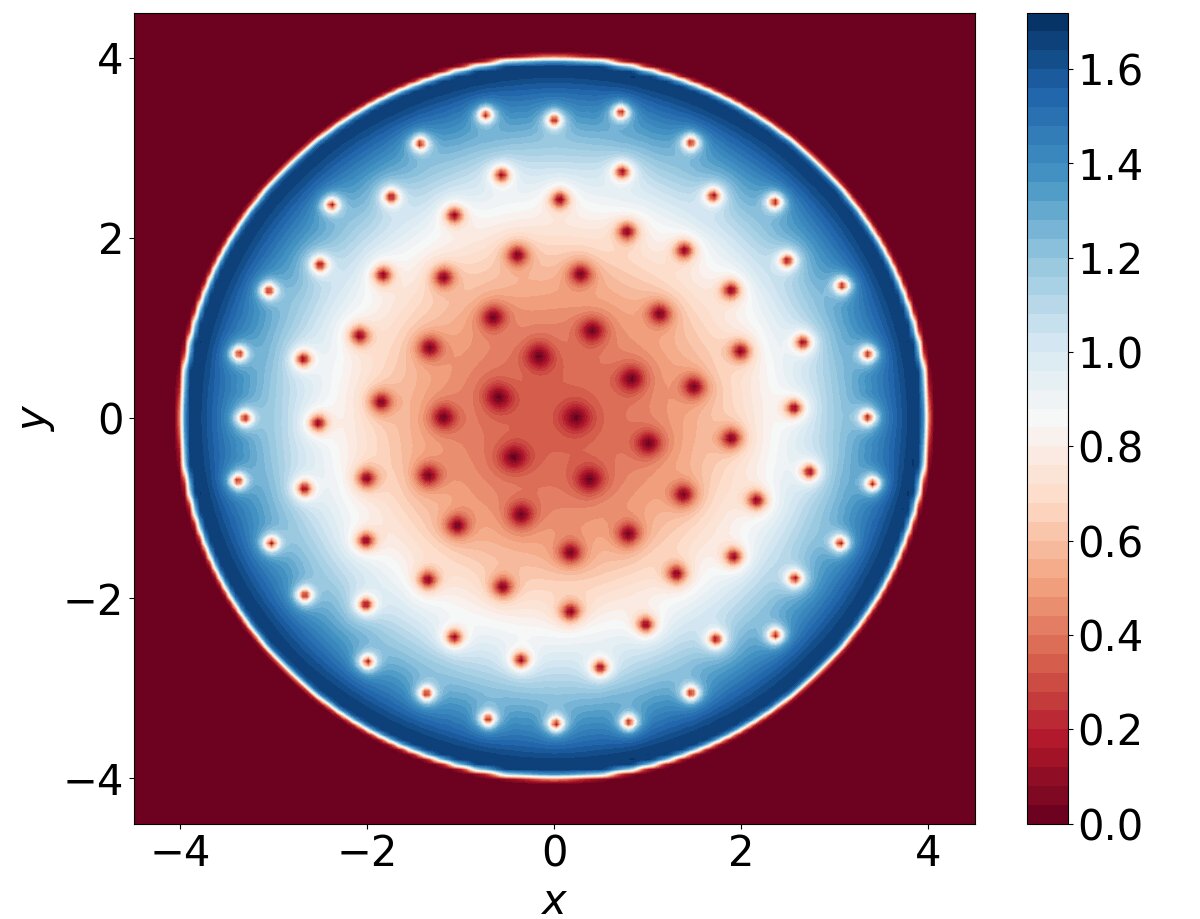}
		\subcaption{A minimizer of the energy $E_{\varepsilon}^\Delta$ using the EPG method.}
	\end{subfigure}
	\hfill
	\begin{subfigure}[t]{0.45\textwidth}
		\centering
		\includegraphics[width=.8\textwidth]{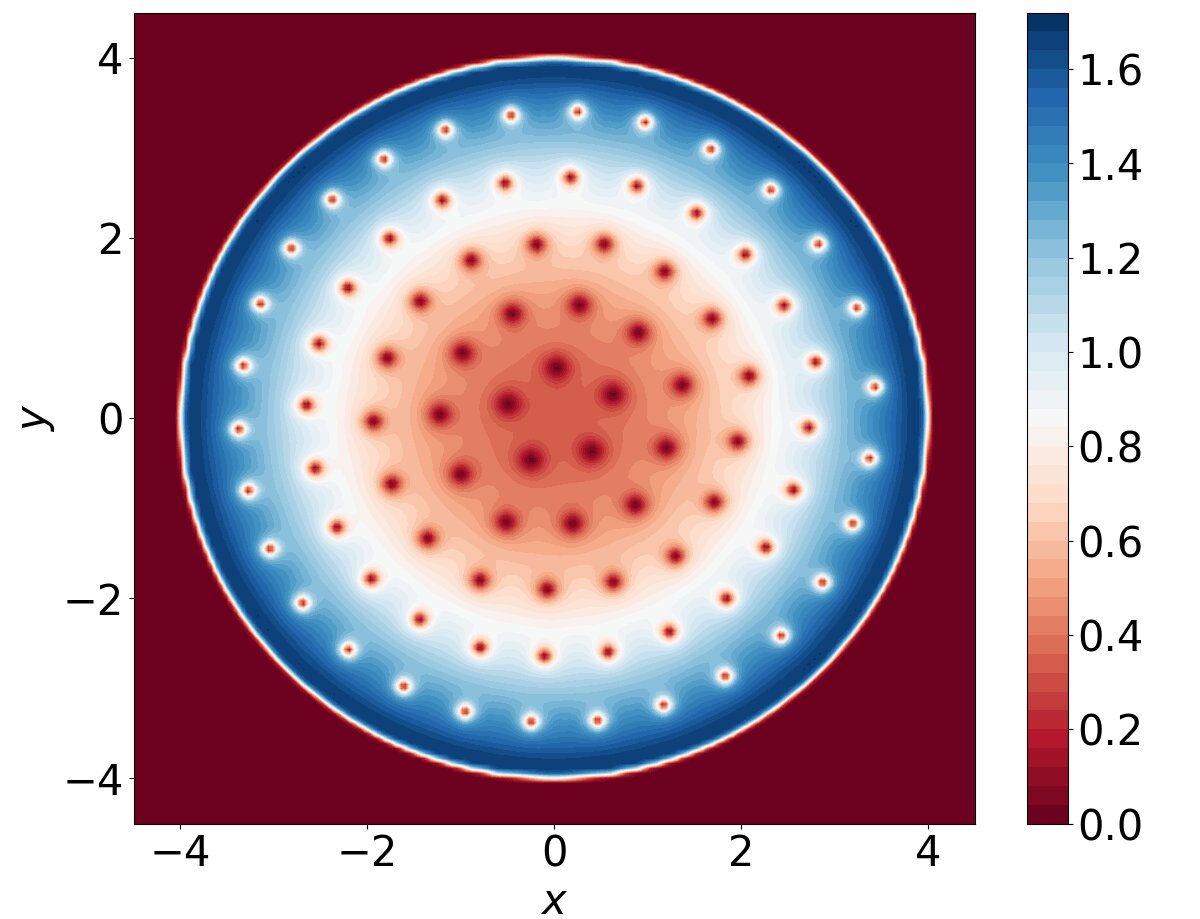}
		\subcaption{A minimizer of the energy $E_{\varepsilon}^\Delta$ using the GPELab method.}
	\end{subfigure}
	\vskip\baselineskip	
	\begin{center}
		\begin{tabular}{|| c || c | c | c | c | c | r||}
			\cline{2-7}
			\multicolumn{1}{c|}{} & $E^\Delta_{\varepsilon,\delta}$ & $G_0$ & $K^\Delta$ & \# of iterations & $E_{\varepsilon,\delta}^\Delta(u_n)-E_{\varepsilon,\delta}^\Delta(u_{n-1})$ & Time(s) \\
			\cline{2-7} \hline
			GPELab & $-7334$ & $10^{-5}$ & $0.0699$ & $224473$ & $-4.3\times10^{-4}$ & $199\thinspace321$ \\
			\hline 
			EPG & $-7338$ & $5\times10^{-6}$ & $0.00578$ & $707141$ & $-1.1 \times 10^{-5}$ & $64\thinspace942$\\ 
			\hline
		\end{tabular}
	\end{center}	
	\caption{Comparison between GPELab and EPG results for $\Omega=6$.}
	\label{fig:comparison_6_256}
      \end{figure}

        The numerical results displayed in Figures \ref{fig:comparison_1_256},
        \ref{fig:comparison_3_256} and \ref{fig:comparison_6_256} show that
        the methods may converge to different minimizers, with different energies
          (see Figure \ref{fig:comparison_1_256}),
          as well as they may converge to similar minimizers with similar energies
          (see Figure \ref{fig:comparison_6_256}),
          taking into account suitable rescaling of the discrete energies
          (as indicated in Remark \ref{rem:GPELab}).
          Moreover, the qualitative behaviour of both methods is correct
          in the sense that higher rotations end up in creating more vortices,
          and the beginning of a central hole when $\Omega=6$
          (see Figure \ref{fig:comparison_6_256}).
          Note that the EPG method uses more interations in all cases.
          However, since the EPG method is explicit and the GPELab method is linearly implicit,
          the EPG method is faster in all cases.
          The order of magnitude of the speed up is $3$.
          Finally, the stopping criterion $K^\Delta$ is much smaller in minimizers obtained by EPG
          than in that computed using GPELab.
      
\subsection{Two components comparison}
In this section, we consider two components Bose--Einstein condensates.
We also compare both algorithms on the same computer.
Each test is done separately.

\subsubsection{Parameters for the EPG method}
We take the following parameters for all the tests.
We take $\varepsilon=5\times 10^{-2}$.
For the initial datum, we choose $\psi^1=\psi^2=\frac{1}{5} \exp({-10x^2-10y^2})$.
For the discretization parameters we take the values $N_1=0.55$ and $N_2=0.45$,
$N=256$ and $K_0=10^{-2}$.
For the rest of the parameters, we refer to \ref{subsec:com_par_gpe}.

\subsubsection{Parameters for GPELab's method}
The equivalent of these parameters in GPELab are
\begin{itemize}
	\item Ncomponents$=2$,
	\item Type='BESP',
	\item Delta$=0.5$ (the coefficient in front of the kinetic energy),
	\item Beta$=200$,
	\item Beta\textunderscore coupled$\displaystyle=\begin{pmatrix}
		1 & \delta \\
		\delta & 1 \\
              \end{pmatrix},$
	\item $V(x,y)=-\frac{1}{2\varepsilon^2}\min[1,10 (R^2-x^2-y^2)]$ (the same function created before),
	\item Initial datum: $\psi^1=\psi^2=\frac{1}{5} \exp({-10x^2-10y^2})$,
	\item $\Delta_t= 5\times10^{-4}$ and Stop\_crit$ =  2\times10^{-2}$ ($G_0=10^{-5}$),
	\item $x_{min}=y_{min}=-7$ and $x_{max}=y_{max}=7$.
\end{itemize}
We also had to modify the normalization step after each time step $\Delta_t$.
\subsubsection{Numerical results}
For a first comparison, we set the rotation speed to $\Omega=3$
and the interaction strength to $\delta=0.7$ (coexistence). 
We display the results in Figure \ref{fig:comparison2_3_256}.
For a second comparison, we set the rotation speed to $\Omega=3$
and the interaction strength to $\delta=1.5$ (segregation).
For the EPG method, to break off the symmetry, we chose the following initial datum
$$ \psi^1=\frac15 \exp(-10(x-0.5)^2-10(y+0.2)^2), \qquad \psi^2=\frac15 \exp(-10(x+0.5)^2-10(y-0.3)^2).$$
We display the results in Figure \ref{fig:comparison2_3_256_seg}.

\begin{figure}[ht]
	\centering
	\begin{subfigure}[t]{0.45\textwidth}
		\centering
		\includegraphics[width=.8\textwidth]{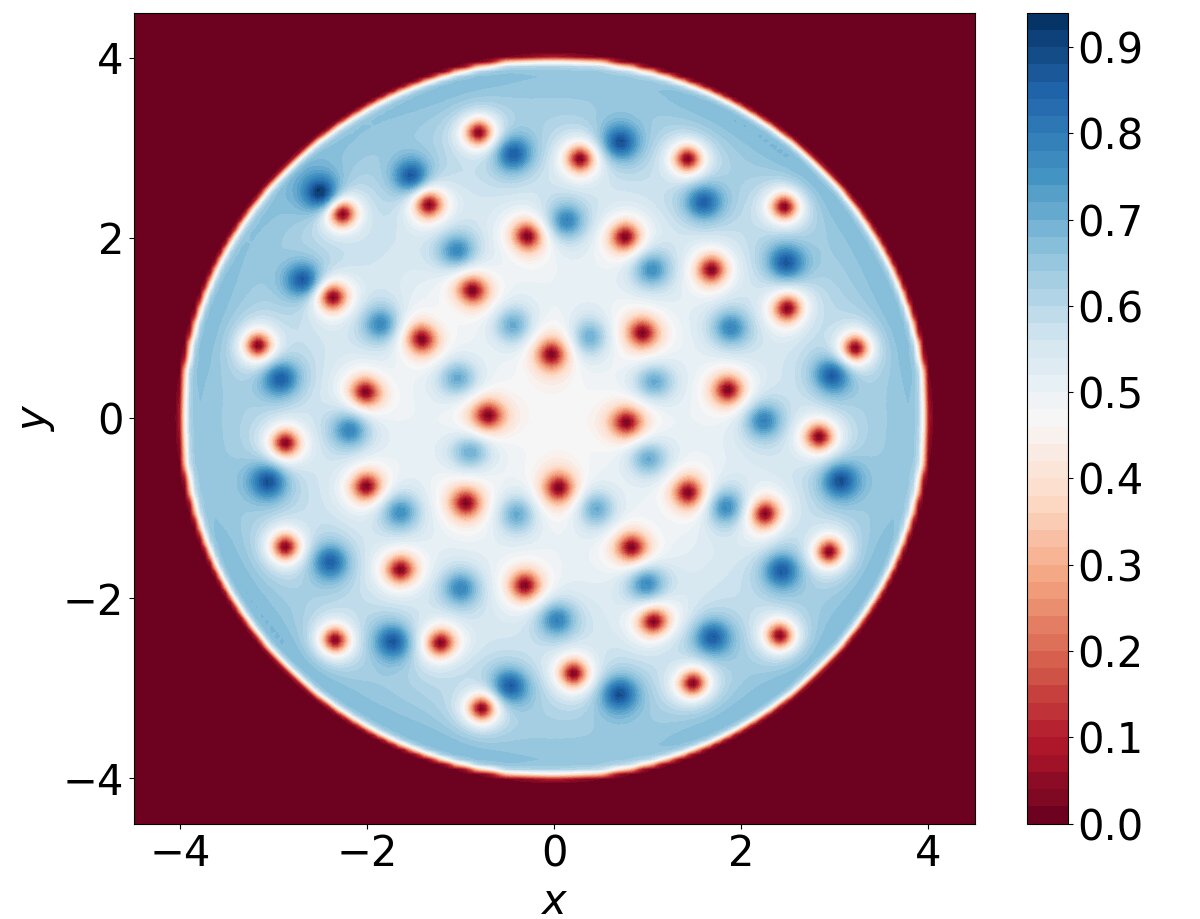}
		\subcaption{First component of a minimizer of the energy $E_{\varepsilon}^\Delta$ using the EPG method.}
	\end{subfigure}
	\hfill
	\begin{subfigure}[t]{0.45\textwidth}
		\centering
		\includegraphics[width=.8\textwidth]{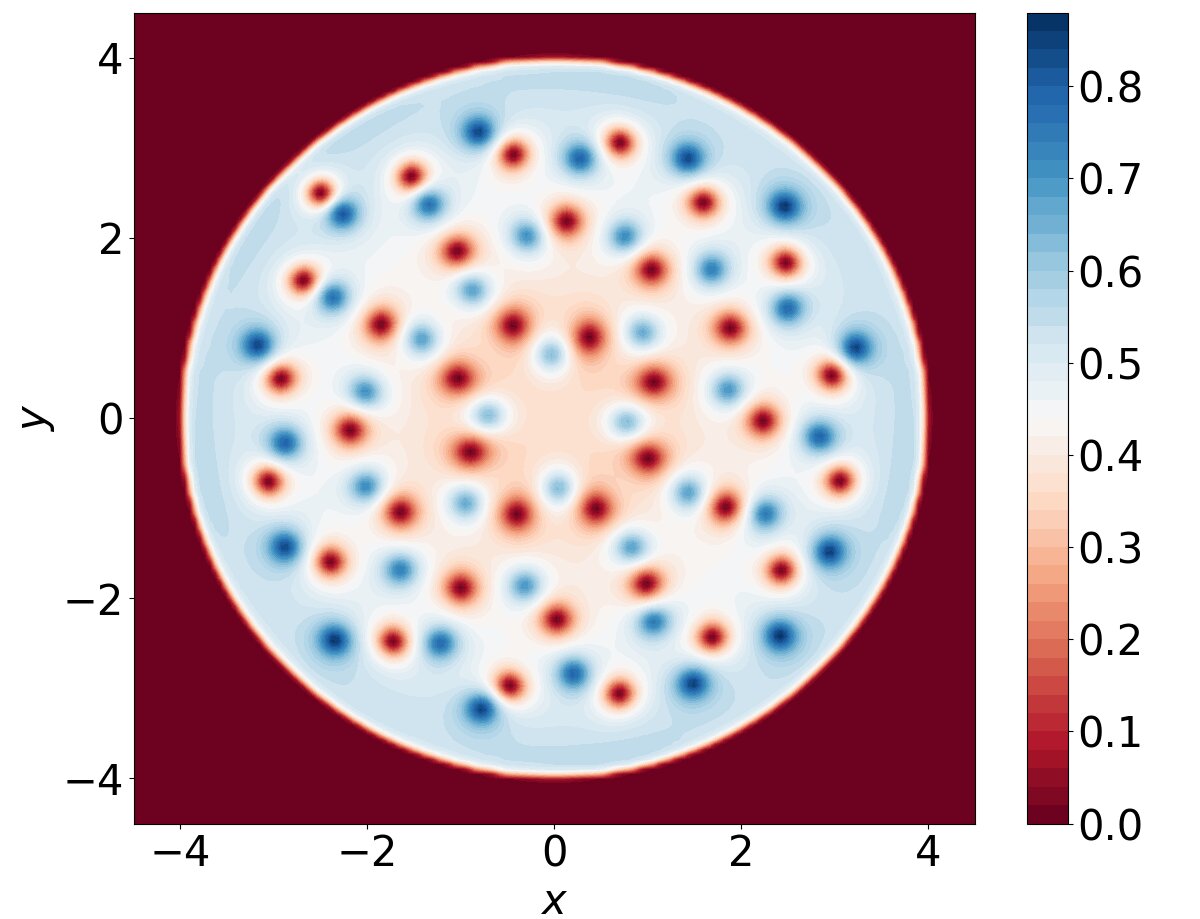}
		\subcaption{Second component of a minimizer of the energy $E_{\varepsilon}^\Delta$ using the EPG method.}
	\end{subfigure}
	\vskip\baselineskip	
	\centering
	\begin{subfigure}[t]{0.45\textwidth}
		\centering
		\includegraphics[width=.8\textwidth]{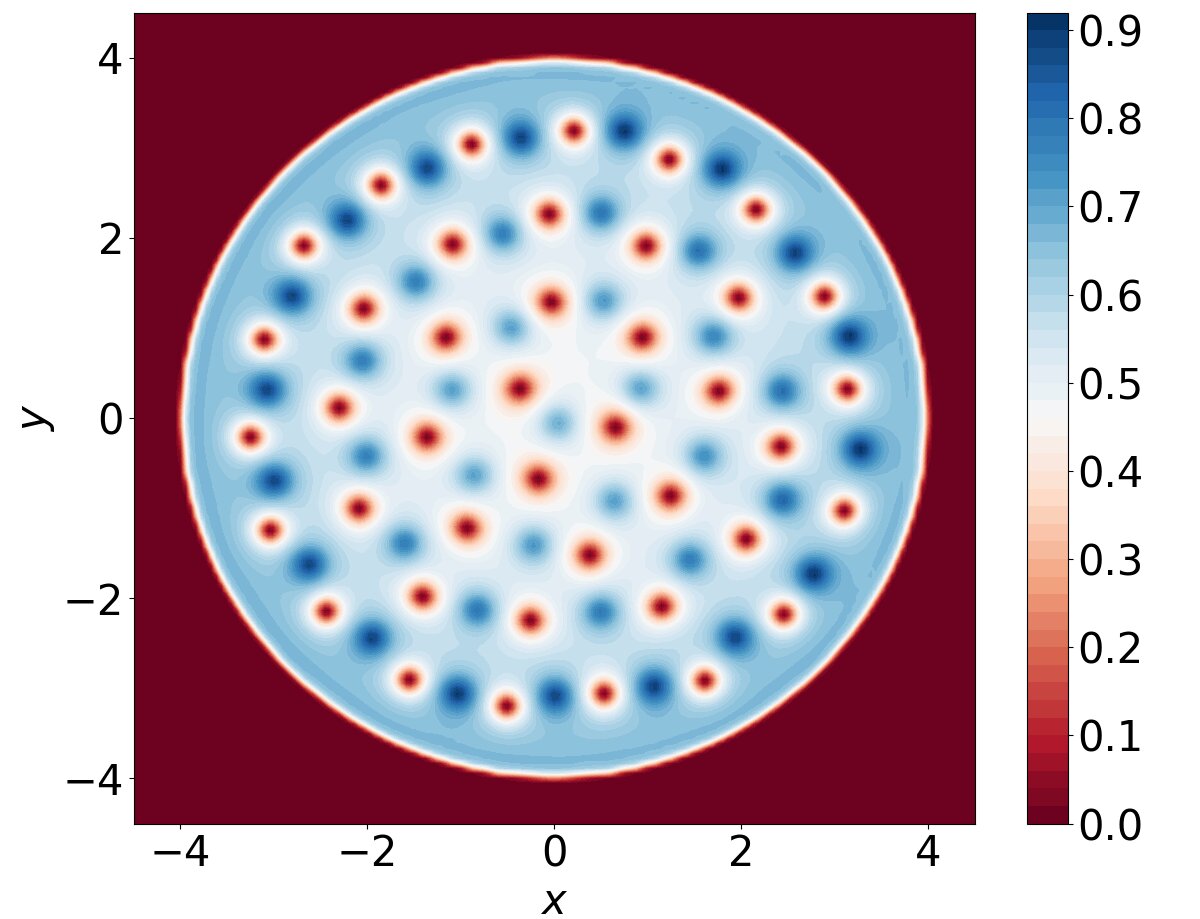}
		\subcaption{First component of a minimizer of the energy $E_{\varepsilon}^\Delta$ using the GPELab method.}
	\end{subfigure}
	\hfill
	\begin{subfigure}[t]{0.45\textwidth}
		\centering
		\includegraphics[width=.8\textwidth]{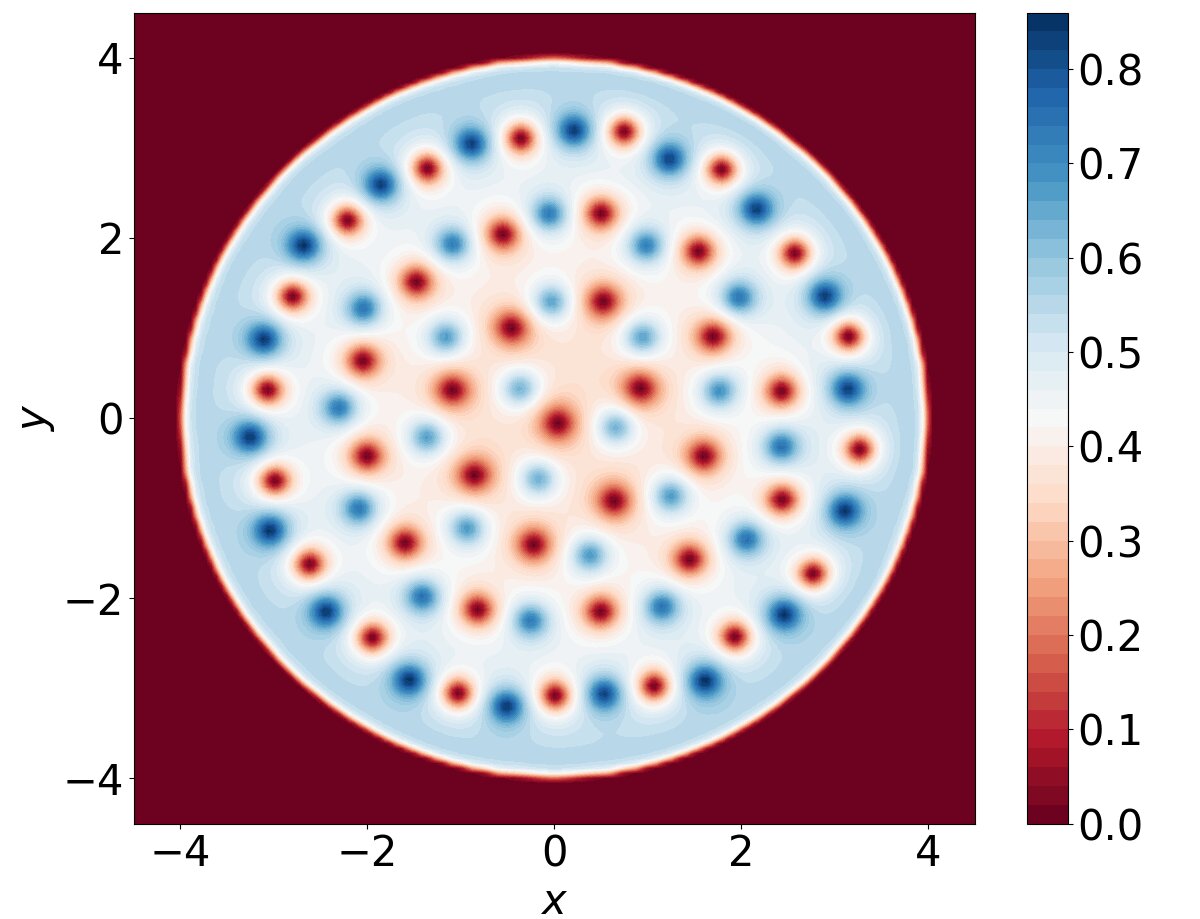}
		\subcaption{Second component of a minimizer of the energy $E_{\varepsilon}^\Delta$ using the GPELab method.}
	\end{subfigure}
	\vskip\baselineskip
	\begin{center}
		\begin{tabular}
			{|| c || c | c | c | c | c | r||}
			\cline{2-7}
			\multicolumn{1}{c|}{} & $E^\Delta_{\varepsilon,\delta}$ & $G_0$ & $K^\Delta$ & \# of iterations & $E_{\varepsilon,\delta}^\Delta(u_n)-E_{\varepsilon,\delta}^\Delta(u_{n-1})$ & Time(s) \\
			\cline{2-7} \hline
			GPELab & $-2234$ & $10^{-5}$ & $0.07$ & $206833$ & $-2.3\times10^{-6}$ & $163\thinspace690$ \\
			\hline 
			EGP & $-2227$ & $10^{-5}$ & $0.01$ & $259383$ & $-2 \times 10^{-5}$ & $50\thinspace527$\\ 
			\hline
		\end{tabular}
	\end{center}	
	\caption{Comparison between GPELab and EPG results for $\Omega=3$ and $\delta=0.7$.}
	\label{fig:comparison2_3_256}
\end{figure}

\begin{figure}[ht]
	\centering
	\begin{subfigure}[t]{0.45\textwidth}
		\centering
		\includegraphics[width=.8\textwidth]{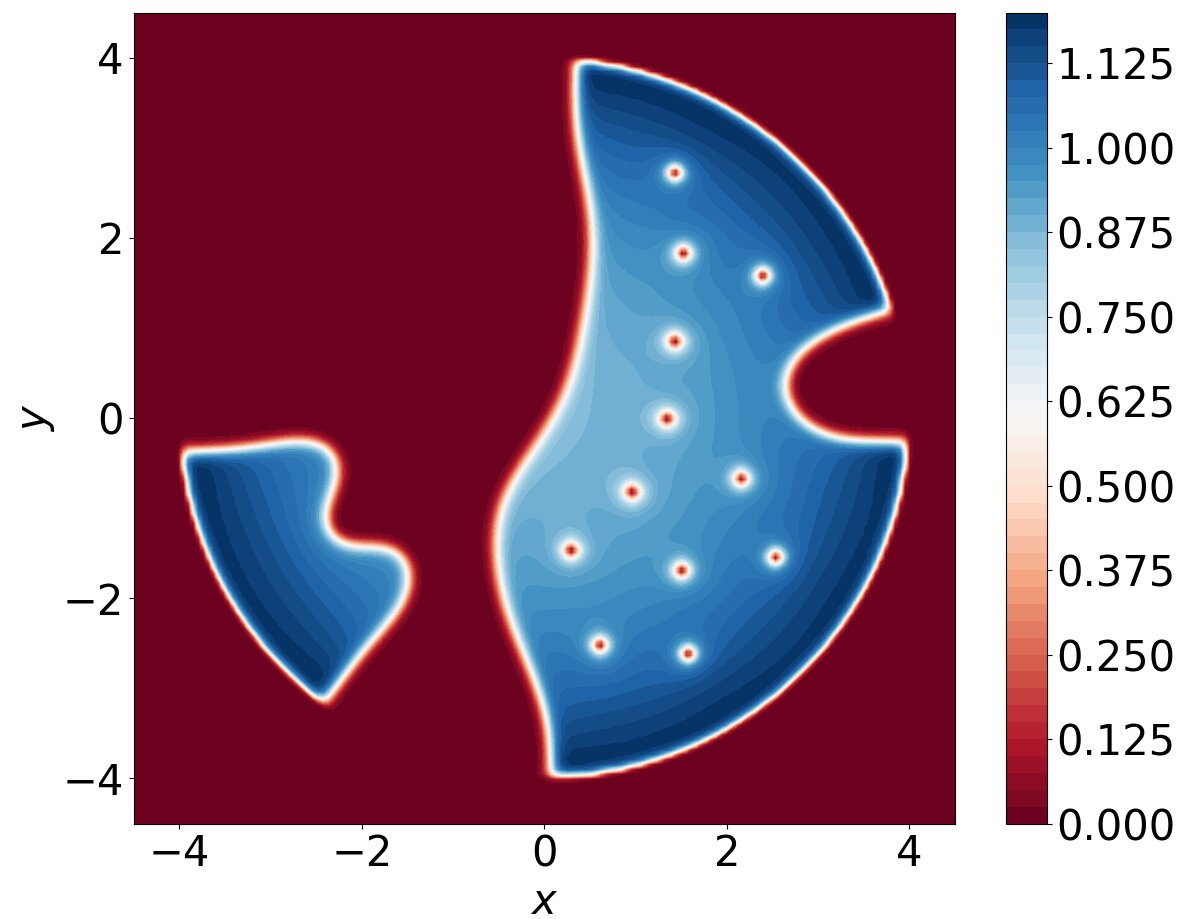}
		\subcaption{First component of a minimizer of the energy $E_{\varepsilon}^\Delta$ using the EPG method.}
	\end{subfigure}
	\hfill
	\begin{subfigure}[t]{0.45\textwidth}
		\centering
		\includegraphics[width=.8\textwidth]{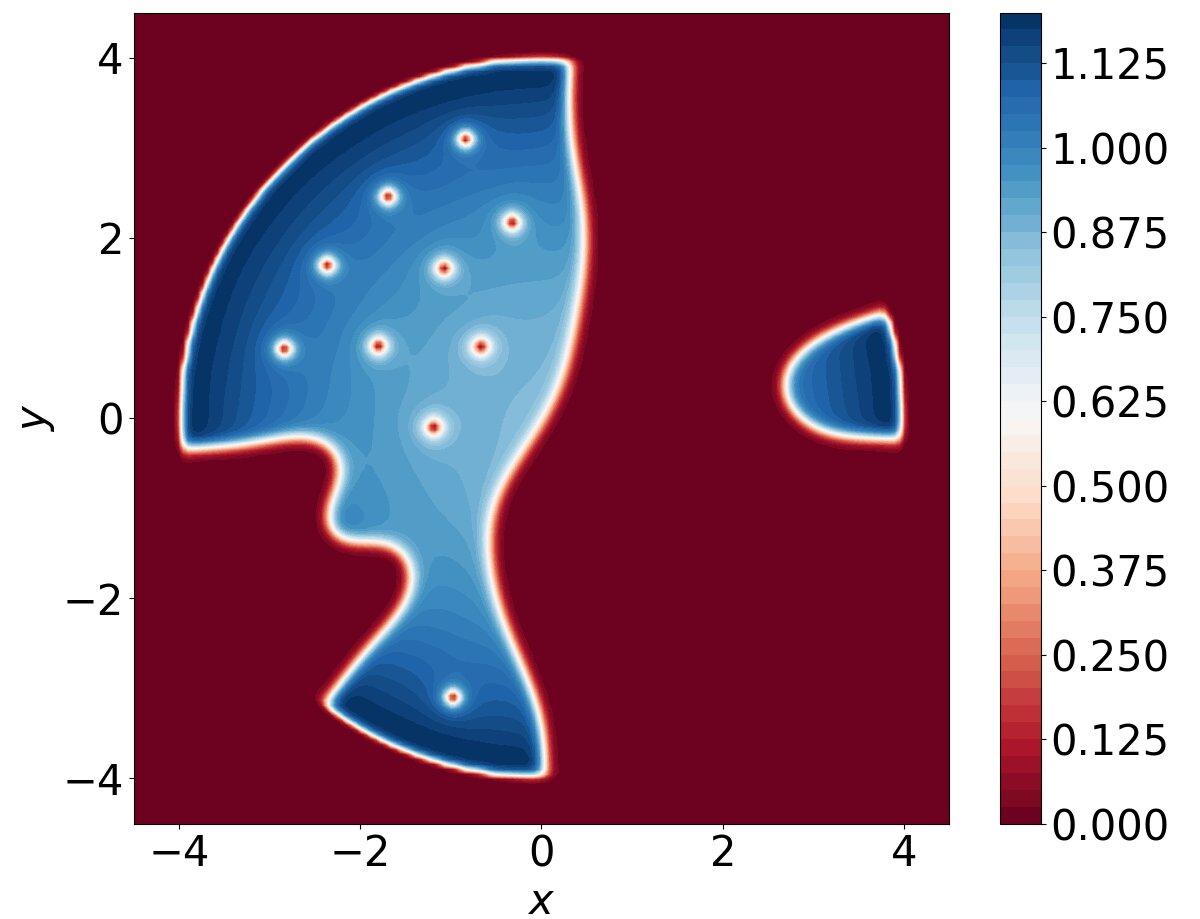}
		\subcaption{Second component of a minimizer of the energy $E_{\varepsilon}^\Delta$ using the EPG method.}
	\end{subfigure}
	\vskip\baselineskip	
	\centering
	\begin{subfigure}[t]{0.45\textwidth}
		\centering
		\includegraphics[width=.8\textwidth]{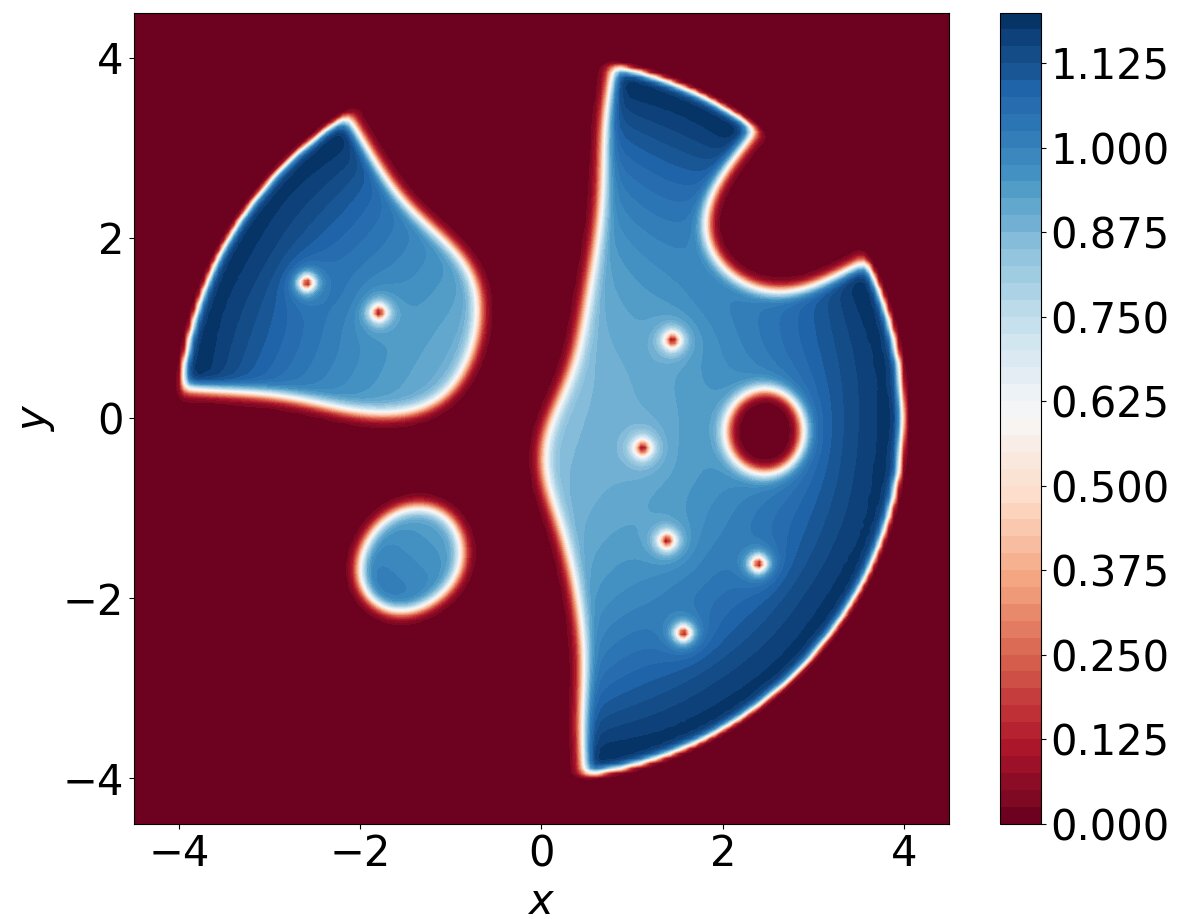}
		\subcaption{First component of a minimizer of the energy $E_{\varepsilon}^\Delta$ using the GPELab method.}
	\end{subfigure}
	\hfill
	\begin{subfigure}[t]{0.45\textwidth}
		\centering
		\includegraphics[width=.8\textwidth]{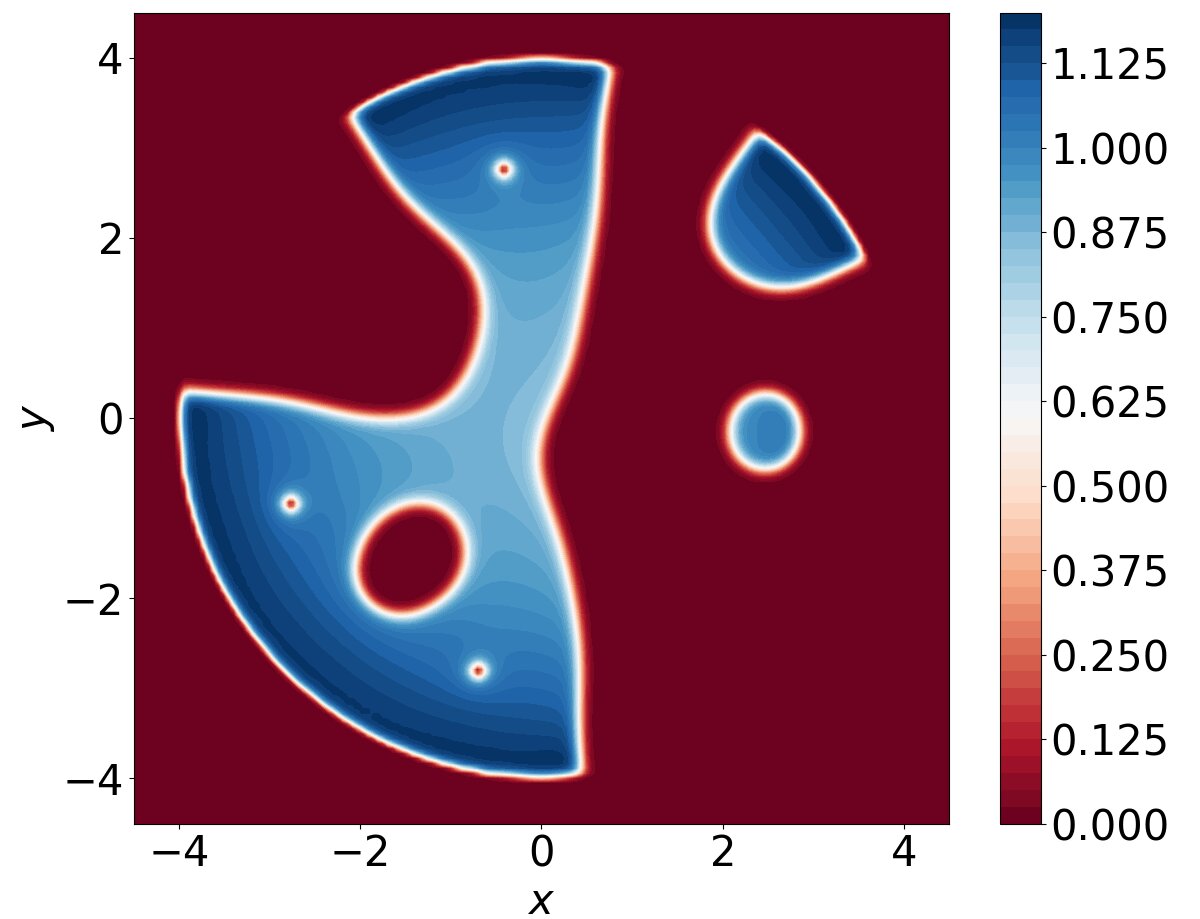}
		\subcaption{Second component of a minimizer of the energy $E_{\varepsilon}^\Delta$ using the GPELab method.}
	\end{subfigure}
	\vskip\baselineskip
	\begin{center}
		\begin{tabular}
			{|| c || c | c | c | c | c | r||}
			\cline{2-7}
			\multicolumn{1}{c|}{} & $E^\Delta_{\varepsilon,\delta}$ & $G_0$ & $K^\Delta$ & \# of iterations & $E_{\varepsilon,\delta}^\Delta(u_n)-E_{\varepsilon,\delta}^\Delta(u_{n-1})$ & Time(s) \\
			\cline{2-7} \hline
			GPELab & $-1396$ & $10^{-5}$ & $0.07$ & $508797$ & $-10^{-6}$ & $1\thinspace526\thinspace529$ \\
			\hline 
			EGP & $-1389$ & $10^{-5}$ & $0.01$ & $321882$ & $-1.7 \times 10^{-5}$ & $62\thinspace745$\\ 
			\hline
		\end{tabular}
	\end{center}	
	\caption{Comparison between GPELab and EPG results for $\Omega=3$ and $\delta=1.5$.}
	\label{fig:comparison2_3_256_seg}
\end{figure}

The results displayed in Figures \ref{fig:comparison2_3_256} and \ref{fig:comparison2_3_256_seg} both
indicate that the methods may converge to different minimizers, even if their energies are
similar (see Remark \ref{rem:GPELab} for details on the rescaling).
The qualitative behaviour of both methods is correct in the sense that they provide coexisting
minimizers when $\delta=0.7$  and segregated minimizers when $\delta=1.5$.
The number of iterations in both methods has the same order of magnitude
(the maximal ratio is of order $2$), in contrast to the one component simulations carried out
in Section \ref{subsec:onecompcomparison}.
The EPG method is much faster in both cases: The order of magnitude of the speed up is $3$
when $\delta=0.7$ and $25$ when $\delta=1.5$.
Finally, the stopping criterion $K_\Delta$ is much smaller in minimizers obtained by EPG
than in that computed using GPELab.

\section{Conclusion}

In this paper, we considered a Gross--Pitaevskii energy as a model for rotating
one component and two components Bose--Einstein condensates in two dimensions
in a strong confinement and strong rotation regime. 
First, we have introduced a discretization of this energy, using the FFT scheme
and Plancherel's equality, which allowed a reasonable computation of the energy gradients.
Second, in contrast to the literature, we proposed a minimization method for this discrete energy
using an explicit $L^2$ gradient method with projection (EPG).
This method allowed for the derivation of a stopping criterion. 
Third, we introduced two post processing algorithms for the numerical minimizers.
One is aimed for the single vortices and the other for vortex sheets.
Both allow to detect these structures and compute their indices.
Fourth, we have ran our methods and algorithms 
for different physical regimes from one component condensates with high rotation to two components
condensates in coexistence and segregation regimes.
This validates recent theoretical results as well as it supports conjectures
(as for example the existence of vortex sheets in the segregation regime). 
Last, we have compared the efficiency of the EPG method to that of the GPELab method \cite{GPE2}.
On all the numerical tests presented, EPG appears to be, roughly speaking, at least
$3$ times faster than GPELab.
\section*{Acknowledgments}
This work was partially supported by the Labex CEMPI (ANR-11-LABX-0007-01).

\bibliographystyle{abbrv}
\bibliography{labib}

\end{document}

%% file: test_convergence_N_128.tex
\begin{tikzpicture}

\definecolor{color0}{rgb}{0.16078431372549,0.501960784313725,0.725490196078431}
\definecolor{color1}{rgb}{0.133333333333333,0.6,0.329411764705882}
\definecolor{color2}{rgb}{0.905882352941176,0.298039215686275,0.235294117647059}
\definecolor{color3}{rgb}{0.0901960784313725,0.125490196078431,0.164705882352941}
\definecolor{color4}{rgb}{0.490196078431373,0.235294117647059,0.596078431372549}
\definecolor{color5}{rgb}{0.952941176470588,0.611764705882353,0.0705882352941176}
\definecolor{color6}{rgb}{0.650980392156863,0.674509803921569,0.686274509803922}

\begin{axis}[
legend cell align={left},
legend style={
  fill opacity=0.8,
  draw opacity=1,
  text opacity=1,
  at={(0.99,0.01)},
  anchor=south east,
  draw=white!80!black,
  legend columns=2
},
tick align=outside,
tick pos=left,
x grid style={white!69.0196078431373!black},
xlabel={\(\displaystyle \log_{10}(\varepsilon)\)},
xmin=-2.1, xmax=0.1,
xtick style={color=black},
y grid style={white!69.0196078431373!black},
ylabel={\(\displaystyle \log_{10}\left(\| |\psi^{\ast}_\varepsilon|^2 - (\rho)_+ \|_\Delta\right)\)},
ymin=-1.56351521096202, ymax=1.67814430578154,
ytick style={color=black}
]
\addplot [semithick, color0, mark=asterisk, mark size=3, mark options={solid,fill opacity=0}, only marks]
table {%
0 1.53079614592956
-0.25 1.34711169256716
-0.5 1.13038120003539
-0.75 0.835355373677379
-1 0.339600171580581
-1.25 -0.00633861560336469
-1.5 -0.387003637161629
-1.75 -0.794932282467178
-2 -1.27568966035768
};
\addlegendentry{\tiny{$K^\Delta=10^0$}}
\addplot [semithick, color1, mark=o, mark size=3, mark options={solid,fill opacity=0}, only marks]
table {%
0 1.26083191898465
-0.25 1.04439728591878
-0.5 0.696077253875229
-0.75 0.325804609771731
-1 0.109972709606536
-1.25 -0.135909638957123
-1.5 -0.540306912341637
-1.75 -0.794281971936477
-2 -1.32161651578224
};
\addlegendentry{\tiny{$K^\Delta=10^{-0.5}$}}
\addplot [semithick, color2, mark=diamond, mark size=3, mark options={solid,fill opacity=0}, only marks]
table {%
0 0.928880628302727
-0.25 0.655608925078026
-0.5 0.408387662937559
-0.75 0.195498412451011
-1 0.062390095677471
-1.25 -0.209867135049738
-1.5 -0.533453070422638
-1.75 -0.842879257678568
-2 -1.37890801803138
};
\addlegendentry{\tiny{$K^\Delta=10^{-1.0}$}}
\addplot [semithick, color3, mark=+, mark size=3, mark options={solid}, only marks]
table {%
0 0.678680650299638
-0.25 0.501483494817985
-0.5 0.32623275095132
-0.75 0.15439391630627
-1 -0.0391611025850008
-1.25 -0.252612930389231
-1.5 -0.57073012121486
-1.75 -0.936114966848831
-2 -1.41616705111004
};
\addlegendentry{\tiny{$K^\Delta=10^{-1.5}$}}
\addplot [semithick, color4, mark=triangle, mark size=3, mark options={solid,fill opacity=0}, only marks]
table {%
0 0.591004509852432
-0.25 0.450779736161526
-0.5 0.300131164744156
-0.75 0.141428927549172
-1 -0.0392219344117746
-1.25 -0.277600048281327
-1.5 -0.571844242434146
-1.75 -0.95776926286345
-2 -1.41598828487146
};
\addlegendentry{\tiny{$K^\Delta=10^{-2.0}$}}
\addplot [semithick, color5, mark=square, mark size=3, mark options={solid,fill opacity=0}, only marks]
table {%
0 0.561489927288214
-0.25 0.434544540111474
-0.5 0.291913537228913
-0.75 0.137340292585322
-1 -0.039214619973243
-1.25 -0.277607193918031
-1.75 -0.958266968140438
-2 -1.41586227421855
};
\addlegendentry{\tiny{$K^\Delta=10^{-2.5}$}}
\addplot [semithick, color6, mark=x, mark size=3, mark options={solid,fill opacity=0}, only marks]
table {%
0 0.552000412416976
-0.25 0.429422500187169
-0.5 0.28932739144745
-0.75 0.136044141028862
-1 -0.0392041494276569
-1.25 -0.277610404911542
-2 -1.41581101254869
};
\addlegendentry{\tiny{$K^\Delta=10^{-3.0}$}}
\end{axis}

\end{tikzpicture}